\documentclass[11pt]{article}
\usepackage{amssymb}
\usepackage{amsmath, amsthm, color}
\usepackage[colorlinks]{hyperref}
\usepackage{fullpage}

\usepackage{caption}
\usepackage{subcaption}

\usepackage{makecell}

\usepackage{longtable}

\usepackage{bbm}

\usepackage{graphicx}

\usepackage{chngcntr}

\usepackage{mathtools}

\theoremstyle{definition}
\newtheorem*{notation*}{Notation}
\newtheorem{theorem}{Theorem}

\newtheorem{corollary}{Corollary}
\newtheorem{lemma}{Lemma}
\newtheorem{proposition}{Proposition}

\newtheorem{definition}{Definition}
\newtheorem{assumption}{Assumption}
\newtheorem{condition}{Condition}
\theoremstyle{definition}
\newtheorem{remark}{Remark}

\newlength{\widebarargwidth}
\newlength{\widebarargheight}
\newlength{\widebarargdepth}

\newcommand{\RNum}[1]{\uppercase\expandafter{\romannumeral #1\relax}}

\newcommand{\real}{\ensuremath{\mathbb{R}}}

\newcommand{\Ball}{\ensuremath{\mathbb{B}}}

\newcommand{\opnorm}[1]{\left|\!\left|\!\left|{#1}\right|\!\right|\!\right|}

\begin{document}

\title{{\bf Adaptive Estimation and Statistical Inference for High-Dimensional Graph-Based Linear Models}}

\makeatletter
\renewcommand\@date{{%
  \vspace{-\baselineskip}%
  \large\centering
  \begin{tabular}{@{}c@{}}
    Duzhe Wang\textsuperscript{$\ast$} \\
    \normalsize \texttt{dwang282@wisc.edu}
  \end{tabular}%
\quad \quad
  \begin{tabular}{@{}c@{}}
    Po-Ling Loh\textsuperscript{$\ast\ddagger$} \\
    \normalsize  \texttt{ploh@stat.wisc.edu}
     \end{tabular}

  \bigskip

  \textsuperscript{$\ast$}Department of Statistics, University of Wisconsin-Madison\par
  \textsuperscript{$\ddagger$}Department of Statistics, Columbia University
   \bigskip

  \today
}}
\makeatother

\maketitle

\begin{abstract}
We consider adaptive estimation and statistical inference for high-dimensional graph-based linear models. In our model, the coordinates of regression coefficients correspond to an underlying undirected graph. Furthermore, the given graph governs the piecewise polynomial structure of the regression vector. In the adaptive estimation part, we apply graph-based regularization techniques and propose a family of locally adaptive estimators called the Graph-Piecewise-Polynomial-Lasso. We further study a one-step update of the Graph-Piecewise-Polynomial-Lasso for the problem of statistical inference. We develop the corresponding theory, which includes the fixed design and the sub-Gaussian random design. Finally, we illustrate the superior performance of our approaches by extensive simulation studies and conclude with an application to an \emph{Arabidopsis thaliana} microarray dataset. 
\end{abstract}


\section{Introduction}
\label{intro}

Consider the high-dimensional linear model  
\begin{equation}
\label{linear}
    y=X\beta^*+\varepsilon, 
\end{equation}
where $X=(X_{1},...,X_{N})^T\in \real^{N\times n}$ is the design matrix with $X_{i}\in \mathbb{R}^{n}$ and $N\ll n$, $y=(y_{1},...,y_{N})^T\in \real^{N}$ is the response vector, $\beta^*\in \mathbb{R}^n$ is the unknown true regression parameter, and $\varepsilon\in \mathbb{R}^{N}$ is the additive noise. The linear model in $(\ref{linear})$ has been widely used for analyzing modern datasets collected from diverse scientific applications, such as climate science, medical imaging and biology \cite{chatterjee2012, kandel2013, buhlman2014}. One major area of research for high-dimensional linear models is developing novel methods to estimate the regression coefficients $\beta^*$ with certain desired structure. To estimate the sparse regression coefficients, Tibshirani \cite{tibshirani1996} proposed the celebrated Lasso which is defined as the minimizer of the following convex program
\begin{equation}
\label{lasso}
\begin{aligned}
& \underset{\beta\in \mathbb{R}^{n}}{\text{minimize}}
& & \frac{1}{2N}\|y-X\beta\|_{2}^2+\lambda\| \beta \|_{1},
\end{aligned}
\end{equation}
where $\lambda>0$ is a tuning parameter. Over the last two decades, there is a very substantial literature studying theory of the Lasso from various statistical perspectives. For example, Bickel, Ritov and Tsybakov \cite{bickel2009} introduced the restricted eigenvalue assumptions, and derived the non-asymptotic upper bounds of estimation error and prediction error of the Lasso in the linear model. Zhao and Yu \cite{zhao2006model} and Wainwright \cite{wainwright2009sharp} studied the model selection consistency and variable selection consistency of the Lasso, respectively. Furthermore, van de Geer et al.\ \cite{vandegeer2014}, Zhang and Zhang \cite{zhang2014confidence}, and Javanmard and Montanari \cite{javanmard14a} established de-biased methods to construct confidence intervals and perform statistical tests for low-dimensional components of the regression coefficients. A detailed overview on the Lasso can be found in \cite{buhlmann2011statistics} and \cite{wainwright2019}. 

Recent developments in this area also take into account other prior structures besides the sparsity. For  example, in the gene expression data measured from a microarray, the regression vector $\beta^*$ in $(\ref{linear})$ corresponds to a list of ordered genes, where correlated genes are placed consecutively. It is a common point of view that $\beta^*$ is both sparse and locally constant \cite{tibshirani2004}. Tibshirani et al.\ \cite{tibshirani2004} proposed the fused Lasso to estimate $\beta^*$ in the setting described above:
\begin{equation}
\label{fusedlassoprog}
\hat{\beta}^{\text{FL}}=
    \begin{aligned}
& \underset{\beta\in \real^{n}}{\text{argmin}}
& & \frac{1}{2N}\|y-X\beta\|_{2}^2+\lambda_{1}\| \beta \|_{1}+\lambda_{2} \|\Delta_{u}^{(1)}\beta\|,
\end{aligned}
\end{equation}
where $\lambda_{1}>0$ and $\lambda_{2}>0$ are tuning parameters, and $\Delta_{u}^{(1)}\in \mathbb{R}^{(n-1)\times n}$ is the first-order difference operator defined by  
\begin{equation}
\label{1storder}
\Delta^{(1)}_{u}=\left[\begin{array}{rrrrr}{-1} & {1} & {0} & {\dots} & {0} \\ {0} & {-1} & {1} & {\dots} & {0} \\ {\vdots} & {} & {\ddots} & {\ddots} \\ {0} & {0} & {\dots} & {-1} & {1}\end{array}\right].
\end{equation}
On the other hand, for applications where the nonzero components of $\beta^*$ might vary smoothly rather than being exactly locally constant, Hebiri and van de Geer \cite{hebiri2011} and Guo et al.\ \cite{guo2016} proposed the Smooth-Lasso and the Spline-Lasso, respectively, which are defined as follows: 
\begin{equation}
\label{smoothlasso}
\hat{\beta}^{\text{smooth}}=
    \begin{aligned}
& \underset{\beta\in \real^{n}}{\text{argmin}}
& & \frac{1}{2N}\|y-X\beta\|_{2}^2+\lambda_{1}\| \beta \|_{1}+\lambda_{2} \|\Delta_{u}^{(1)}\beta\|_{2}^{2},
\end{aligned}
\end{equation}
and 
\begin{equation}
\label{splinelasso}
\hat{\beta}^{\text{spline}}=
    \begin{aligned}
& \underset{\beta\in \real^{n}}{\text{argmin}}
& & \frac{1}{2N}\|y-X\beta\|_{2}^2+\lambda_{1}\| \beta \|_{1}+\lambda_{2} \|\Delta^{(2)}_{u}\beta\|_{2}^{2}.
\end{aligned}
\end{equation}
Here, the matrix $\Delta_{u}^{(2)}\in \mathbb{R}^{(n-2)\times n}$ is the second-order difference operator defined by   
\begin{equation*}
\Delta_{u}^{(2)}=\left[
\begin{array}{rrrrrrr}
{1} & {-2} & {1} & {0} & {\ldots} & {0} & {0} \\ 
{0} & {1} & {-2} & {1} & {\ldots} & {0} & {0} \\ 
{0} & {0} & {1} & {-2} & {\ldots} & {0} & {0} \\ 
{\vdots} & {} & {} & {\ddots} & {} & {\vdots} & {\vdots}  \\
{0} & {0} & {} & {\ldots}  & {1}  & {-2} & {1} 
\end{array}
\right]. 
\end{equation*}

The fused Lasso, the Smooth-Lasso, and the Spline-Lasso are popular examples of adaptive estimation, which is an active line of research in statistics. However, these existing methods can only deal with the univariate setting where $\beta^*$ is in a simple sequence form. In many real-world applications, the structure of $\beta^*$ might need to be modeled in the form of a complex graph. For instance, in the analysis of medical imaging, the regression vector $\beta^*$ may correspond to a two-dimensional grid graph which is smooth across adjacent nodes. Therefore, special care must be taken when considering such graph settings. In this paper, we are interested in studying scenarios where the coordinates of $\beta^*$ correspond to the nodes of  some \emph{known} underlying undirected graph. 

The primary focus of our paper is adaptive estimation and statistical inference for regression coefficients with graph-based piecewise polynomial structure, which has long been observed in multiple fields \cite{wang2016} and will be introduced in Section~\ref{background}. Our work significantly broadens the scope of existing methods. We summarize the main contributions of our paper as follows:  
\begin{itemize}
\item Introduce a family of adaptive estimators called the Graph-Piecewise-Polynomial-Lasso for the adaptive estimation problem. The central idea of our approach is to apply graph-based regularization techniques from the graph trend filtering by Wang et al.\ \cite{wang2016}, who generalized the idea of trend filtering \cite{kim2009, tibshirani2014} used for the univariate setting to graphs. Nevertheless, we emphasize two fundamental differences between the setting in our paper and that in \cite{wang2016}. Firstly, we consider the high-dimensional linear model given in $(\ref{linear})$, whereas Wang et al.\ focused on the Gaussian sequence model. Secondly, in addition to assuming the graph-based structure for the underlying signal as in \cite{wang2016}, we further assume the sparsity of $\beta^*$ based on the high-dimensional regime which will be studied in the paper. These two points make the theoretical properties and analysis of our proposed method very different from those in \cite{wang2016}. To the best of
our knowledge, this is the first work that studies piecewise polynomial structure for the linear model. 

\item Propose a new one-step estimator for valid statistical inference. Our proposed estimator shares the same asymptotic properties with other one-step estimators \cite{vandegeer2014, javanmard14a} in the statistics literature. However, our approach is computationally very attractive and requires much weaker conditions. We not only construct asymptotically valid confidence intervals for each component of $\beta^*$ based on the one-step estimator, but also discuss hypothesis testing in the graph setting.     

\item Provide theoretical guarantees and rigorous analysis for our approaches. We consider both the fixed design model and the random design model, and derive upper bounds for the $\ell_{2}$-estimation error, the $\ell_{1}$-estimation error, and the mean-squared prediction error of the Graph-Piecewise-Polynomial-Lasso. Remarkably, our theory does not require any assumptions on the underlying graph. Furthermore, we define the notion of weakly piecewise polynomial and sparse structure over graphs, and extend the corresponding theory to that case.   

\item Demonstrate that our approach outperforms other state-of-the-art methods in a wide variety of settings via extensive simulation studies and an application to an \emph{Arabidopsis thaliana} microarray dataset.
 
\end{itemize}

The remainder of the paper is organized as follows: In Section~\ref{background}, we describe our problem setup in detail and provide background on graphs and graph-based piecewise polynomial structure. In Section~\ref{estimationsection}, we consider the graph-based adaptive estimation problem and propose the Graph-Piecewise-Polynomial-Lasso. We also derive the theoretical properties of our proposed method in this section.  In Section~\ref{inferencesection}, we address the problem of statistical inference via the de-biased approach and state the corresponding theoretical guarantees. We apply our methods to synthetic and real data in Section~\ref{simulationsection} and Section~\ref{realdatasection}, respectively. Finally, we conclude with a discussion of open questions and further research directions in Section~\ref{discussion}. Proofs of all theoretical results and some supplementary simulations and real data analysis results are provided in the appendices. 

\begin{notation*}
For functions $f(n)$ and $g(n)$, we write $f(n)\lesssim g(n)$ to mean that $f(n)\le cg(n)$ for some universal constant $c\in (0, \infty)$. Similarly, we write $f(n)\gtrsim g(n)$ when $f(n)\ge c^{\prime}g(n)$ for some universal constant $c^{\prime}\in (0, \infty)$. We write $f(n)\asymp g(n)$ to mean that $f(n)\lesssim g(n)$  and $f(n)\gtrsim g(n)$ hold simultaneously. We write $f(n)=o(1)$ to mean that $f(n)\rightarrow 0$. We write $X_{n}=\mathcal{O}_{\mathbb{P}}(a_{n})$ to mean that $X_{n}/a_{n}$ is stochastically bounded. For an integer $r$, we write $[r]$ to denote the set $\{1,...,r\}$. We write $|S|$ to denote the cardinality of the set $S$ and $S^{c}$ to denote the complement set of $S$. We write $I_{n}$ to denote the identity matrix of size $n\times n$. For a matrix $M\in \real^{m\times n}$, we write $M_{S}\in \real^{s \times n}$ to denote the submatrix  of $M$ with rows restricted to $S$. We write $\text{null}(M)$ to denote the null space of $M$. We write $\opnorm{M}_{op}$ to denote the $\ell_{2}$-operator norm, $\opnorm{M}_{1}=\max_{j=1,...,n}\sum_{i=1}^{m}|M_{ij}|$ to denote the $\ell_{1}$-operator norm, and $\opnorm{M}_{\infty}=\max_{i=1,...,m}\sum_{j=1}^{n}|M_{ij}|$ to denote the $\ell_{\infty}$-operator norm. We write $\|M\|_{1,1}=\sum_{i=1}^{m}\sum_{j=1}^{n}|M_{ij}|$ to denote the element-wise $\ell_{1}$-norm. We use $\sigma_{\max}(M)$ and $\sigma_{\min}(M)$ to denote the maximum and minimum singular values, respectively. For a symmetric matrix $M$ of size $n\times n$, we write its eigenvalues $\lambda_{1}(M)\le \lambda_{2}(M)\le...\le \lambda_{n}(M)$. We write $\|\cdot\|_{\infty}$ to denote the element-wise infinity norm for both of vectors and matrices. We write $\|\cdot\|_{0}$ to denote the number of non-zero elements in a vector. For $q, r>0$, we write $\mathbb{B}_{q}(r)$ to denote the $\ell_{q}$-ball of radius $r$ centered around 0. We use $c, c^{\prime}, c^{\prime\prime}$, etc., to denote positive constants, where we may use the same notation to refer to different constants as we move between results.


\end{notation*}

\begin{definition}
\label{subgaussianrv}
For a constant $\sigma>0$, a random variable $X\in \mathbb{R}$ is said to be \emph{$\sigma$-sub-Gaussian} if $\mathbb{E}[X]=0$ and its moment generating function satisfies 
\begin{equation*}
\mathbb{E}\left[\exp{(tX)}\right]\le \exp{\left(\frac{\sigma^2t^2}{2}\right)},\quad\forall~t\in \mathbb{R}. 
\end{equation*}
Furthermore, a random vector $X\in \mathbb{R}^{n}$ is said to be \emph{$\sigma$-sub-Gaussian} if $\mathbb{E}[X]=0$ and $u^TX$ is $\sigma$-sub-Gaussian for any unit vector $u\in \mathbb{S}^{n-1}$.
\end{definition}

It can be shown that if $X$ is $\sigma$-sub-Gaussian, then for any $t>0$, we have
\begin{equation*}
\mathbb{P}\left(X>t\right)\le \exp{\left(-\frac{t^2}{2\sigma^2}\right)},\quad\mathbb{P}\left(X<-t\right)\le \exp{\left(-\frac{t^2}{2\sigma^2}\right)}. 
\end{equation*}

\begin{definition}
\label{randmat}
If a random matrix $X\in \mathbb{R}^{N\times n}$ is formed by drawing each row $X_{i}\in \mathbb{R}^{n}$ in an i.i.d.\ manner from a $\sigma$-sub-Gaussian distribution with covariance matrix $\Sigma$, then we say $X$ is a \emph{row-wise $(\sigma, \Sigma)$-sub-Gaussian random matrix}. 
\end{definition}


\section{Problem setup and background}
\label{background}

Throughout, we assume that the components of $\varepsilon$ in $(\ref{linear})$ are i.i.d.\ draws from a $\sigma_{\varepsilon}$-sub-Gaussian distribution defined in Definition~\ref{subgaussianrv}, unless otherwise stated. We are interested in both of the fixed design model, where $X$ is a deterministic matrix, and the random design model, where $X$ is a row-wise $(\sigma_{x}, \Sigma_x)$-sub-Gaussian matrix defined in Definition~\ref{randmat}. Furthermore, for the random design case, we also assume that $\varepsilon$ is independent of $X$. We denote the underlying graph by $\mathcal{G}=(\mathcal{V}, \mathcal{E})$, and assume that the number of nodes is $n$ and the number of edges  is $p$. We also denote the maximum degree of $\mathcal{G}$ by $d$. 

Next, we provide some preliminaries for defining the graph-based spatial structure which will be studied in the paper. We begin by introducing the oriented incidence matrix and the graph Laplacian matrix, which are common tools for studying graphs. 
\begin{definition}
\label{diffoperator}
The \emph{oriented incidence matrix}, denoted by $F\in \{-1, 0, 1\}^{p\times n}$, is defined as follows: if the $k$-th edge is $(i, j)\in \mathcal{E}$ with $i<j$, then the $k$-th row of $F$ is
\begin{equation*}
(0,...,-1,...,+1,...,0),
\end{equation*}
where $-1$ is in the $i$-th entry and $+1$ is in the $j$-th entry. Furthermore, we call $L=F^TF\in \mathbb{R}^{n\times n}$ the \emph{Laplacian matrix}. 
\end{definition}

In terms of spectral properties, it is well known that the smallest eigenvalue of the Laplacian matrix $L$ is 0. Furthermore, the largest eigenvalue of the Laplacian matrix can be bounded by $2d$, which is proved in Lemma~\ref{eigenvaluelem} in Appendix~\ref{supporting-appendix}.\footnote{The result in Lemma~\ref{eigenvaluelem} is quite rough, but it is still practically useful for several interesting graphs with bounded maximum degrees. For example, path graphs have maximum degree 2 and two-dimensional grid graphs have maximum degree 4. We refer the reader to sharper results on upper bounds of the largest eigenvalue of the Laplacian matrix in \cite{Wang2007} and \cite{ShuHongKai}.} The oriented incidence matrix and the Laplacian matrix are also called the first and second graph difference operators, respectively. Furthermore, the graph difference operator with order greater than 2 is defined in \cite{wang2016} by the following recursion:         

\begin{definition}
\label{diffoperator2}
For $k>1$, the \emph{graph difference operator} of order $k+1$, denoted by $\Delta^{(k+1)}$, is
\begin{equation*}
\Delta^{(k+1)}=
    \begin{cases}
      F^T\Delta^{(k)}=L^{\frac{k+1}{2}} & \text{for odd k} \\
      F\Delta^{(k)}=FL^{\frac{k}{2}}& \text{for even k}. 
   \end{cases}
\end{equation*}
\end{definition}

For ease of notation, we use $m$ to denote the number of rows for $\Delta^{(k+1)}$, i.e., $m=n$ for odd $k$ and $m=p$ for even $k$. The graph difference operator in Definition~\ref{diffoperator2} is closely related to the usual difference operator in the univariate setting, which is also defined recursively by 
\begin{equation*}
\Delta_{u}^{(k+1)}=\Delta^{(1)}_{u}\Delta_{u}^{(k)}\in \mathbb{R}^{(n-k-1)\times n}. 
\end{equation*}
 Here, the matrix $\Delta_{u}^{(1)}$ is the $(n-k-1)\times (n-k)$ version of $(\ref{1storder})$. It can be shown that if the underlying graph is a path graph, then when $k$ is even, removing the first $\frac{k}{2}$ rows and the last $\frac{k}{2}$ rows of $\Delta^{(k+1)}$ recovers $\Delta_{u}^{(k+1)}$, and when $k$ is odd, removing the first $\frac{k+1}{2}$ rows and the last $\frac{k+1}{2}$ rows of $\Delta^{(k+1)}$ recovers $\Delta_{u}^{(k+1)}$.  

With the graph difference operator at hand, we are now in a position to introduce the graph-based piecewise polynomial structure. Wang et al.\ \cite{wang2016} also defined a similar notion for studying the trend filtering problem of the Gaussian sequence model. 

\begin{definition}
\label{piecewisepoly}
For $k\ge 0$ and $s>0$, if $\|\Delta^{(k+1)}\beta^*\|_{0}\le s$, then $\beta^*$ is called \emph{$(k, s)$-piecewise polynomial} over the underlying graph $\mathcal{G}$. 
\end{definition}

It is obvious that $(0, s)$-piecewise polynomial structure implies 
\begin{equation}
\label{piecewiseconstant}
 |\{ (i, j)\in \mathcal{E}: \beta_{i}^*\ne \beta_{j}^*\}|\le s. 
\end{equation} 
When $\mathcal{G}$ is a path graph, condition $(\ref{piecewiseconstant})$ is equivalent to piecewise constant structure in the univariate setting, so we say $\beta^*$ which satisfies $(\ref{piecewiseconstant})$ has \emph{$s$-piecewise constant} structure. Similarly, we refer to $(1, s)$-and $(2, s)$-piecewise polynomial structures as \emph{$s$-piecewise linear} and \emph{$s$-piecewise quadratic}, respectively. We will further extend the graph-based piecewise polynomial structure in Definition~\ref{piecewisepoly} to the notion of weakly piecewise polynomial structure in Section~\ref{extension-section}.  
 
We will work within the high-dimensional framework, which allows the number of predictors $n$ to grow and exceed the sample size $N$. Hence, it is also very natural for us to assume sparsity of $\beta^*$. Therefore, in this paper, we focus on the type of regression coefficients which are simultaneously sparse and piecewise polynomial over the underlying graph $\mathcal{G}$. In other words, we are interested in the parameter space of $\beta^*$ defined by 
\begin{equation}
\label{betaspace}
\mathcal{S}(k, s_{1}, s_{2})=\left\{\beta\in \mathbb{R}^{n}: \|\Delta^{(k+1)}\beta\|_{0}\le s_{1}, \|\beta\|_{0}\le s_{2}\right\},
\end{equation}
where $s_{1}>0$ and $s_{2}>0$ are allowed to increase with the triple $(N, n, p)$, and $k$ is a \emph{fixed} and \emph{known} user-specified integer. See Section~\ref{discussion} for a discussion of the case where $k$ is unknown. Figure~\ref{piecewisepolypathgraph} and Figure~\ref{piecewisepoly2dgraph} show some instances of regression coefficients which are simultaneously $(k, s_{1})$-piecewise polynomial and $s_{2}$-sparse for specific $k$, $s_{1}$, and $s_{2}$ over the path graph and the 2d grid graph, respectively.  

\begin{figure}
\begin{subfigure}{.5\textwidth}
  \centering
  \includegraphics[width=.8\linewidth]{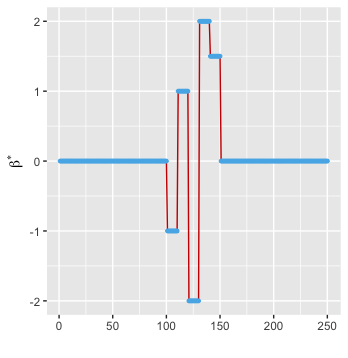}  
  \caption{$k=0, s_{1}=6, s_{2}=50, \Delta^{(1)}\in \mathbb{R}^{249\times 250}$}
\end{subfigure}
\begin{subfigure}{.5\textwidth}
 \centering
  \includegraphics[width=.8\linewidth]{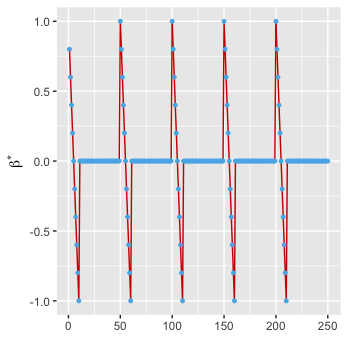}  
  \caption{$k=1, s_{1}=19, s_{2}=49, \Delta^{(2)}\in \mathbb{R}^{250\times 250}$}
\end{subfigure}
\begin{center}
\begin{subfigure}{.5\textwidth}
  \includegraphics[width=.8\linewidth]{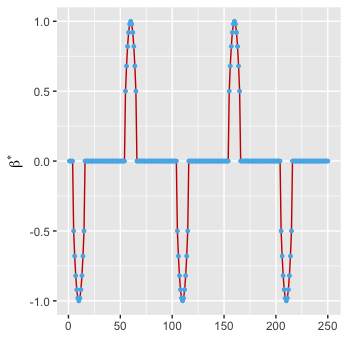}  
  \caption{$k=2, s_{1}=30, s_{2}=55, \Delta^{(3)}\in \mathbb{R}^{249\times 250}$}
\end{subfigure}
\end{center}
\caption{Three examples of simultaneously piecewise polynomial and sparse regression coefficients over a path graph. Coordinates of $\beta^*$ correspond to a path graph with 250 nodes ($n=250$ and $p=249$). See Section~\ref{simu1section} for details about the construction of $\beta^*$ in (a), (b) and (c).}
\label{piecewisepolypathgraph}
\end{figure}

\begin{figure}
\begin{center}
\begin{subfigure}{.6\textwidth}
  \includegraphics[width=0.475\linewidth]{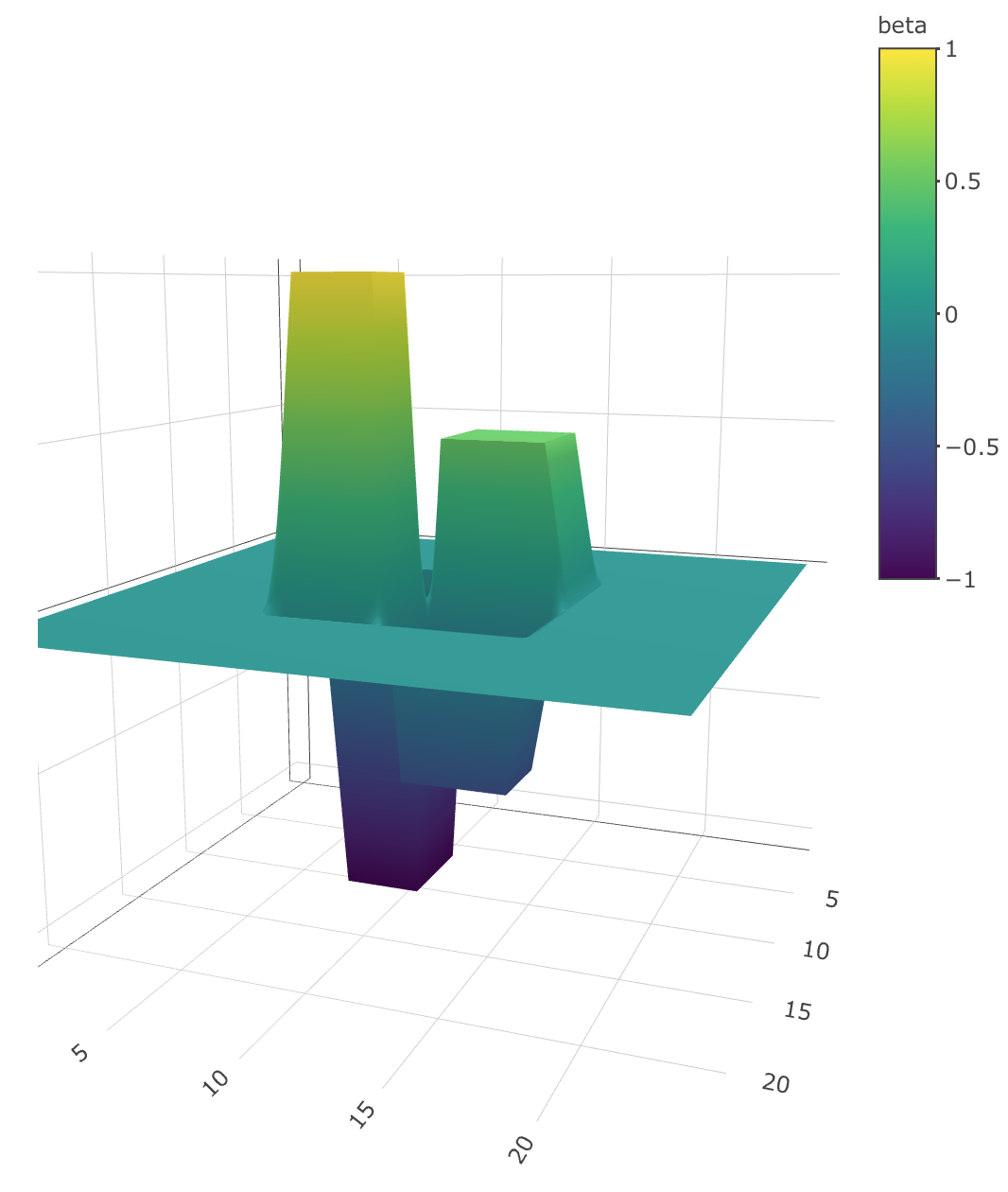}  
 \hfill
  \includegraphics[width=0.475\linewidth]{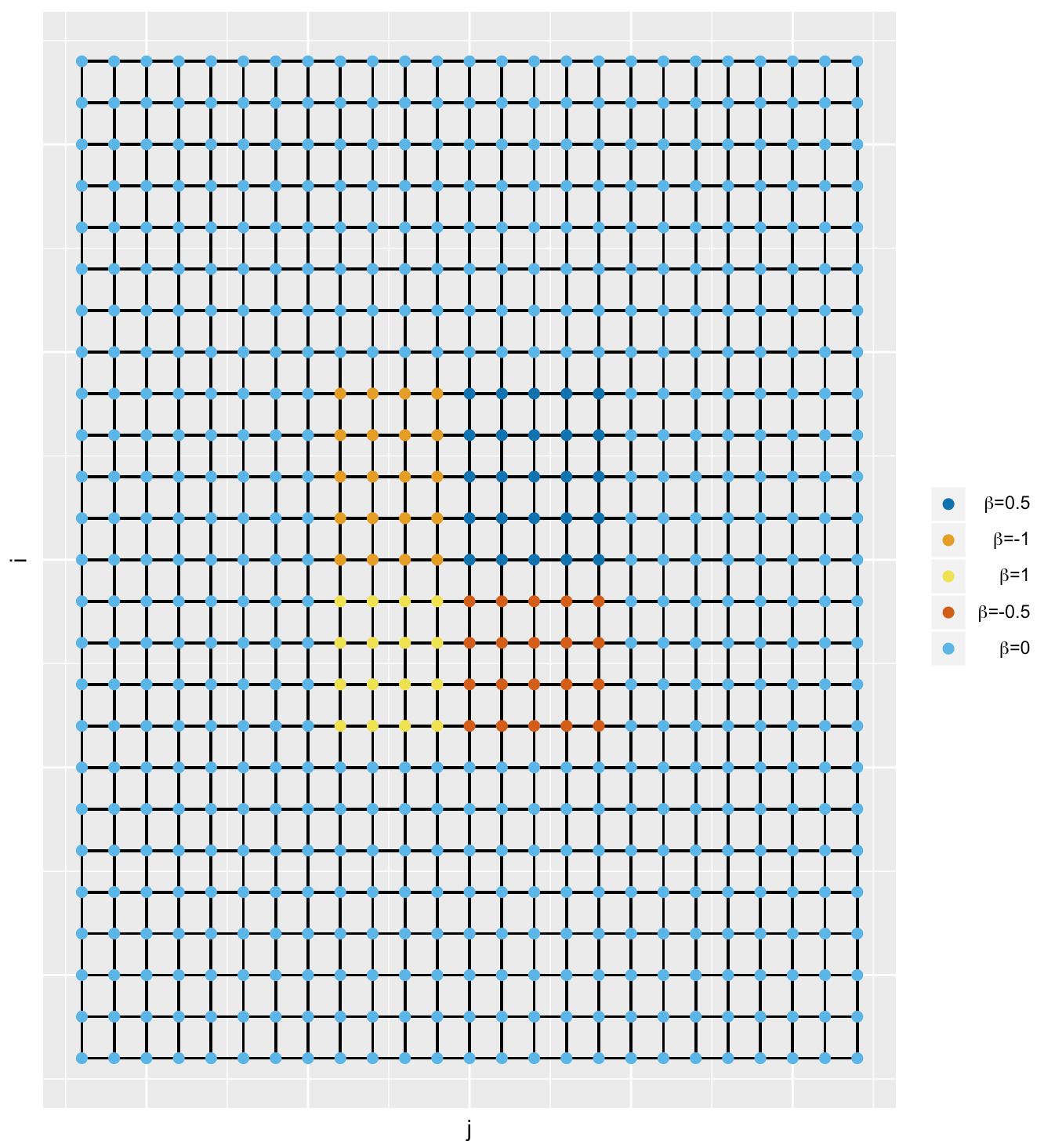} 
   \caption{$k=0, s_{1}=54, s_{2}=81, \Delta^{(1)}\in \mathbb{R}^{1200\times 625}$}   
\end{subfigure}
\end{center}
\begin{center}
\begin{subfigure}{.6\textwidth}
  \includegraphics[width=0.475\linewidth]{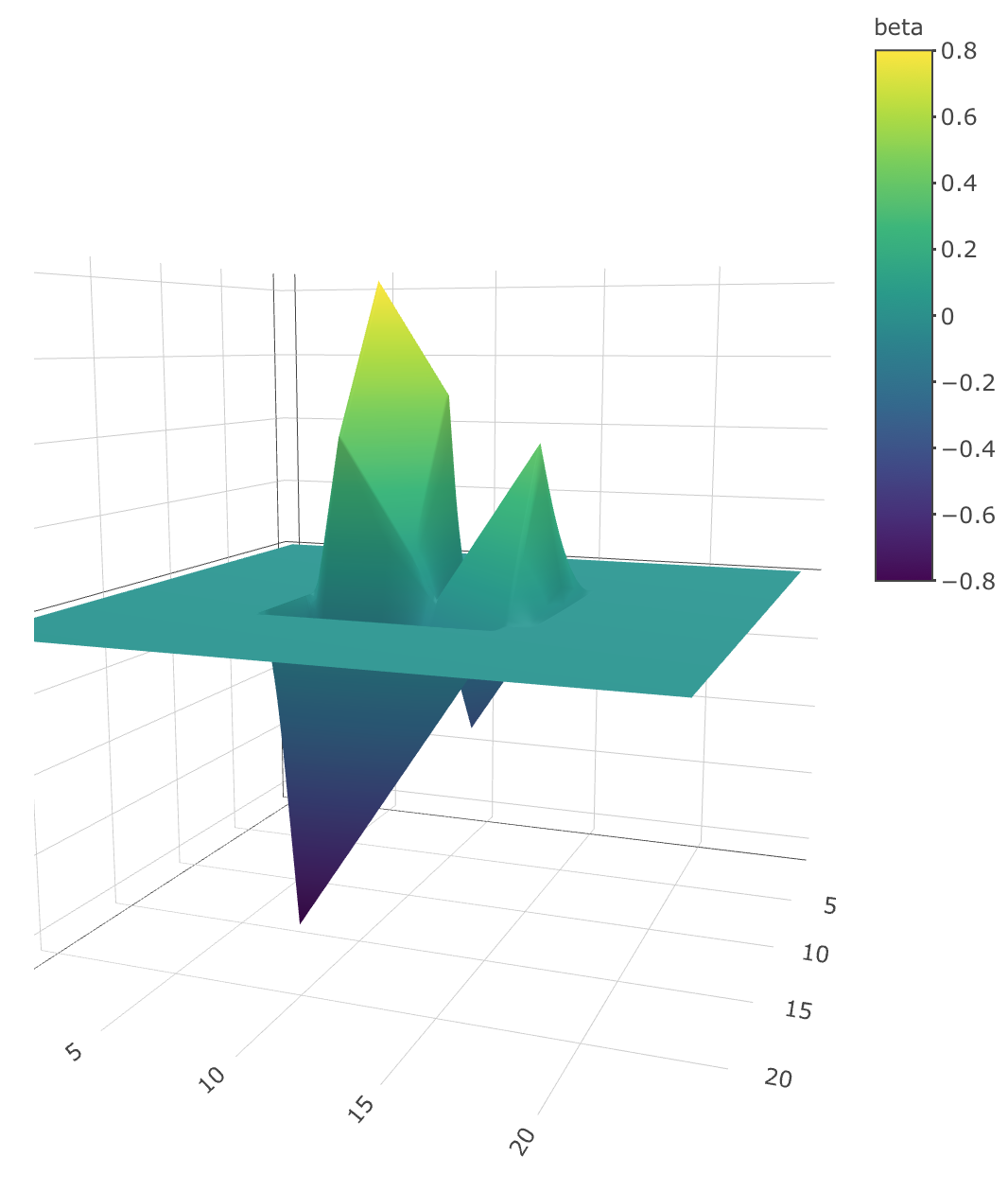}  
 \hfill
  \includegraphics[width=0.475\linewidth]{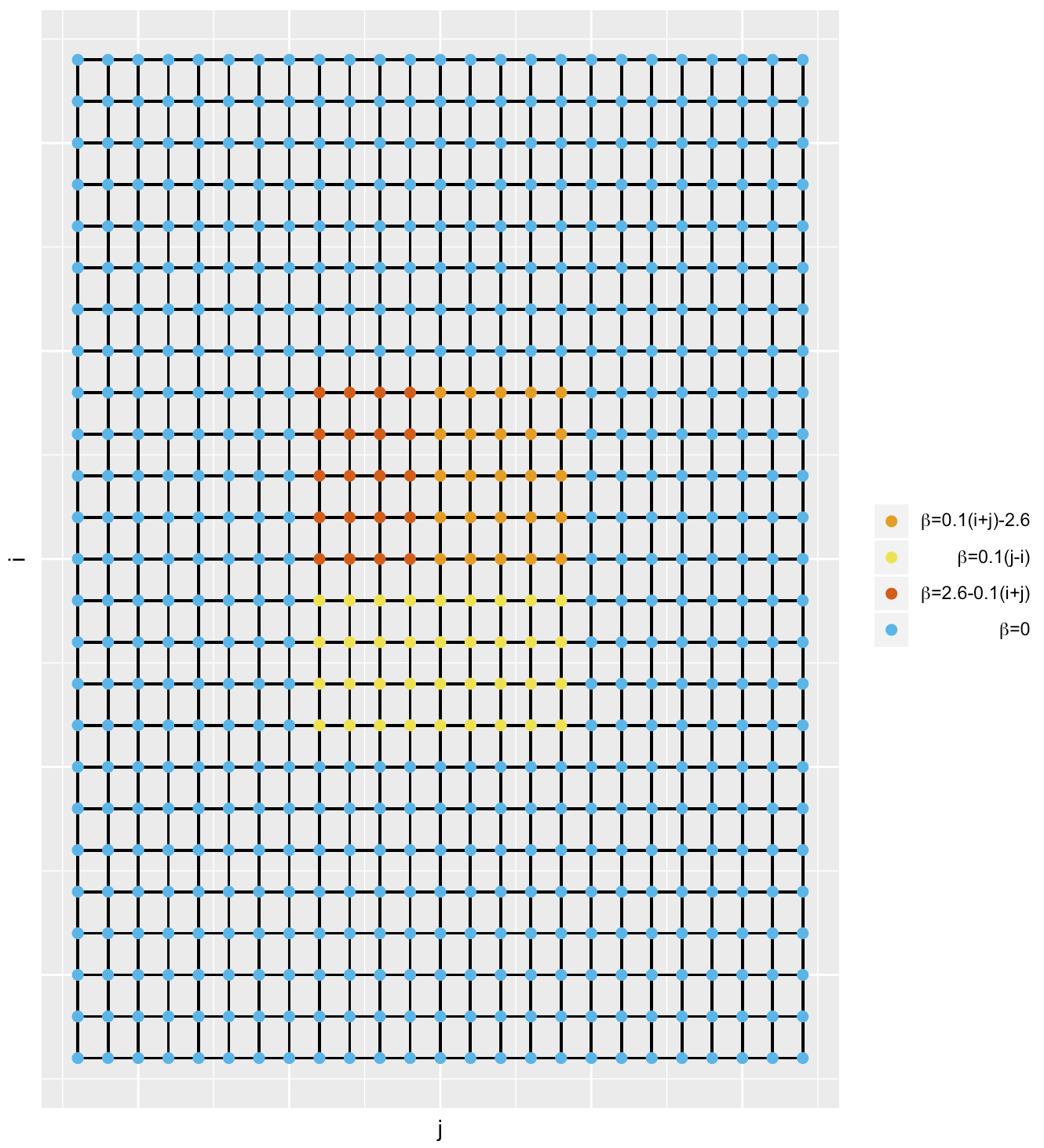}  
  \caption{$k=1, s_{1}=77, s_{2}=72, \Delta^{(2)}\in \mathbb{R}^{625\times 625}$}
\end{subfigure}
\end{center}
\begin{center}
\begin{subfigure}{.6\textwidth}
  \includegraphics[width=0.475\linewidth]{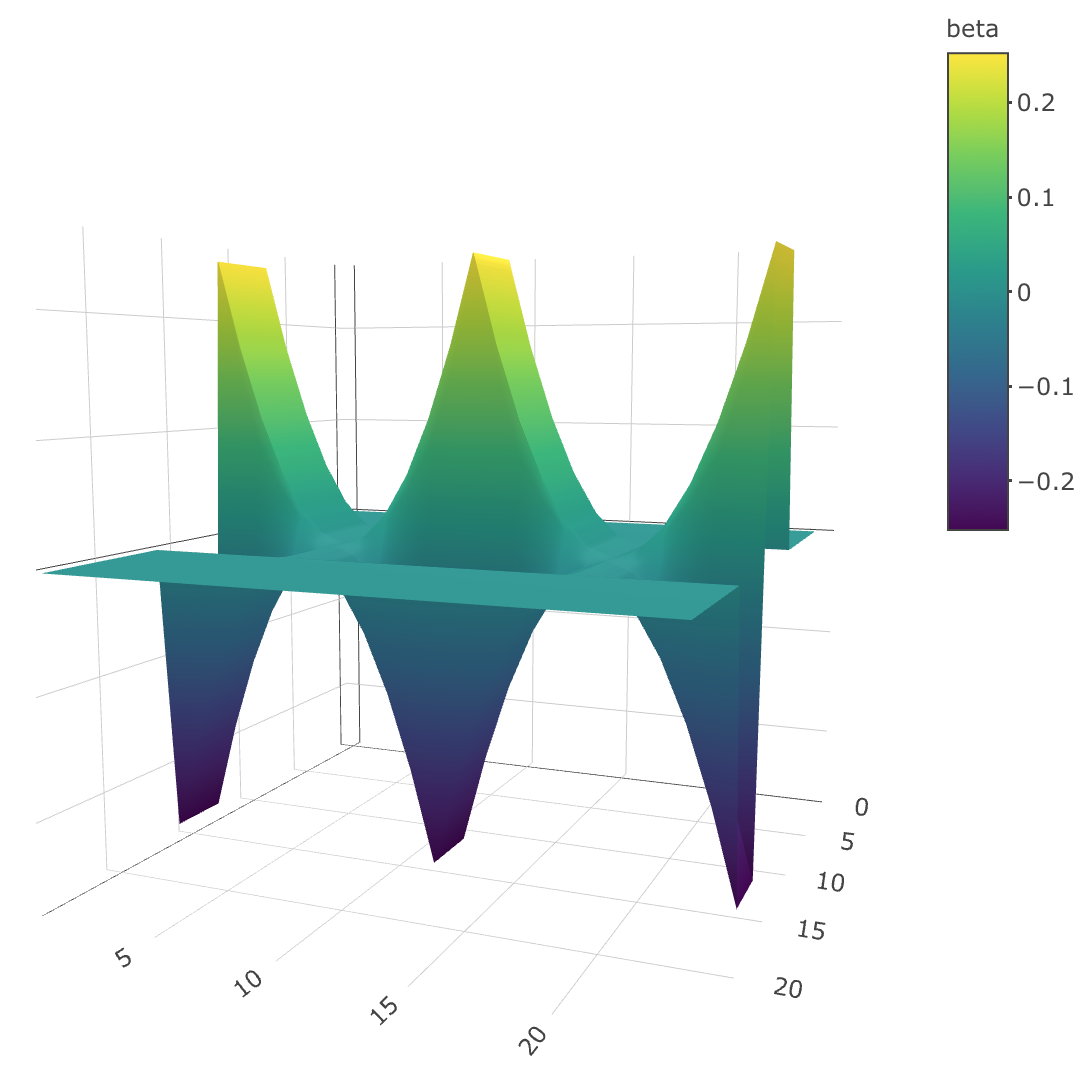}  
\hfill
  \includegraphics[width=0.475\linewidth]{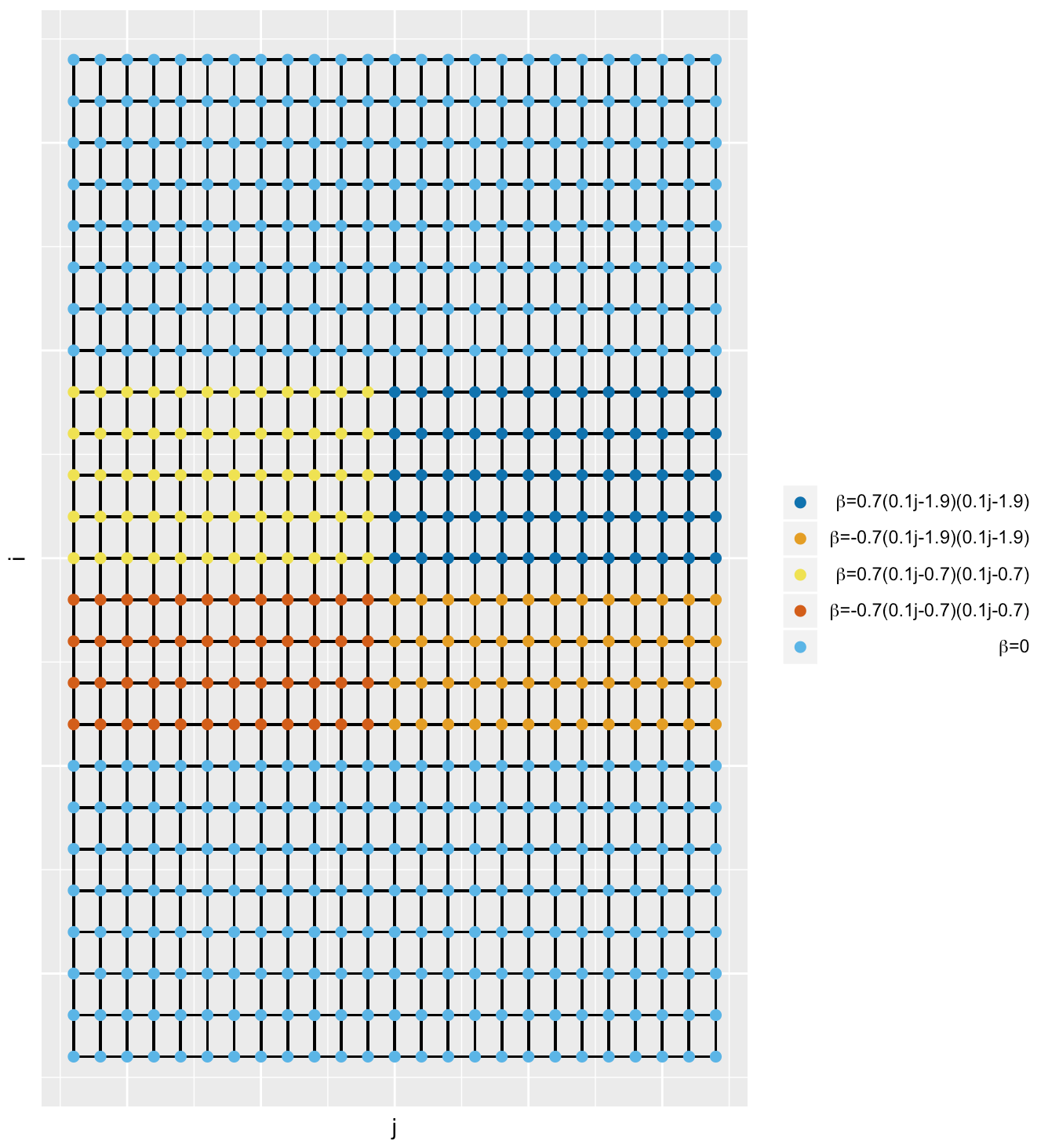}  
 \caption{$k=2, s_{1}=365, s_{2}=207, \Delta^{(3)}\in \mathbb{R}^{1200\times 625}$}
 \end{subfigure}
 \end{center}
\caption{Three examples of simultaneously piecewise polynomial and sparse regression coefficients over a 2d grid graph. Coordinates of $\beta^*$ correspond to a 2d grid graph with 25 rows and 25 columns ($n=625$ and $p=1200$). Figures in the right-hand side show the value of $\beta^*$ in each node. See more details in Section~\ref{simu2section}. }
\label{piecewisepoly2dgraph}
\end{figure}


\section{Graph-based adaptive estimation}
\label{estimationsection}

In this section, we propose an adaptive estimation procedure for regression coefficients which are simultaneously piecewise polynomial and sparse over the underlying graph, and then present the main theoretical results for deterministic and random designs. We also extend the theory to the case of weakly piecewise polynomial and sparse structure, which will be defined in Definition~\ref{weaklypiecewisepoly}. 

\subsection{The Graph-Piecewise-Polynomial-Lasso}

To estimate $\beta^*$ in the  parameter space defined by $(\ref{betaspace})$, we propose the following adaptive estimator: 
\begin{equation}
\label{tf-sl}
\hat{\beta}=
    \begin{aligned}
& \underset{\beta\in \real^{n}}{\text{argmin}}
& & \frac{1}{2N}\|y-X\beta\|_{2}^2+\lambda_{g} \|\Delta^{(k+1)}\beta\|_{1}+\lambda\| \beta \|_{1}, 
\end{aligned}
\end{equation}
where $\lambda_{g}>0$ and $\lambda>0$ are tuning parameters. Here, the $\ell_{1}$-regularizers $\|\Delta^{(k+1)}\beta\|_{1}$ and $\|\beta\|_{1}$ are used to encourage the sparsity of $\Delta^{(k+1)}\beta$ and $\beta$, respectively. We refer to $\hat{\beta}$ as the \emph{$k$-th order Graph-Piecewise-Polynomial-Lasso} in this paper. The optimization problem $(\ref{tf-sl})$ is equivalent to 
\begin{equation*}
\hat{\beta}=
    \begin{aligned}
& \underset{\beta\in \real^{n}}{\text{argmin}}
& & \frac{1}{2N}\|y-X\beta\|_{2}^2+\lambda\left\|D\beta\right\|_{1}, 
\end{aligned}
\end{equation*}
where
\begin{equation}
\label{Dmatrix}
D=\begin{bmatrix}\frac{\lambda_{g}}{\lambda}\Delta^{(k+1)}\\ I_{n} \end{bmatrix}\in \mathbb{R}^{(m+n)\times n}. 
\end{equation}
Note that $D$ has full column rank. We write $D^{+}=(D^TD)^{-1}D^T\in \mathbb{R}^{n\times (m+n)}$ to denote its Moore-Penrose inverse. It is clear that when the underlying graph $\mathcal{G}$ is a path graph and $k=0$, the problem $(\ref{tf-sl})$ is identical with the fused Lasso in $(\ref{fusedlassoprog})$. 

We applied our approach $(\ref{tf-sl})$, Lasso (\ref{lasso}), Smooth-Lasso (\ref{smoothlasso}), and Spline-Lasso (\ref{splinelasso}) to estimate three scenarios of $\beta^*$ shown in Figure~\ref{piecewisepolypathgraph}, respectively. The estimated regression coefficients are displayed in Figure~\ref{piececonstantbetagraph}, Figure~\ref{piecelinearbetagraph}, and Figure~\ref{piecequadraticbetagraph}, respectively. Overall, our approach performs better than other approaches for recovering the desired structures over the path graph.     
Furthermore, we can use a similar idea as $(\ref{tf-sl})$ to extend the Smooth-Lasso (\ref{smoothlasso}) and the Spline-Lasso (\ref{splinelasso}) to the graph setting as follows: 
\begin{equation}
\label{gsmoothlasso}
\hat{\beta}^{\text{gsmooth}}=
    \begin{aligned}
& \underset{\beta\in \real^{n}}{\text{argmin}}
& & \frac{1}{2N}\|y-X\beta\|_{2}^2+\lambda_{1}\| \beta \|_{1}+\lambda_{2} \|\Delta^{(1)}\beta\|_{2}^{2}, 
\end{aligned}
\end{equation}
and 
\begin{equation}
\label{gsplinelasso}
\hat{\beta}^{\text{gspline}}=
    \begin{aligned}
& \underset{\beta\in \real^{n}}{\text{argmin}}
& & \frac{1}{2N}\|y-X\beta\|_{2}^2+\lambda_{1}\| \beta \|_{1}+\lambda_{2} \|\Delta^{(2)}\beta\|_{2}^{2}.  
\end{aligned}
\end{equation}
We will refer to $\hat{\beta}^{\text{gsmooth}}$ and $\hat{\beta}^{\text{gspline}}$ as the \emph{Graph-Smooth-Lasso} and the \emph{Graph-Spline-Lasso}, respectively. In Appendix~\ref{simuappendix}, we compare the performance of our approach with the Lasso, Graph-Smooth-Lasso, and Graph-Spline-Lasso in terms of the structure recovery for the regression coefficients over a 2d grid graph. The simulation results strongly suggest using our approach in practice if $\beta^*$ has the spatial structure considered in the paper. 
  
\begin{figure}[!t]
\begin{subfigure}{.5\textwidth}
  \centering
  \includegraphics[width=.8\linewidth]{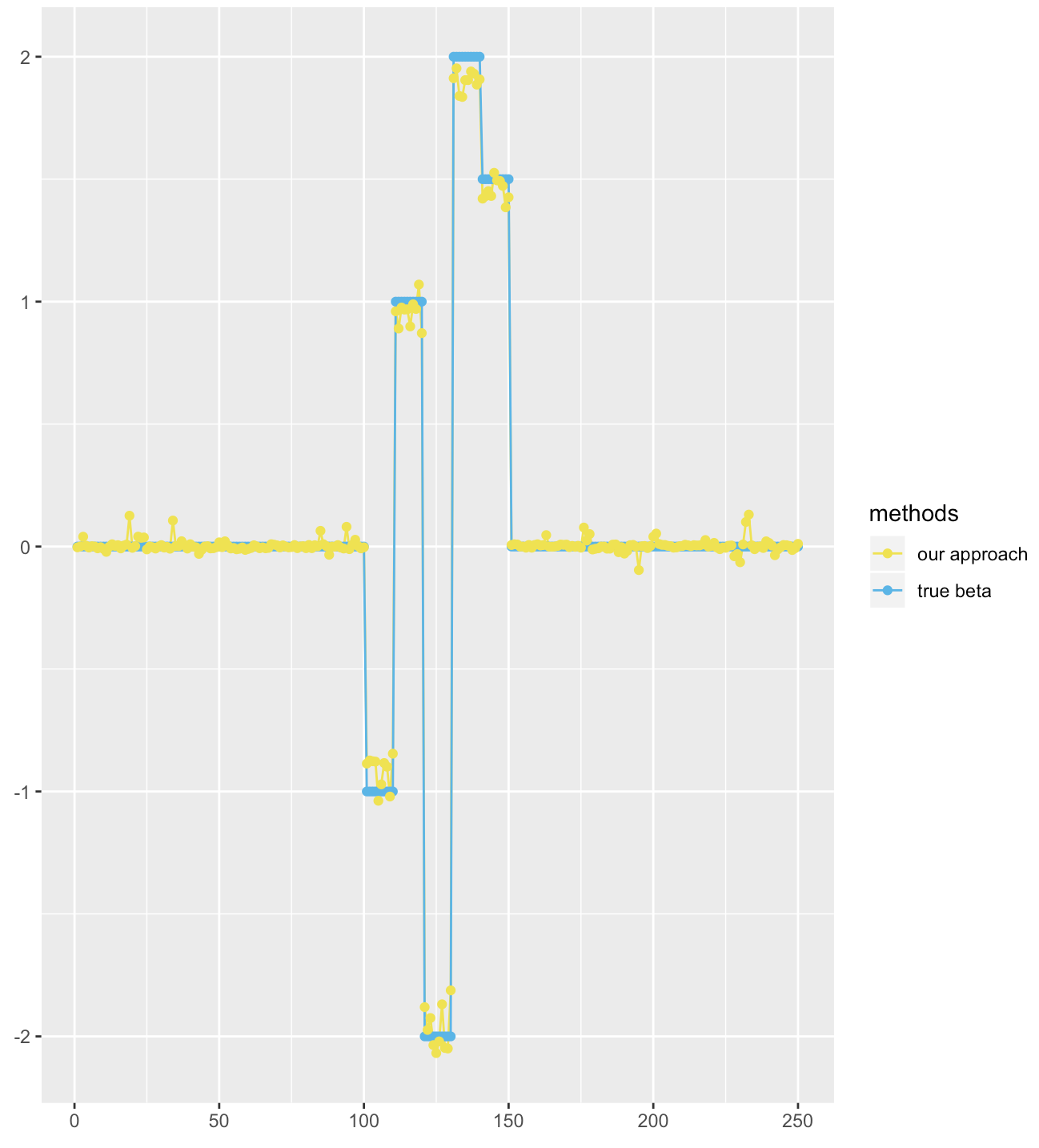}  
\end{subfigure}
\begin{subfigure}{.5\textwidth}
 \centering
  \includegraphics[width=.8\linewidth]{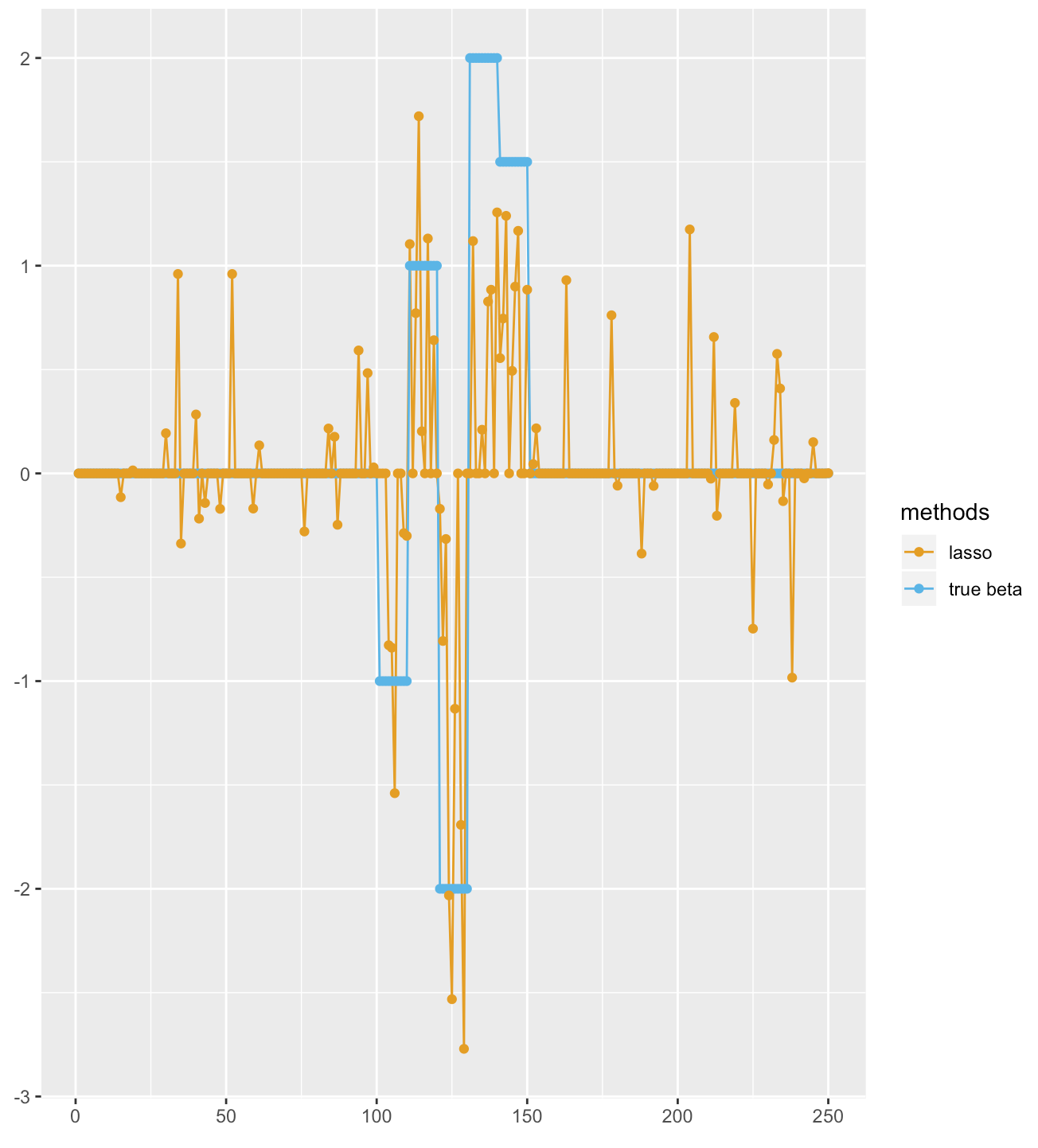}  
\end{subfigure}
\begin{subfigure}{.5\textwidth}
  \centering
  \includegraphics[width=.8\linewidth]{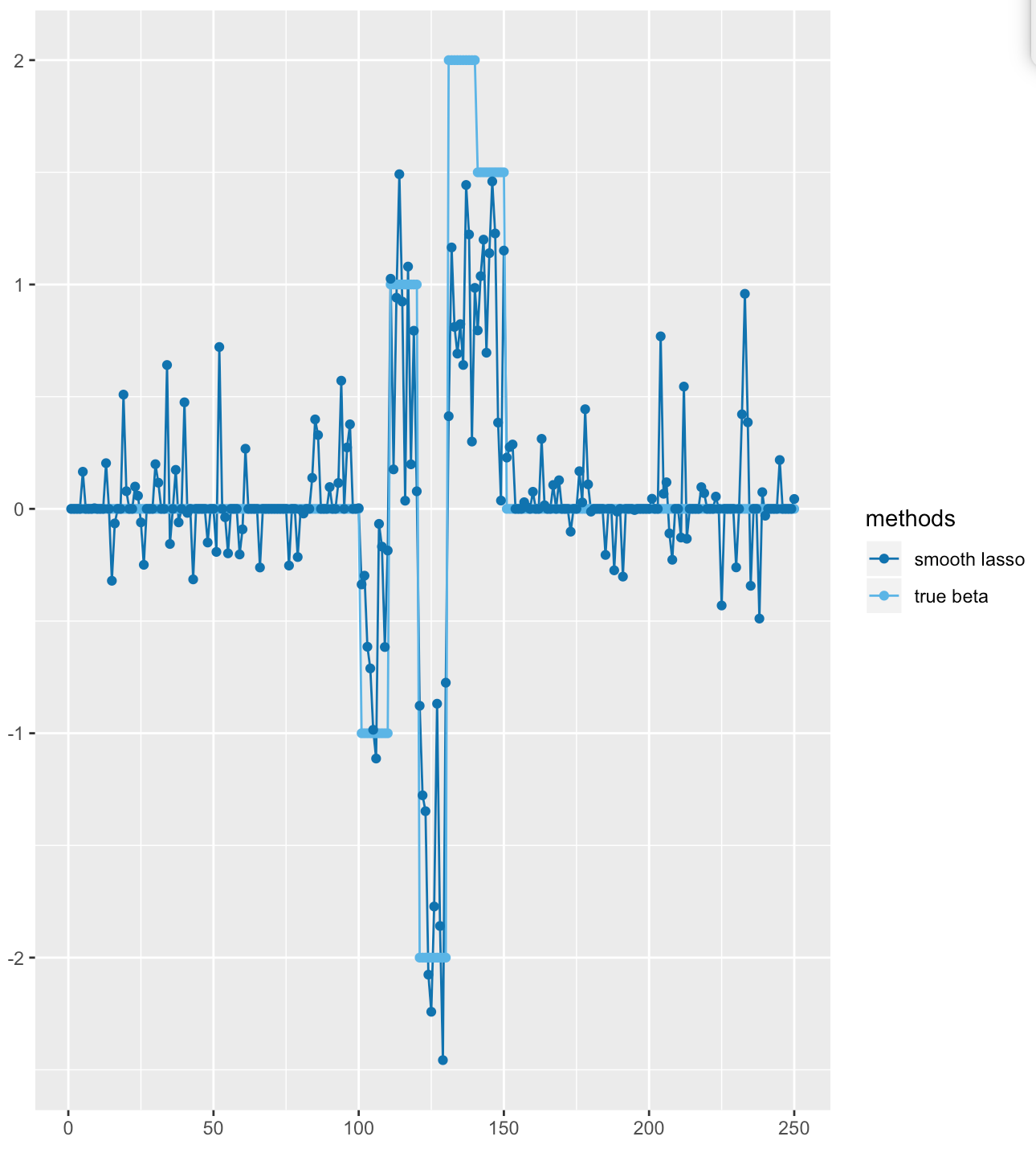}  
\end{subfigure}
\begin{subfigure}{.5\textwidth}
 \centering
  \includegraphics[width=.8\linewidth]{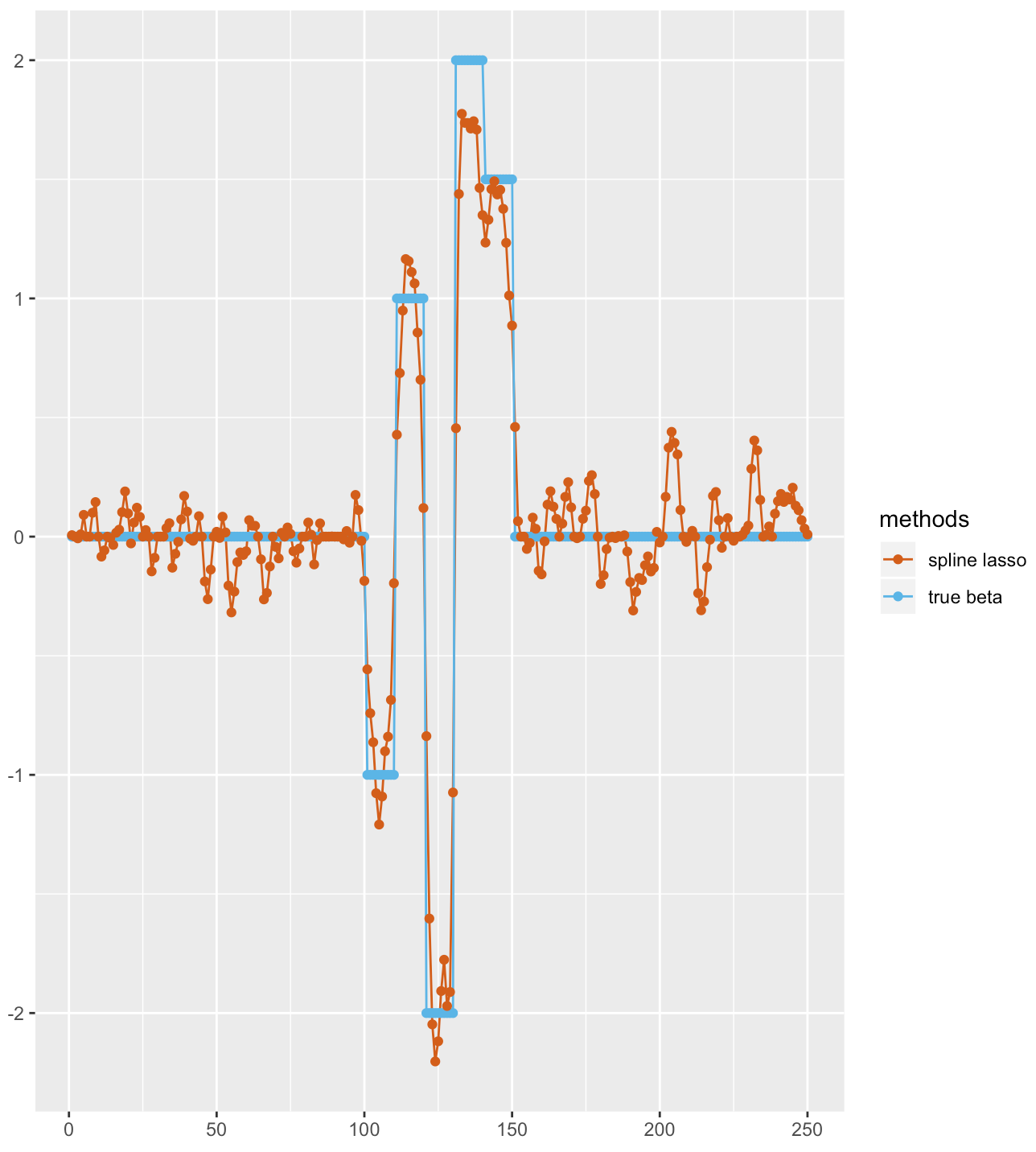}  
\end{subfigure}
\caption{We consider estimation of $\beta^*$ plotted in (a) of Figure~\ref{piecewisepolypathgraph}. We set $N=100$. The data generating process and the tuning parameter selection are the same with the simulation in Section~\ref{simu1section}. Each panel displays the true regression coefficients and the estimated values. }
\label{piececonstantbetagraph}
\end{figure}

\begin{figure}[!t]
\begin{subfigure}{.5\textwidth}
  \centering
  \includegraphics[width=.8\linewidth]{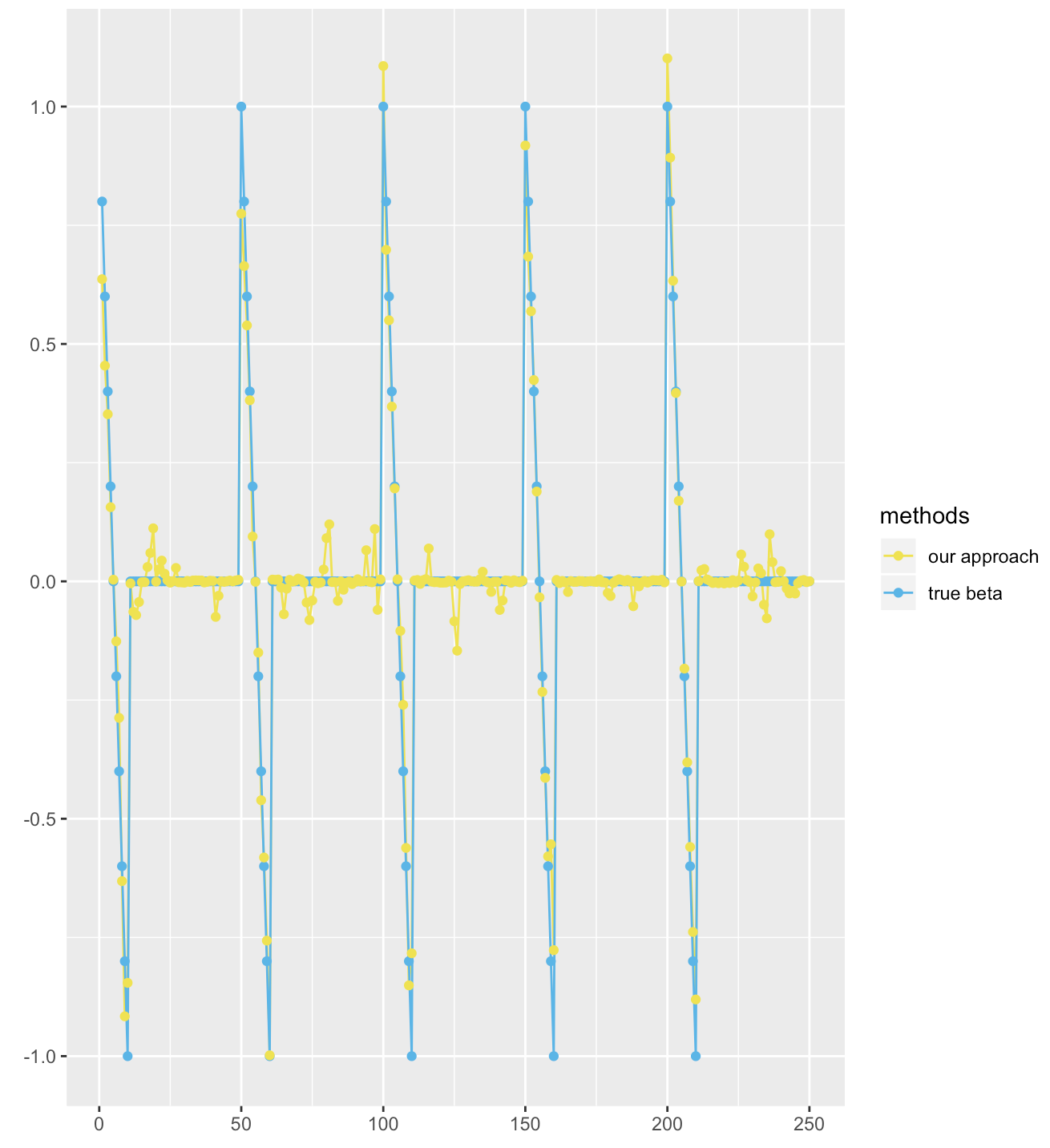}  
\end{subfigure}
\begin{subfigure}{.5\textwidth}
 \centering
  \includegraphics[width=.8\linewidth]{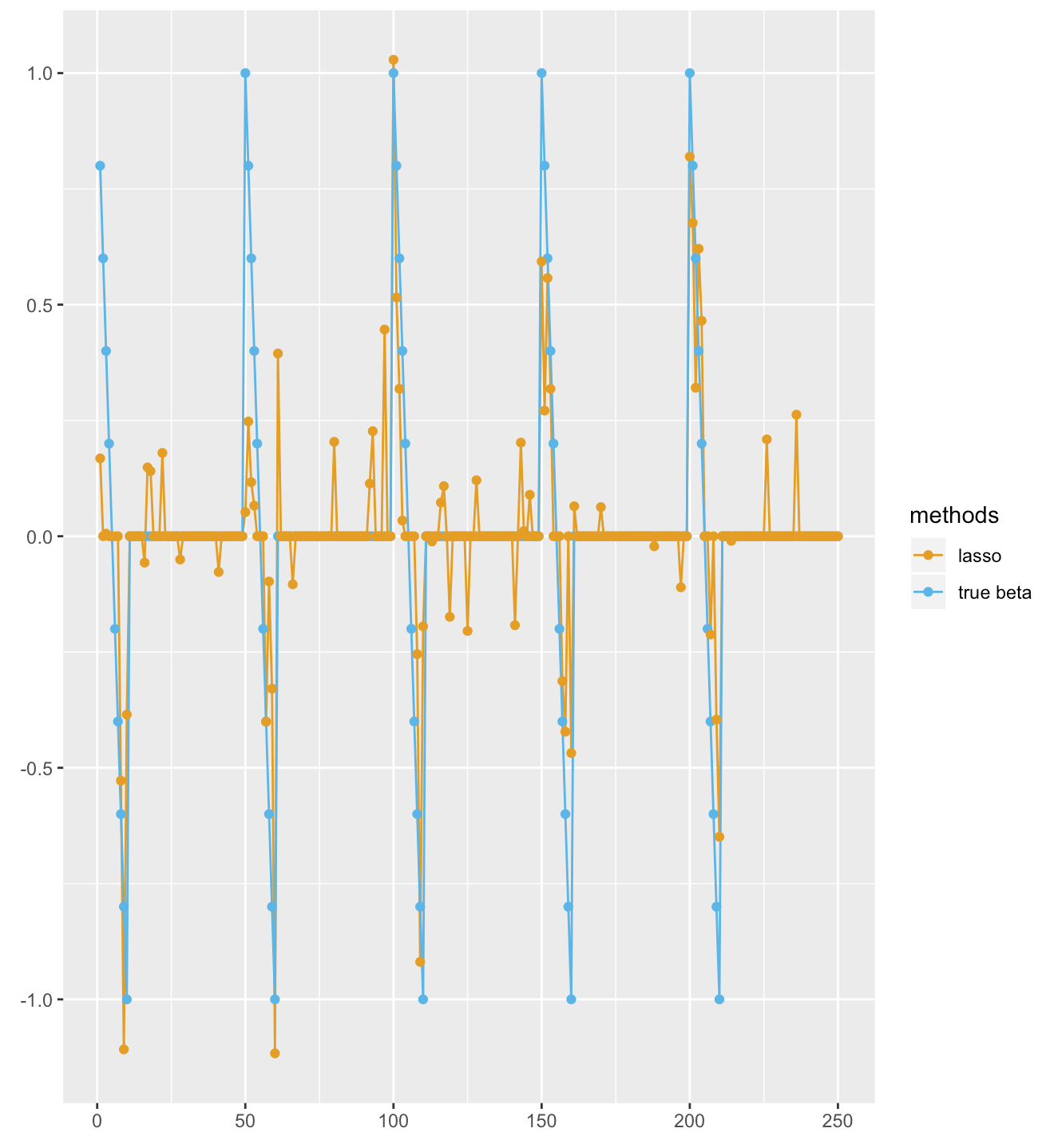}  
\end{subfigure}
\begin{subfigure}{.5\textwidth}
  \centering
  \includegraphics[width=.8\linewidth]{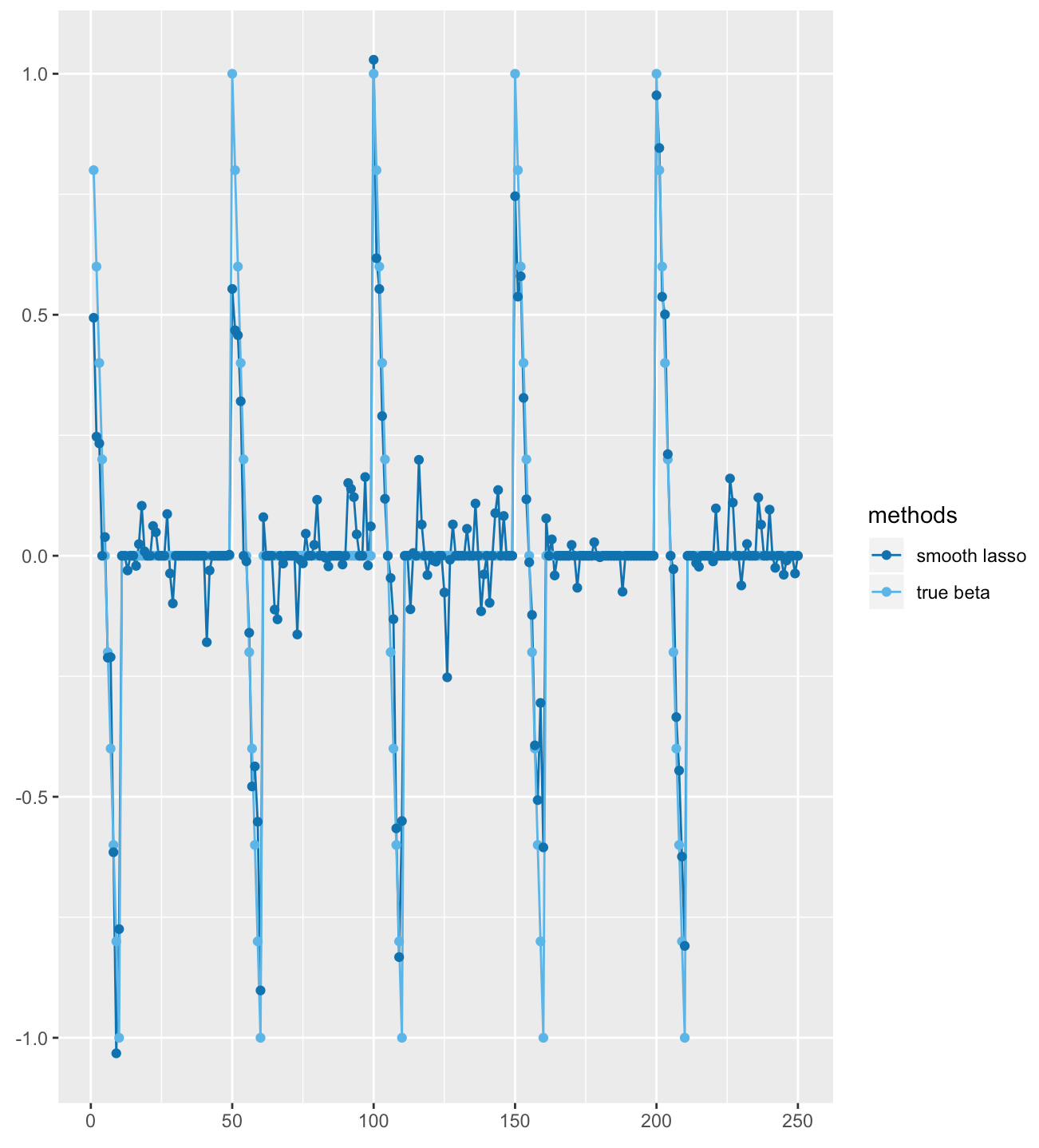}  
\end{subfigure}
\begin{subfigure}{.5\textwidth}
 \centering
  \includegraphics[width=.8\linewidth]{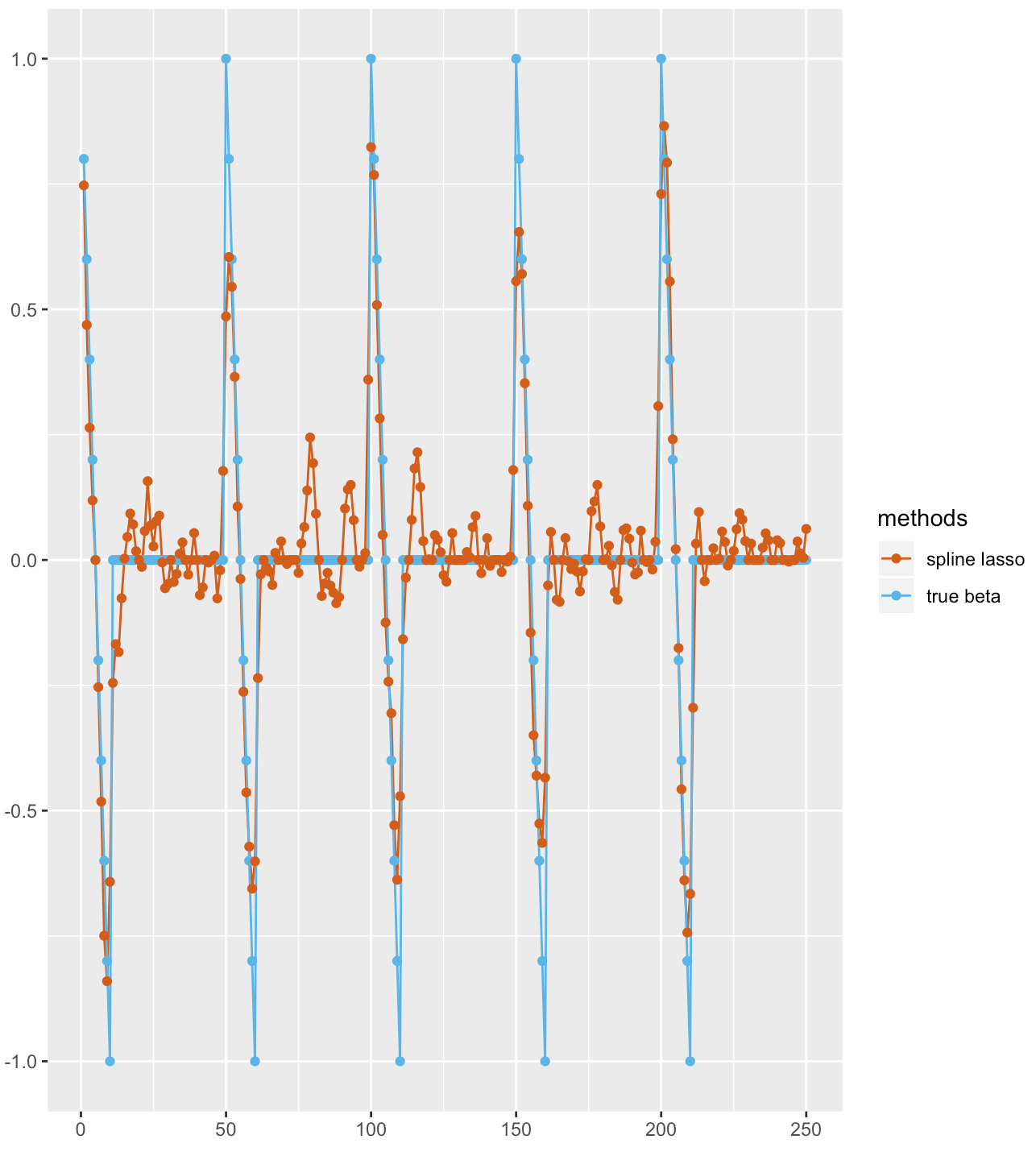}  
\end{subfigure}
\caption{We consider estimation of $\beta^*$ plotted in (b) of Figure~\ref{piecewisepolypathgraph}. We set $N=100$. The data generating process and the tuning parameter selection are the same with the simulation in Section~\ref{simu1section}. Each panel displays the true regression coefficients and the estimated values.}
\label{piecelinearbetagraph}
\end{figure}

\begin{figure}[!t]
\begin{subfigure}{.5\textwidth}
  \centering
  \includegraphics[width=.8\linewidth]{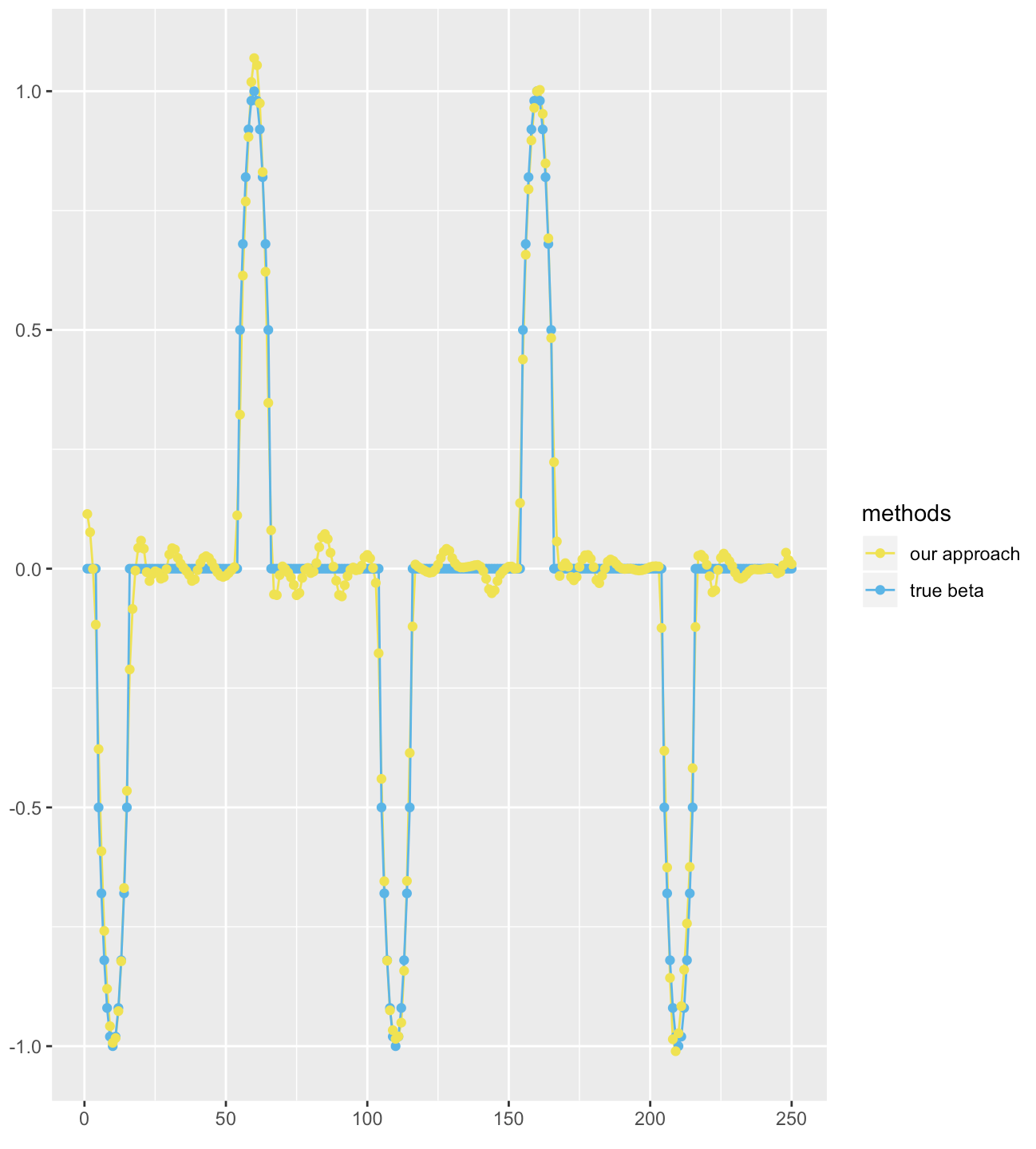}  
\end{subfigure}
\begin{subfigure}{.5\textwidth}
 \centering
  \includegraphics[width=.8\linewidth]{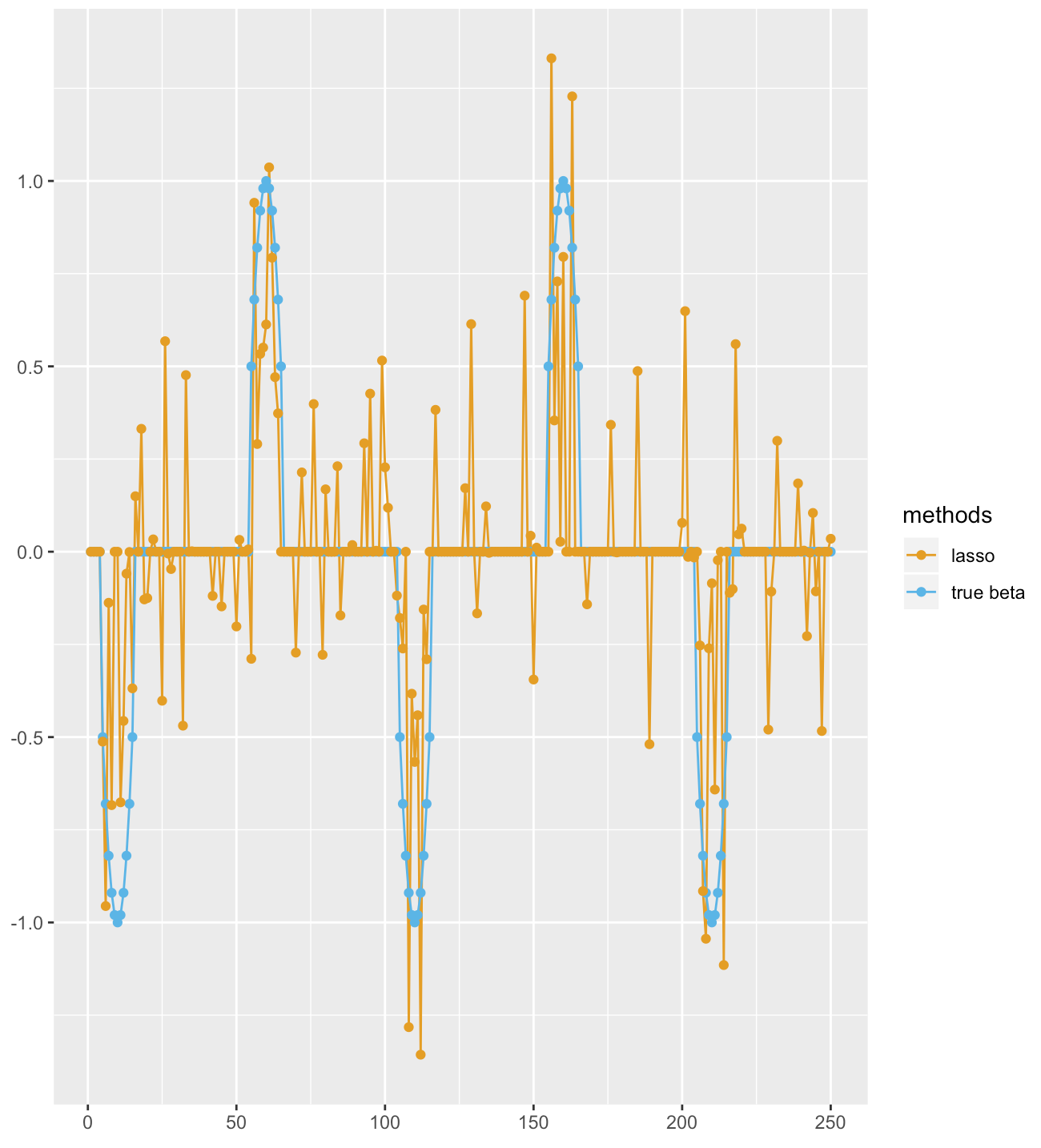}  
\end{subfigure}
\begin{subfigure}{.5\textwidth}
  \centering
  \includegraphics[width=.8\linewidth]{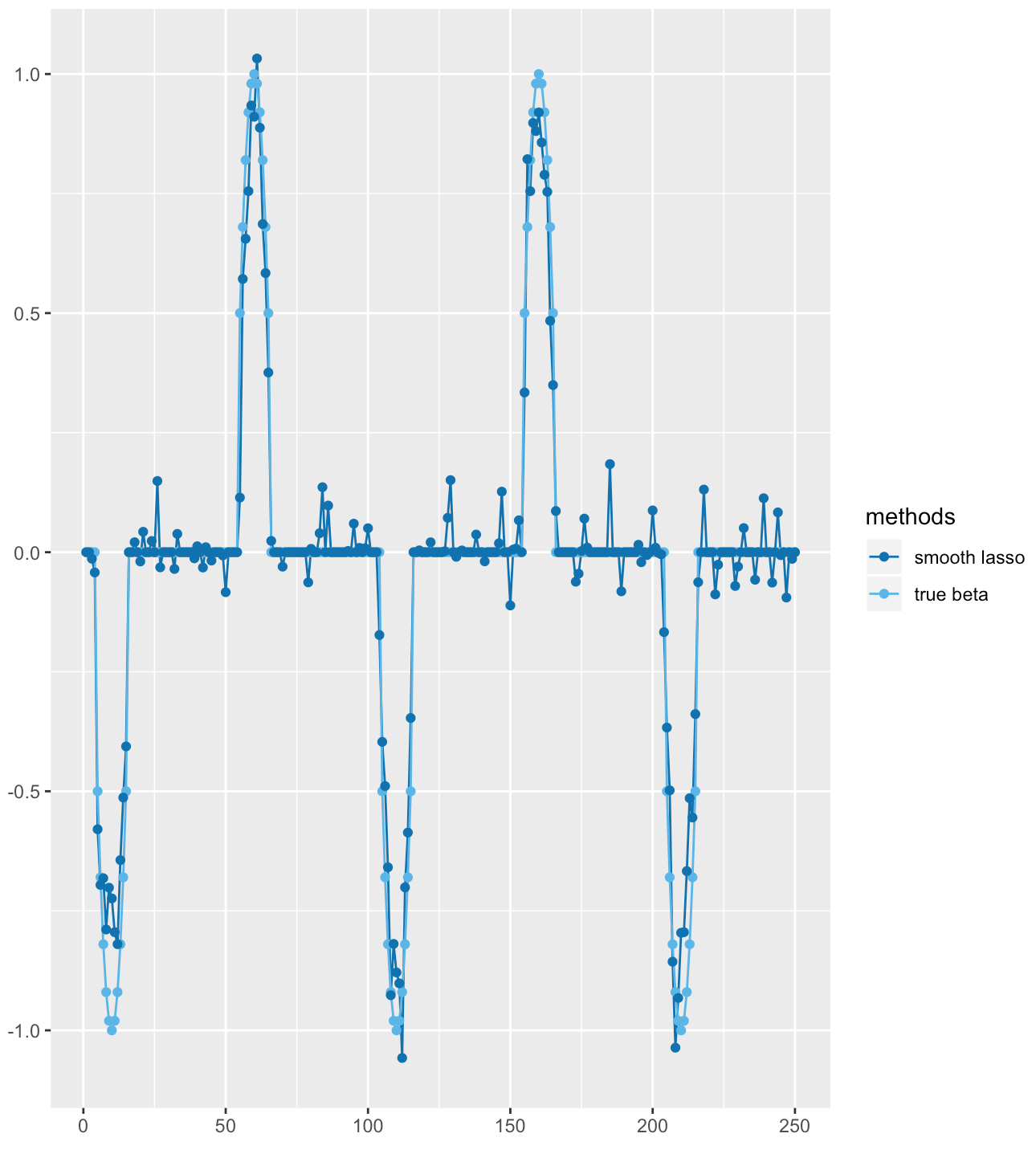}  
\end{subfigure}
\begin{subfigure}{.5\textwidth}
 \centering
  \includegraphics[width=.8\linewidth]{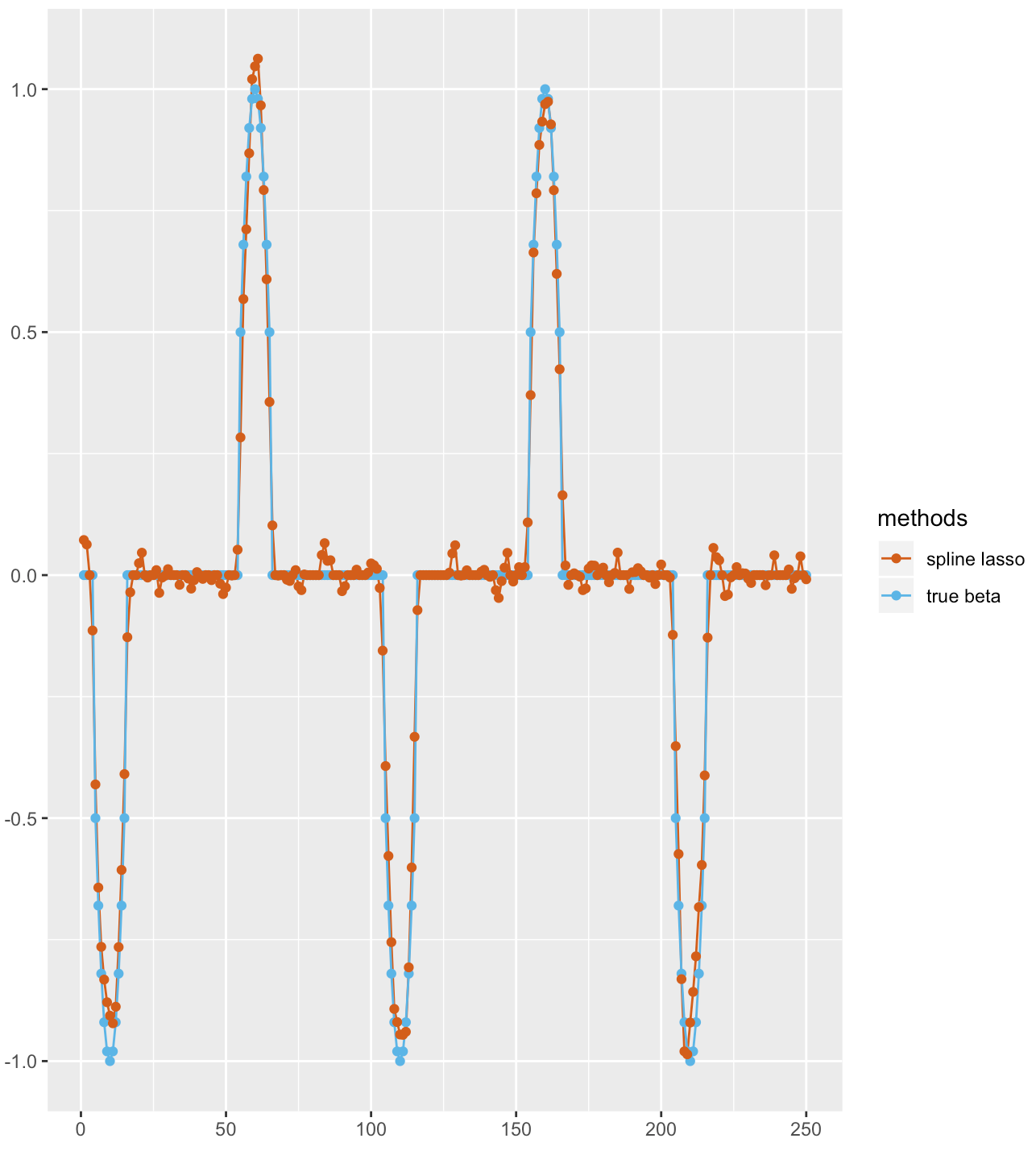}  
\end{subfigure}
\caption{We consider estimation of $\beta^*$ plotted in (c) of Figure~\ref{piecewisepolypathgraph}. We set $N=100$. The data generating process and the tuning parameter selection are the same with the simulation in Section~\ref{simu1section}. Each panel displays the true regression coefficients and the estimated values.}
\label{piecequadraticbetagraph}
\end{figure}

Next, we present two preliminary results, which justify adaptivity and validity of the Graph-Piecewise-Polynomial-Lasso, respectively. Let $\widehat{S}_{1}\in [m]$ be the support set of $\Delta^{(k+1)}\hat{\beta}$ and $\Delta^{(k+1)}_{-\widehat{S}_{1}}$ be the submatrix of $\Delta^{(k+1)}$ after removing the rows indexed by $\widehat{S}_{1}$. Our first result describes the basic structure of the Graph-Piecewise-Polynomial-Lasso via the null space of $\Delta^{(k+1)}_{-\widehat{S}_{1}}$: 

\begin{proposition}
\label{structure}

Assume, without loss of generality, that the graph $\mathcal{G}$ has a single connected component. For even $k$, let $\mathcal{G}_{-\widehat{S}_{1}}$ be the subgraph induced by removing the edges indexed by $\widehat{S}_{1}$. Let $C_{1},...,C_{j}$ be the connected components of the subgraph $\mathcal{G}_{-\widehat{S}_{1}}$. Then the null space of $\Delta_{-\widehat{S}_{1}}^{(k+1)}$ is 
\begin{equation*}
\label{evencase}
\text{span}(\mathbbm{1}_{n})+ \text{span}(\mathbbm{1}_{n})^{\perp}\cap\left(L^{\frac{k}{2}}+\mathbbm{1}_{n}\mathbbm{1}_{n}^T\right)^{-1}\text{span}(\mathbbm{1}_{C_{1}},...,\mathbbm{1}_{C_{j}} ), 
\end{equation*}
where $\mathbbm{1}_{n}=(1,...,1)^T\in \mathbb{R}^{n}$ and $\mathbbm{1}_{C_{i}}\in \mathbb{R}^{n}$ is the indicator vector over connected component $C_{i}$. 

Similarly, for odd $k$, the null space of $\Delta_{-\widehat{S}_{1}}^{(k+1)}$ is 
\begin{equation*}
\label{oddcase}
\text{span}(\mathbbm{1}_{n})+\text{span}(\mathbbm{1}_{n})^{\perp}\cap \left\{ u\in \mathbb{R}^{n}: u=(L^{\frac{k+1}{2}}+\mathbbm{1}_{n}\mathbbm{1}_{n}^T)^{-1}v,\quad v_{-\widehat{S}_{1}}=0  \right\}. 
\end{equation*}
\end{proposition}

The proof of Proposition~\ref{structure} is given in Appendix~\ref{structureproof}. 
 
\begin{remark}
Proposition~\ref{structure} provides a justification for adaptivity of the Graph-Piecewise-Polynomial-Lasso. For $k=0$, Proposition~\ref{structure} shows that $\hat{\beta}\in \text{span}\left(\mathbbm{1}_{C_{1}},...,\mathbbm{1}_{C_{j}}\right)$. Therefore, $\hat{\beta}$ is piecewise constant over connected components $C_{1},...,C_{j}$. Furthermore, let $\widehat{S}_{2}\in [n]$ denote the support set of $\hat{\beta}$. Then for $i\in [j]$, the index set $C_{i}$ is either in $\widehat{S}_{2}$ or in $\widehat{S}_{2}^{c}$. In general, for even $k$, Proposition~\ref{structure} implies that the structure of $\hat{\beta}$ is smoothed by multiplying $\text{span}(\mathbbm{1}_{C_{1}},...,\mathbbm{1}_{C_{j}})$ by $(L^{k/2}+\mathbbm{1}_{n}\mathbbm{1}_{n}^T)^{-1}$. For odd $k$, Proposition~\ref{structure} implies that the structure is based on the support set $\widehat{S}_{1}$ and the smoother $(L^{(k+1)/2}+\mathbbm{1}_{n}\mathbbm{1}_{n}^T)^{-1}$. 
\end{remark}

Our second preliminary result interprets validity of the Graph-Piecewise-Polynomial Lasso from a Bayesian perspective. Following \cite{park2008}, for a given $\lambda>0$ and $\lambda_{g}>0$, we consider the hierarchical model described below: 
\begin{equation}
\label{bayes1}
y\mid X, \beta, \sigma^2_{\varepsilon}\sim N(X\beta, \sigma^{2}_{\varepsilon}I_{N}),  
\end{equation}
\begin{equation}
\label{bayes2}
\beta\mid \tau_{1}^{2},..., \tau_{n}^2, \omega_{1}^2,..., \omega_{m}^{2}, \sigma_{\varepsilon}^2\sim N(0, \sigma_{\varepsilon}^2\Sigma_{\beta}), 
\end{equation}
\begin{equation}
\label{bayes3}
\begin{split}
\pi\left(\tau_{1}^2,..., \omega_{m}^{2}\right) \propto \left|\Sigma_{\beta}\right|^{\frac{1}{2}}\prod_{j=1}^{n}\left(\frac{1}{\sqrt{\tau^2_{j}}}\exp{\left(-\frac{\lambda^{2}_{1}\tau_{j}^{2}}{2}\right)}\right)\prod_{i=1}^{m}\left(\frac{1}{\sqrt{\omega^2_{i}}}\exp{\left(-\frac{\lambda^{2}_{2}\omega_{i}^{2}}{2}\right)}\right), 
\end{split}
\end{equation}
where $\left|\Sigma_{\beta}\right|$ is the determinant of $\Sigma_{\beta}$, $\pi(\tau_{1}^2,..., \omega_{m}^2)$ is the joint prior distribution for $(\tau_{1}^2,..., \omega_{m}^2)$, and 
\begin{equation*}
\lambda_{1}=\frac{N\lambda}{\sigma_{\varepsilon}},\quad\lambda_{2}=\frac{N\lambda_{g}}{\sigma_{\varepsilon}},\quad\Sigma_{\beta}^{-1}=\text{diag}\left(\frac{1}{\tau_{1}^{2}},..., \frac{1}{\tau_{n}^2}\right)+(\Delta^{(k+1)})^T\text{diag}\left(\frac{1}{\omega_{1}^{2}},...,\frac{1}{\omega_{m}^2}\right)\Delta^{(k+1)}.  
\end{equation*}
Then we have the following result: 
\begin{proposition} 
\label{bayes}
The proposed estimator $\hat{\beta}$ in (\ref{tf-sl}) is a maximum a posteriori (MAP) estimator for the hierarchical model in $(\ref{bayes1})$, $(\ref{bayes2})$, and $(\ref{bayes3})$. 
\end{proposition} 

The proof of Proposition~\ref{bayes} is contained in Appendix~\ref{bayesproof}.  

\begin{remark}
The conditional prior of $\beta$ in $(\ref{bayes2})$ is a graph-based prior. In other words, we construct the inverse covariance matrix based on $\Delta^{(k+1)}$. For example, if the underlying graph is a path graph and $k=0$, then 
\begin{equation*}
\Sigma_{\beta}^{-1}=\left[\begin{array}{ccccccc}{\frac{1}{\tau_{1}^{2}}+\frac{1}{\omega_{1}^{2}}} & {-\frac{1}{\omega_{1}^{2}}} & {0} & {\cdots} & {0} & {0} \\ {-\frac{1}{\omega_{1}^{2}}} & {\frac{1}{\tau_{2}^{2}}+\frac{1}{\omega_{1}^{2}}+\frac{1}{\omega_{2}^{2}}} & {-\frac{1}{\omega_{2}^2}} & {\cdots} & {0} & {0} \\ {\vdots} & {\vdots} & {\vdots} &  & {\vdots} & {\vdots}  \\ 
{0} & {0} & {0} & {\cdots} & {-\frac{1}{\omega_{n-1}^{2}}} & {\frac{1}{\tau_{n}^2}+\frac{1}{\omega_{n-1}^{2}}}
\end{array}\right], 
\end{equation*}
which is a tridiagonal matrix. On the other hand, it is a well-known fact that zeros in the inverse covariance matrix of a multivariate Gaussian distribution indicate absent edges in
the corresponding undirected graph \cite{graphical1996}. Therefore, the tridiagonal structure of $\Sigma_{\beta}^{-1}$ implies that the underlying graph which governs $\beta$ is a path graph and only adjacent components of $\beta$ are correlated with each other.   
\end{remark}


\subsection{Main results}
We now state the main theoretical results on $\ell_{2}$-estimation error, $\ell_{1}$-estimation error, and mean-squared prediction error of the Graph-Piecewise-Polynomial-Lasso for the deterministic design (Theorem~\ref{fixedconthm}) and random design (Theorem~\ref{probconthm}), respectively. 

We start with the fixed design. To obtain estimation error bounds, it is necessary to impose some conditions on the design matrix. It has been shown that the restricted eigenvalue condition is sufficient to bound the estimation error of the standard Lasso for linear models \cite{bickel2009, negahban2012}. In this paper, we require a similar notion of restricted eigenvalue condition.  

\begin{condition}
\label{lowerRE}
Let $S_{1}\subset [m]$ and $S_{2}\subset [n]$ be the support sets of $\Delta^{(k+1)}\beta^*$ and $\beta^*$, respectively. Let $\gamma=\lambda_{g}/\lambda$. There exists $\eta_{\gamma}>0$, such that
\begin{equation*}
\frac{1}{N}\|XD^{+}v\|_{2}^2\ge \eta_{\gamma}\|v\|_{2}^{2}, 
\end{equation*}
for all $v\in\mathbb{C}=\left\{v\in \real^{m+n}: \|v_{S^c}\|_{1}\le 3\|v_{S}\|_{1} \right\}$, where $S=S_{1}\cup \left\{i+m, i\in S_{2}\right\}\subset [m+n]$. 
\end{condition}

\begin{remark}
When $\gamma=0$, Condition~\ref{lowerRE} reduces to the usual restricted eigenvalue condition for the standard Lasso problem. When $\gamma\ne 0$, then both of $D$ and $D^{+}$ depend on the value of $\gamma$, so we allow the restricted eigenvalue $\eta_{\gamma}$ to be related to $\gamma$. In the sequel, we use $\eta_{0}$ to denote $\eta_{\gamma}$ for $\gamma=0$. 
\end{remark}

With Condition~\ref{lowerRE} at hand, we have the following main results for the fixed design: 


\begin{theorem}[Fixed design]
\label{fixedconthm}
Consider the linear model $(\ref{linear})$ where $\beta^*\in \mathcal{S}(k, s_{1}, s_{2})$. Assume Condition~\ref{lowerRE} holds and let $\lambda$ in $(\ref{tf-sl})$ satisfy the condition that  
\begin{equation*}
\lambda\ge \frac{2}{N}\|\varepsilon^TXD^{+}\|_{\infty}. 
\end{equation*}

\begin{enumerate}
\item[(a)] We have 
\begin{equation}
\label{l2error}
\|\hat{\beta}-\beta^*\|_{2}\le \frac{3\lambda_{g}\sqrt{(2d)^{k+1}s_{1}}+3\lambda\sqrt{s_{2}}}{\eta_{\gamma}},
\end{equation}
and 
\begin{equation*}
\begin{split}
\frac{1}{N}\|X(\hat{\beta}-\beta^*)\|_{2}^{2}\le \frac{\left( 3\lambda_{g}\sqrt{(2d)^{k+1}s_{1}}+3\lambda\sqrt{s_{2}}  \right)^{2}}{\eta_{\gamma}}, 
\end{split}
\end{equation*}
where $d$ is the maximum degree of the underlying graph. 

\item[(b)] Furthermore, if $\gamma^2<1/(2d)^{k+1}$, where $\gamma$ is defined in Condition~\ref{lowerRE}, then we have
\begin{equation}
\label{l1error}
\begin{split}
\|\hat{\beta}-\beta^*\|_{1}\le \frac{12\lambda\left(  
\gamma\sqrt{(2d)^{k+1}s_{1}}+\sqrt{s_{2}}\right)^{2}}{\eta_{\gamma}\left(1-\gamma^2(2d)^{k+1}\right)}. 
\end{split}
\end{equation}



\end{enumerate}

\end{theorem}

The proof of Theorem~\ref{fixedconthm} is contained in Appendix~\ref{fixedconthmprf}. 

\begin{remark}
\label{fixedremark}
We provide some comments about the results in Theorem~\ref{fixedconthm} below: 
\begin{enumerate}
\item[(a)] Theorem~\ref{fixedconthm} suggests that the key ingredients for statistical consistency of the Graph-Piecewise-Polynomial-Lasso include the existence of $\eta_{\gamma}$ in Condition~\ref{lowerRE} and the appropriate choices of tuning parameters $\lambda$ and $\lambda_{g}$, both of which appear explicitly in the upper bounds in $(\ref{l2error})$ and $(\ref{l1error})$. In the sequel, we will discuss their choices for sub-Gaussian random designs. 

\item[(b)] If $\gamma=0$, i.e., $\lambda_{g}=0$, then Part (a) and Part (b) of the theorem recover the following standard results for the Lasso: 
\begin{equation*}
\label{lasso1}
\|\hat{\beta}-\beta^*\|_{2}\le \frac{3\lambda\sqrt{s_{2}}}{\eta_{0}},\quad\|\hat{\beta}-\beta^*\|_{1}\le \frac{12\lambda s_{2}}{\eta_{0}},\quad\frac{1}{N}\|X(\hat{\beta}-\beta^*)\|_{2}^{2}\le \frac{9\lambda^2 s_{2}}{\eta_{0}}. 
\end{equation*}
Therefore, in the random design, if we choose $\lambda\asymp \sigma_{\varepsilon}\sqrt{\frac{\log n}{N}}$, we obtain 
\begin{equation*}
\label{lasso1}
\|\hat{\beta}-\beta^*\|_{2}\le \mathcal{O}_{\mathbb{P}}\left(\sigma_{\varepsilon}\sqrt{\frac{s_{2}\log n}{N}}\right),\quad\|\hat{\beta}-\beta^*\|_{1}\le \mathcal{O}_{\mathbb{P}}\left(\sigma_{\varepsilon}s_{2}\sqrt{\frac{\log n}{N}}\right),  
\end{equation*}
and 
\begin{equation*}
\label{lasso2}
\begin{split}
\frac{1}{N}\|X(\hat{\beta}-\beta^*)\|_{2}^{2}\le \mathcal{O}_{\mathbb{P}}\left(\sigma_{\varepsilon}^2\frac{s_{2}\log n}{N}\right).  
\end{split}
\end{equation*}
In Theorem~\ref{probconthm} to follow, we will show that when $\gamma\ne 0$, the Graph-Piecewise-Polynomial-Lasso is able to achieve the same rate. 


\end{enumerate}
\end{remark}

Next, we turn to the random design setting and consider the situation where $N$, $n$, $p$, $s_{1}$, and $s_{2}$ are able to increase to $\infty$. We make use of the following assumption on the design matrix $X$: 

\begin{assumption}
\label{subgaussianmat}
The design matrix $X\in \mathbb{R}^{N\times n}$ in the linear model $(\ref{linear})$ is a row-wise $(\sigma_{x}, \Sigma_{x})$-sub-Gaussian random matrix defined in Definition~\ref{randmat}. Furthermore, eigenvalues of $\Sigma_{x}$ are bounded by dimension-free constants.\footnote{For ease of presentation, we only consider designs which satisfy the bounded eigenvalue condition in this paper. The proposed adaptive estimator and our theory could also be easily adapted to the setting of highly-correlated designs, but such derivations are beyond the scope of our present work.}  That is, 
\begin{equation*}
c\le \lambda_{1}(\Sigma_{x})\le \lambda_{n}(\Sigma_{x})\le c^{\prime},
\end{equation*}
where $c>0$ and $c^{\prime}>0$ are constants. 
\end{assumption}

We begin by verifying the restricted eigenvalue condition stated in Condition~\ref{lowerRE}. 

\begin{lemma}
\label{eta1rand}
Assume that Assumption~\ref{subgaussianmat} holds. Let $\gamma^2=\nu/(2d)^{k+1}$, where $\gamma$ is defined in Condition~\ref{lowerRE} and $0\le \nu<1$. Then when 
$\eta_{\gamma}=\frac{1}{2}(\nu+1)^{-1}\lambda_{1}(\Sigma_{x})$, for any $v\in \mathbb{C}$, 
\begin{equation*}
\frac{1}{N}\|XD^{+}v\|_{2}^2\ge \eta_{\gamma}\|v\|_{2}^2, 
\end{equation*}
with probability at least $1-2\exp{\left(c\left(s_{1}+s_{2}\right)\log n-c^{\prime}N\right)}$, where $c>0$ and $c^{\prime}>0$ are constants. 

\end{lemma}

The proof of Lemma~\ref{eta1rand} is contained in Appendix~\ref{eta1randprf}. Our next lemma concerns the choice of the tuning parameter $\lambda$. We have the following result: 

\begin{lemma}
\label{tuninglem}
Assume that Assumption~\ref{subgaussianmat} holds and let $\lambda\asymp\sigma_{\varepsilon}\sqrt{\frac{\log n}{N}}$. Then we have
\begin{equation*}
\mathbb{P}\left(\lambda\ge \frac{2}{N}\|\varepsilon^TXD^{+}\|_{\infty}\right)\ge 1-2\exp{(-\log n)}-\exp{\left(c\log n-c^{\prime}N\right)},
\end{equation*}
where $c>0$ and $c^{\prime}>0$ are constants. 
\end{lemma}

The proof of Lemma~\ref{tuninglem} is contained in Appendix~\ref{tuninglemprf}. Altogether, we arrive at the main result for sub-Gaussian random designs: 

\begin{theorem}[Random design]
\label{probconthm}
Consider the linear model $(\ref{linear})$ where $\beta^*\in \mathcal{S}(k, s_{1}, s_{2})$. Assume that Assumption~\ref{subgaussianmat} holds. Let $\lambda\asymp \sigma_{\varepsilon}\sqrt{\frac{\log n}{N}}$ and $\lambda_{g}=\lambda\sqrt{\nu/(2d)^{k+1}}$, where $0\le \nu<1$ is a constant and $d$ is the maximum degree of the underlying graph. Assume $s_{2}/s_{1}\ge\nu$. Then we have 
\begin{equation*}
\|\hat{\beta}-\beta^*\|_{2}\le \mathcal{O}_{\mathbb{P}}\left(\sigma_{\varepsilon}\sqrt{\frac{s_{2}\log n}{N}}\right),\quad\|\hat{\beta}-\beta^*\|_{1}\le \mathcal{O}_{\mathbb{P}}\left(\sigma_{\varepsilon}s_{2}\sqrt{\frac{\log n}{N}}\right), 
\end{equation*}
and 
\begin{equation*}
\frac{1}{N}\|X(\hat{\beta}-\beta^*)\|^{2}_{2}\le \mathcal{O}_{\mathbb{P}}\left(\sigma^{2}_{\varepsilon}\frac{s_{2}\log n}{N}\right). 
\end{equation*}





\end{theorem}

The proof of Theorem~\ref{probconthm} is contained in Appendix~\ref{probconthmprf}. 
\begin{remark}
\begin{enumerate}
\item[(a)] Theorem~\ref{probconthm} implies that if $s_{2}\log n/N=o(1)$, then the Graph-Piecewise-Polynomial-Lasso is $\ell_{2}$-consistent. 

\item[(b)] The convergence rates in Theorem~\ref{probconthm} are the same as those of the Lasso in Remark~\ref{fixedremark}. However, as empirically demonstrated in Figure~\ref{piececonstantbetagraph}, Figure~\ref{piecelinearbetagraph} and Figure~\ref{piecequadraticbetagraph}, our approach is adaptive in the sense that it is able to promote the desired graph-based structure while the Lasso is not.  
\end{enumerate}
\end{remark}

\subsection{Some extensions}
\label{extension-section}

We now extend our theory to the case where the true regression parameter $\beta^*$ is weakly piecewise polynomial and sparse, which is defined as follows:   

\begin{definition}
\label{weaklypiecewisepoly}
For $k\ge 0$, $0<q_{1}<1$, $0<q_{2}<1$, $R_{1}>0$, and $R_{2}>0$, let 
\begin{equation*}
\mathcal{S}\left(k, q_{1}, q_{2}, R_{1}, R_{2}\right)=\left\{ \beta\in \mathbb{R}^{n}: \sum_{i=1}^{m}\left|(\Delta_{i}^{(k+1)})^T\beta\right|^{q_{1}}\le R_{1},\quad\sum_{i=1}^{n}\left|\beta_{i}\right|^{q_{2}}\le R_{2}\right\}, 
\end{equation*}
where $(\Delta_{i}^{(k+1)})^T$ is the $i$-th row of $\Delta^{(k+1)}$. Then $\beta^*$ is called \emph{simultaneously $(k, q_{1}, R_{1})$-weakly piecewise polynomial and $(q_{2}, R_{2})$-weakly sparse} over the underlying graph $\mathcal{G}$ if $\beta^*\in \mathcal{S}(k, q_{1}, q_{2}, R_{1}, R_{2})$. 
\end{definition}

Obviously, the notion of weakly piecewise polynomial and sparse structure is a generalization of our previously defined piecewise polynomial and sparse structure. For a given $q\in (0, 1)$, recall that the $\ell_{q}$-ball is defined as 
\begin{equation*}
\mathbb{B}_{q}(R_{q})=\left\{ \theta\in \mathbb{R}^{m}: \sum_{i=1}^{m}\left|\theta_{i}\right|^{q}\le R_{q}  \right\}. 
\end{equation*}
Therefore, $\Delta^{(k+1)}\beta^*\in\mathbb{B}_{q_{1}}(R_{1})$ and $\beta^*\in\mathbb{B}_{q_{2}}(R_{2})$ if $\beta^*$ is simultaneously $(k, q_{1}, R_{1})$-weakly piecewise polynomial and $(q_{2}, R_{2})$-weakly sparse. Furthermore, it can be shown that if 
\begin{equation*}
\left|\Delta^{(k+1)}\beta^*\right|_{(i)}\le ci^{-\alpha},\quad\left|\beta^*\right|_{(j)}\le cj^{-\alpha},\quad\forall~i\in[m], j\in [n],  
\end{equation*}
where $c>0$ and $\alpha>1$ are constants, and $\left|\Delta^{(k+1)}\beta^*\right|_{(i)}$ and $\left|\beta^*\right|_{(j)}$ are the order statistics of $\Delta^{(k+1)}\beta^*$ and $\beta^*$ in absolute value ordered from largest to smallest, then 
\begin{equation*}
\Delta^{(k+1)}\beta^*\in \mathbb{B}_{q_{1}}(R_1),\quad\beta^*\in \mathbb{B}_{q_{2}}(R_{2}),
\end{equation*}
where
\begin{equation*}
\frac{1}{\alpha}<q_{1}<1,\quad\frac{1}{\alpha}<q_{2}<1,\quad R_{1}=\frac{c^{q_{1}}\alpha q_{1}}{\alpha q_{1}-1},\quad R_{2}=\frac{c^{q_{2}}\alpha q_{2}}{\alpha q_{2}-1}. 
\end{equation*}

Again we start with the deterministic design, and make use of the following condition.  
\begin{condition}
\label{lowerRE2}
There exist a curvature $\eta^{\prime}_{\gamma}>0$ and tolerance $\tau(N, n)>0$ such that
\begin{equation*}
\frac{1}{N}\|XD^{+}v\|_{2}^{2}\ge \eta^{\prime}_{\gamma}\|v\|_{2}^2-\tau(N, n)\|v\|_{1}^{2}, 
\end{equation*}
for all $v\in \mathbb{R}^{m+n}$. 
\end{condition}

Condition~\ref{lowerRE2} is a generalization of our previous Condition~\ref{lowerRE} to any vector in $\mathbb{R}^{m+n}$. It is similar to the lower restricted eigenvalue condition defined in \cite{LohWai12}. We now state an extended result of Theorem~\ref{fixedconthm}.  

\begin{lemma}
\label{modelmisspecified-fixed} 

Consider the linear model $(\ref{linear})$ where $\beta^*$ is any general vector in $\mathbb{R}^{n}$. Assume that Condition~\ref{lowerRE2} holds. Let $\lambda$ in $(\ref{tf-sl})$ satisfy the condition $\lambda\ge \frac{2}{N}\|\varepsilon^TXD^{+}\|_{\infty}$ and let $S\subset [m+n]$ be any subset with cardinality $|S|\le \frac{\eta^{\prime}_{\gamma}}{64\tau(N, n)}$. 

\begin{enumerate}
\item[(a)] We have 
\begin{equation*}
\|\hat{\beta}-\beta^*\|_{2}^{2}\le \frac{144\lambda^2|S|}{(\eta^{\prime}_{\gamma})^2}+\frac{32\lambda}{\eta^{\prime}_{\gamma}}\|D_{S^c}\beta^*\|_{1}+\frac{128\tau(N, n)}{\eta^{\prime}_{\gamma}}\|D_{S^c}\beta^*\|_{1}^{2}, 
\end{equation*}
and 
\begin{equation*}
\begin{split}
\frac{1}{N}\|X(\hat{\beta}-\beta^*)\|_{2}^{2}
\le 4\lambda\|D_{S^c}\beta^*\|_{1}+3\lambda\sqrt{|S| \left(\frac{144\lambda^2|S|}{(\eta^{\prime}_{\gamma})^2}+\frac{32\lambda}{\eta^{\prime}_{\gamma}}\|D_{S^c}\beta^*\|_{1}+\frac{128\tau(N, n)}{\eta^{\prime}_{\gamma}}\|D_{S^c}\beta^*\|_{1}^{2}\right)}. 
\end{split}
\end{equation*}

\item[(b)] If $\lambda_{g}=2^{-(1+k/2)}d^{-(k+1)/2}\lambda$, then we have 
\begin{equation*}
\|\hat{\beta}-\beta^*\|^{2}_{1}\le128\|D_{S^c}\beta^*\|_{1}^{2}+128|S|\left(\frac{144\lambda^2|S|}{(\eta^{\prime}_{\gamma})^2}+\frac{32\lambda}{\eta^{\prime}_{\gamma}}\|D_{S^c}\beta^*\|_{1}+\frac{128\tau(N, n)}{\eta^{\prime}_{\gamma}}\|D_{S^c}\beta^*\|_{1}^{2}\right). 
\end{equation*}

\end{enumerate}

%
%
%

\end{lemma}

The proof of Lemma~\ref{modelmisspecified-fixed} is contained in Appendix~\ref{modelmisspecified-fixed-proof}. 

\begin{remark}
The results in Lemma~\ref{modelmisspecified-fixed} are oracle inequalities, which hold without any assumptions on the true regression vector $\beta^{*}$. Furthermore, Lemma~\ref{modelmisspecified-fixed} yields a family of upper bounds with a tunable subset $S$ to be optimized. 
\end{remark}

Applying Lemma~\ref{modelmisspecified-fixed} to weakly piecewise polynomial and sparse regression coefficients, we obtain the following theorem in the fixed design case:  

\begin{theorem}[Fixed design]
\label{weaklypiecewisepolycorollary}

Consider the linear model $(\ref{linear})$ where $\beta^*\in \mathcal{S}(k, q_{1}, q_{2}, R_{1}, R_{2})$. Assume that Condition~\ref{lowerRE2} holds and let $\lambda$ in $(\ref{tf-sl})$ satisfy the condition $\lambda\ge \frac{2}{N}\| \varepsilon^TXD^{+}\|_{\infty}$. Furthermore, assume there exist  constants $\eta_{1}>0$ and $\eta_{2}>0$ such that 
\begin{equation*}
R_{1}\eta_{1}^{-q_{1}}+R_{2}\eta_{2}^{-q_{2}}\le \frac{\eta^{\prime}_{\gamma}}{64\tau(N, n)}. 
\end{equation*}

\begin{enumerate}
\item[(a)] We have 
\begin{equation*}
\|\hat{\beta}-\beta^*\|_{2}^{2}\le \RNum{1}+\RNum{2},
\end{equation*}
and 
\begin{equation*}
\frac{1}{N}\|X(\hat{\beta}-\beta^*)\|_{2}^{2}\le 4\left(\lambda_{g}R_{1}\eta_{1}^{1-q_{1}}+\lambda R_{2}\eta_{2}^{1-q_{2}}\right)+3\lambda\sqrt{\left(R_{1}\eta_{1}^{-q_{1}}+R_{2}\eta_{2}^{-q_{2}}\right)\left(\RNum{1}+\RNum{2}\right)},  
\end{equation*}
where
\begin{equation*}
\begin{split}
\RNum{1}=\frac{144\lambda^2\left(R_{1}\eta_{1}^{-q_{1}}+R_{2}\eta_{2}^{-q_{2}}\right)}{(\eta^{\prime}_{\gamma})^2},
\end{split}
\end{equation*}
and 
\begin{equation*}
\begin{split}
\RNum{2}=\frac{32}{\eta^{\prime}_{\gamma}}\left[\lambda_{g}R_{1}\eta_{1}^{1-q_{1}}+\lambda R_{2}\eta_{2}^{1-q_{2}}+4\tau(N, n)\left(\frac{\lambda_{g}}{\lambda}R_{1}\eta_{1}^{1-q_{1}}
+R_{2}\eta_{2}^{1-q_{2}} \right)^2\right].
\end{split}
\end{equation*}

\item[(b)] If $\lambda_{g}=2^{-(1+k/2)}d^{-(k+1)/2}\lambda$, then we have
\begin{equation*}
\begin{split}
\|\hat{\beta}-\beta^*\|^{2}_{1}\le 128\left(\frac{\lambda_{g}}{\lambda}R_{1}\eta_{1}^{1-q_{1}}+R_{2}\eta_{2}^{1-q_{2}}\right)^2+128\left(R_{1}\eta_{1}^{-q_{1}}+R_{2}\eta_{2}^{-q_{2}}\right)\left(\RNum{1}+\RNum{2}\right),
\end{split}
\end{equation*}
where $\RNum{1}$ and $\RNum{2}$ are defined in Part (a).

\end{enumerate}
\end{theorem}

The proof of Theorem~\ref{weaklypiecewisepolycorollary} is contained in Appendix~\ref{weaklypiecewisepolycorollary-proof}. Next, we consider the random design setting. In the following lemma, we confirm that  Condition~\ref{lowerRE2} holds with appropriate choices of $\eta^{\prime}_{\gamma}$ and $\tau(N, n)$, with high probability. 

\begin{lemma}
\label{lowerRE2random}
Assume that Assumption~\ref{subgaussianmat} holds and let $\lambda_{g}=2^{-(1+k/2)}d^{-(k+1)/2}\lambda$. Then we have 
\begin{equation*}
\frac{1}{N}\left\|XD^{+}v\right\|_{2}^{2}\ge \frac{1}{3}\lambda_{1}(\Sigma_{x})\left\|v\right\|_{2}^{2}-\frac{c\log n}{N}\left\|v\right\|_{1}^{2}\quad\forall~v\in \mathbb{R}^{m+n}, 
\end{equation*}
with probability at least $1-2\exp{(-c^{\prime}N)}$, where $c>0$ and $c^{\prime}>0$ are constants. 


\end{lemma}

The proof of Lemma~\ref{lowerRE2random} is contained in Appendix~\ref{lowerRE2random-proof}. Altogether, we obtain the following probabilistic consequence of Theorem~\ref{weaklypiecewisepolycorollary} in the random design: 

\begin{theorem}[Random design]
\label{weaklypiecewisepolyrandomthm}

Consider the linear model $(\ref{linear})$ where $\beta^*\in \mathcal{S}(k, q_{1}, q_{2}, R_{1}, R_{2})$. Assume that Assumption~\ref{subgaussianmat} holds. Let $\lambda\asymp \sigma_{\varepsilon}\sqrt{\frac{\log n}{N}}$ and $\lambda_{g}=2^{-(1+k/2)}d^{-(k+1)/2}\lambda$. Furthermore, assume that   
\begin{equation*}
R_{1}\left( \frac{\log n}{N}  \right)^{1-\frac{q_{1}}{2}}+R_{2}\left( \frac{\log n}{N} \right)^{1-\frac{q_{2}}{2}}\lesssim 1. 
\end{equation*}

\begin{enumerate}
\item[(a)] We have 
\begin{equation*}
\|\hat{\beta}-\beta^*\|_{2}^{2}\le \mathcal{O}_{\mathbb{P}}\left( \sigma_{\varepsilon}^{2}\left( R_{1}\left(\frac{\log n}{N}  \right)^{1-\frac{q_1}{2}}+R_{2}\left( \frac{\log n}{N} \right)^{1-\frac{q_{2}}{2}}  \right)  \right),
\end{equation*}
and 
\begin{equation*}
\frac{1}{N}\|X(\hat{\beta}-\beta^*)\|_{2}^{2}\le  \mathcal{O}_{\mathbb{P}}\left( \sigma_{\varepsilon}^{2}\left( R_{1}\left(\frac{\log n}{N}  \right)^{1-\frac{q_1}{2}}+R_{2}\left( \frac{\log n}{N} \right)^{1-\frac{q_{2}}{2}}  \right)  \right). 
\end{equation*}

\item[(b)] We have  
\begin{equation*}
\begin{split}
\|\hat{\beta}-\beta^*\|_{1}^{2}\le \mathcal{O}_{\mathbb{P}}\left(\sigma^2_{\varepsilon}\left(R_{1}^{2}\left( \frac{\log n}{N} \right)^{1-q_{1}}+R_{2}^2\left(\frac{\log n}{N}\right)^{1-q_{2}}\right)\right).  
\end{split}
\end{equation*}

\end{enumerate}


%
%

\end{theorem}

The proof of Theorem~\ref{weaklypiecewisepolyrandomthm} is contained in Appendix~\ref{weaklypiecewisepolyrandomthmprf}. It is clear that if $q_{1}=q_{2}=0$ and $R_{1}<R_{2}$, then Theorem~\ref{weaklypiecewisepolyrandomthm} recovers the previous results in Theorem~\ref{probconthm}. 


\section{Statistical inference}
\label{inferencesection}

As we can see from the optimization problem $(\ref{tf-sl})$, the Graph-Piecewise-Polynomial-Lasso is non-linear and non-explicit given the finite sample size. Hence, it is generally difficult to derive its exact distribution. Furthermore, from an asymptotic viewpoint, it is a well-known fact that estimators with $\ell_{1}$-type regularization do not have a uniform tractable limiting distribution \cite{knight2000}. Therefore, it is challenging to directly use the Graph-Piecewise-Polynomial-Lasso for the task of statistical inference. To tackle these issues, recent work in \cite{vandegeer2014, javanmard14a, zhang2014confidence, loh2018scale} recommends one-step modifications of Lasso-type estimators via the de-biasing procedure. In this section, we propose a one-step update of the Graph-Piecewise-Polynomial-Lasso.

\subsection{One-step estimators}
\label{methodsec}

We begin by briefly introducing the so-called one-step maximum likelihood estimator (MLE) in classical low-dimensional statistics. For a more detailed overview, we refer the reader to the textbooks by Bickel et al.\ \cite{bickel1993} or Shao \cite{shao2003}. The one-step MLE, which is used to approximate the MLE, is the first Newton iteration with a certain type of consistent estimator as the initial value. More specifically, let $S_{N}(\beta)$ be the score function and $\hat{\beta}_{0}$ be an estimator of $\beta^*$, and define the one-step MLE by 
\begin{equation}
\label{onestep1}
\hat{\beta}_{1}=\hat{\beta}_{0}-[\nabla S_{N}(\hat{\beta}_{0})]^{-1}S_{N}(\hat{\beta}_{0}). 
\end{equation}
It has been shown that $\hat{\beta}_{1}$ is asymptotically efficient under some regularity conditions. In this section, for ease of presentation, we focus on the Gaussian error in $(\ref{linear})$. That is, we assume $\varepsilon\sim N(0, \sigma_{\varepsilon}^2I_{N})$. If we consider the fixed design in the low-dimensional regime (i.e.\ $n<N$), then $(\ref{onestep1})$ becomes 
\begin{equation}
\label{onestep2}
\hat{\beta}_{1}=\hat{\beta}_{0}+(\Sigma_{N})^{-1}\frac{1}{N}X^T(y-X\hat{\beta}_{0}),
\end{equation}
where $\Sigma_{N}=\frac{1}{N}X^TX\in \mathbb{R}^{n\times n}$. However, $\Sigma_{N}$ is singular in the high-dimensional regime where $N\ll n$. Hence, we replace $\Sigma_{N}^{-1}$ in $(\ref{onestep2})$ by $\widehat{\Theta}$, a ``sparse approximate inverse" of $\Sigma_{N}$ via the CLIME estimator proposed in \cite{clime}, to be described in the sequel. We choose the Graph-Piecewise-Polynomial-Lasso as the initial value, leading to the one-step estimator
\begin{equation}
\label{onestep3}
\tilde{\beta}=\hat{\beta}+\frac{1}{N}\widehat{\Theta}X^T(y-X\hat{\beta}). 
\end{equation}

Next, we introduce the CLIME approach to obtain $\widehat{\Theta}$. Cai et al.\ \cite{clime} originally designed this method to estimate a row-wise weakly sparse precision matrix with constrained $\ell_{1}$-minimization. More specifically, we define the CLIME estimator as the solution of the following optimization problem: 
\begin{equation}
\label{CLIME}
\begin{aligned}
&  \widehat{\Theta}=\underset{}{\text{argmin} }
& & \|\Theta\|_{1,1} \\
& \text{subject to}
& &  \|\Theta\Sigma_{N}-I_{n}\|_{\infty}\le \mu,\quad\Theta\in \mathbb{R}^{n\times n}, 
\end{aligned}
\end{equation}
where $\mu>0$ is a tuning parameter. Note that $(\ref{CLIME})$ can be further decomposed into $n$ row-wise vector minimization problems. That is, if $\widehat{\Theta}=(\hat{\theta}_{1},...,\hat{\theta}_{n})^T$, we can obtain $\hat{\theta}_{i}$ via the following optimization:  
\begin{equation}
\label{row-wise-opt}
\begin{aligned}
&\hat{\theta}_{i}=\underset{}{\text{argmin} }
& & \|\theta\|_{1} \\
& \text{subject to}
& &  \|\Sigma_{N}\theta-e_{i}\|_{\infty}\le \mu,\quad\theta\in \mathbb{R}^{n}, 
\end{aligned}
\end{equation}
where $e_{i}$ is the i-th column of the identity matrix $I_{n}$. 

\begin{remark}
\begin{enumerate}
\item[(a)] We need to select an appropriate choice of tuning parameter $\mu$, which will be discussed in Lemma~\ref{tuningmulemma} of the next section. 

\item[(b)] Recent work has developed various alternatives to construct $\widehat{\Theta}$ in one-step estimators. For example, van de Geer et al.\ \cite{vandegeer2014} used the Lasso for nodewise regression and Loh \cite{loh2018scale} used the graphical Lasso. The key idea in these methods is to view $\widehat{\Theta}$ as an estimator of the inverse covariance matrix of the covariates in the random design. We will make a comparison of requirements on the inverse covariance matrix for these different methods in the next section. 

\item[(c)] Our approach to obtain $\widehat{\Theta}$ is close but not identical to the one proposed in \cite{javanmard14a}. Both approaches share the same constraint, but have different objectives. Instead, \cite{javanmard14a} minimizes $\theta^T\Sigma_{N}\theta$. These two different approaches lead to the same asymptotic properties of the one-step estimator. However, using our objective function proposed in $(\ref{CLIME})$ is beneficial to derive the non-asymptotic rate of convergence for $\widehat{\Theta}\Sigma_{N}(\widehat{\Theta})^T$, which can be seen in Theorem~\ref{onesteprandomthm}.     


\end{enumerate}
\end{remark}


\subsection{Main results}
\label{onestepresults}

We now present theoretical properties of the one-step estimator obtained from $(\ref{onestep3})$. Our first result in the following theorem concerns the fixed design and provides a useful decomposition of $\sqrt{N}(\tilde{\beta}-\beta^*)$, which is similar to the results in \cite{javanmard14a}.  

\begin{theorem}[Fixed design]
\label{onestepfixthm}
Consider the linear model $(\ref{linear})$ where $\beta^*\in \mathcal{S}(k, s_{1}, s_{2})$. Then we have $\sqrt{N}(\tilde{\beta}-\beta^*)=\Psi-e$,
where
\begin{equation*}
\Psi=\frac{1}{\sqrt{N}}\widehat{\Theta}X^T\varepsilon\sim N(0, \sigma^{2}_{\varepsilon}\widehat{\Theta}\Sigma_{N}(\widehat{\Theta})^T), \quad e=\sqrt{N}(\widehat{\Theta}\Sigma_{N}-I_{n})(\hat{\beta}-\beta^*).
\end{equation*}
Furthermore, we have $\|e\|_{\infty}\le \sqrt{N}\mu\|\hat{\beta}-\beta^*\|_{1}$.



\end{theorem}

The proof of Theorem~\ref{onestepfixthm} is contained in Appendix~\ref{onestepfixthmprf}. 

\begin{remark}
As shown in Theorem~\ref{onestepfixthm}, if $e$ is negligible, then $\sqrt{N}(\tilde{\beta}-\beta^*)$ is asymptotically normal. Furthermore, under the same conditions as Theorem~\ref{fixedconthm}, Theorem~\ref{onestepfixthm} implies    
\begin{equation*}
\|e\|_{\infty}\le \frac{12\sqrt{N}\lambda\mu\left(  
\gamma\sqrt{(2d)^{k+1}s_{1}}+\sqrt{s_{2}}\right)^{2}}{\eta_{\gamma}\left(1-\gamma^2(2d)^{k+1}\right)}. 
\end{equation*}
\end{remark}

In order to derive the limiting distribution of the one-step estimator, we now turn to the asymptotic framework with the sub-Gaussian random design, and assume the following: 
\begin{assumption}
\label{assumption2}
For the $(\sigma_{x}, \Sigma_{x})$-sub-Gaussian design in Assumption~\ref{subgaussianmat}, let $\Theta_{x}\in \mathbb{R}^{n\times n}$ denote the inverse of $\Sigma_{x}$. We assume $\opnorm{\Theta_{x}}_{\infty} \leq M_{n}$, where $M_{n}$ is allowed to grow as $n$ grows.  
\end{assumption}
Similar assumptions are often used in the literature of covariance matrix and precision matrix estimation \cite{bickel2008, clime, yuan2010}. We do not require sparsity of $\Theta_{x}$, but both \cite{vandegeer2014} and \cite{loh2018scale} assume row-wise sparsity of $\Theta_{x}$. In addition to the sparsity condition, \cite{loh2018scale} also needs to assume the $\alpha$-incoherence condition. Next, we consider the proper choice of the tuning parameter $\mu$ in $(\ref{CLIME})$. We have the following result: 

\begin{lemma}
\label{tuningmulemma}
Assume that Assumption~\ref{subgaussianmat} holds. For $\Sigma_{N}=\frac{1}{N}X^TX$ and $\Theta_{x}$ defined in Assumption~\ref{assumption2}, when $N\gtrsim \log n$, we have
\begin{equation*}
\left\|\Theta_{x}\Sigma_{N}-I_{n}\right\|_{\infty}\le c\sqrt{\frac{\log n}{N}}, 
\end{equation*}
with probability at least $1-2\exp{(-c^{\prime}\log n)}$, where $c>0$ and $c^{\prime}>0$ are constants.

\end{lemma}

The proof of Lemma~\ref{tuningmulemma} is contained in Appendix~\ref{tuningmulemmaprf}. Lemma~\ref{tuningmulemma} shows if $\mu\asymp\sqrt{\frac{\log n}{N}}$, then $\Theta_{x}$ is feasible for the constraint in $(\ref{CLIME})$ with high probability. Altogether, we arrive at the following main results, which present the limiting distribution of the one-step estimator in the sub-Gaussian random design:  

\begin{theorem}[Random design]
\label{onesteprandomthm}
Consider the linear model (\ref{linear}) where $\beta^*\in \mathcal{S}(k, s_{1}, s_{2})$. Assume that Assumption~\ref{subgaussianmat} and Assumption~\ref{assumption2} hold.  Let $\lambda\asymp\sigma_{\varepsilon}\sqrt{\frac{\log n}{N}}$, $\lambda_{g}=\lambda\sqrt{\nu/(2d)^{k+1}}$ where $0\le \nu<1$, and 
$\mu\asymp\sqrt{\frac{\log n}{N}}$. Assume $s_{2}/s_{1}\ge \nu$. Then we have $\sqrt{N}(\tilde{\beta}-\beta^*)=\Psi-e$, 
where 
\begin{equation*}
\Psi|X\sim N(0, \sigma_{\varepsilon}^{2}\widehat{\Theta}\Sigma_{N}(\widehat{\Theta})^T),\quad\|e\|_{\infty}\le \mathcal{O}_{\mathbb{P}}\left(\frac{\sigma_{\varepsilon}s_{2}\log n}{\sqrt{N}}\right).  
\end{equation*}
Furthermore, we have 
\begin{equation*}
\| \widehat{\Theta}\Sigma_{N}(\widehat{\Theta})^T-\Theta_{x}\|_{\infty}\le \mathcal{O}_{\mathbb{P}}\left(M_{n}\sqrt{\frac{\log n}{N}}\right). 
\end{equation*}

\end{theorem}

The proof of Theorem~\ref{onesteprandomthm} is contained in Appendix~\ref{onesteprandomthmprf}. We have a direct consequence of Theorem~\ref{onestepfixthm} and Theorem~\ref{onesteprandomthm}, stated in the following: 
\begin{corollary}
\label{deltaknormal}
Under the same conditions as Theorem~\ref{onesteprandomthm}, for $k\ge 0$ and $d\ge 2$, we have 
\begin{equation*}
\sqrt{N}(\Delta^{(k+1)}\tilde{\beta}-\Delta^{(k+1)}\beta^*)=\Psi^{(k+1)}-e^{(k+1)},
\end{equation*}
where 
\begin{equation*}
\Psi^{(k+1)}|X\sim N(0, \sigma^2_{\varepsilon}\Delta^{(k+1)}\widehat{\Theta}\Sigma_{N}(\widehat{\Theta})^T(\Delta^{(k+1)})^T),
\end{equation*}
and 
\begin{equation*}
\|e^{(k+1)}\|_{\infty}\le \mathcal{O}_{\mathbb{P}}\left( \sigma_{\varepsilon}(2d)^{\frac{k+1}{2}}\frac{s_{2}\log n}{\sqrt{N}}\right).  
\end{equation*}
Furthermore, we have   
\begin{equation*}
\begin{split}
\| \Delta^{(k+1)}\widehat{\Theta}\Sigma_{N}(\widehat{\Theta})^T(\Delta^{(k+1)})^T-\Delta^{(k+1)}\Theta_{x}(\Delta^{(k+1)})^T\|_{\infty} \le \mathcal{O}_{\mathbb{P}}\left( (2d)^{k+1}M_{n}\sqrt{\frac{\log n}{N}}  \right).  
\end{split}
\end{equation*}

\end{corollary}
The proof of Corollary~\ref{deltaknormal} is contained in Appendix~\ref{deltaknormalproof}. Corollary~\ref{deltaknormal} is of particular interest for statistical inference of $\Delta^{(k+1)}\beta^*$ in some applications, which will be discussed in more detail in Section~\ref{statinfer}.


\subsection{Some consequences}
\label{statinfer}

The results in Section~\ref{onestepresults} allow us to build asymptotically valid confidence intervals and perform hypothesis tests. In this section, we briefly discuss these consequences. We start with the simpler case that the standard deviation of the error in the linear model is known. We have the following result: 
 
\begin{corollary}
\label{CI}
Consider the linear model $(\ref{linear})$ where $\beta^*\in \mathcal{S}(k, s_{1}, s_{2})$ and $\sigma_{\varepsilon}$ is known. Under the same conditions as Theorem~\ref{onesteprandomthm}, if $\frac{s_{2}\log n}{\sqrt{N}}\rightarrow 0$ and $M_{n}\sqrt{\frac{\log n}{N}}\rightarrow 0$, then for $j\in [n]$, we have  
\begin{equation*}
\frac{\sqrt{N}(\tilde{\beta}_{j}-\beta^*_{j})}{\sigma_{\varepsilon}\sqrt{e_{j}^T\widehat{\Theta}\Sigma_{N}\widehat{\Theta}^Te_{j}}}\rightarrow N(0, 1). 
\end{equation*}

%
\end{corollary}

The proof of Corollary~\ref{CI} is contained in Appendix~\ref{CIprf}. Therefore, in view of Corollary~\ref{CI}, for $j\in [n]$ and the significance level $\alpha\in (0, 1)$,
\begin{equation}
\label{confidenceinterval1}
\left[ \tilde{\beta}_{j}-\Phi^{-1}(1-\frac{\alpha}{2})\sigma_{\varepsilon}\sqrt{\frac{e_{j}^T\widehat{\Theta}\Sigma_{N}\widehat{\Theta}^Te_{j}}{N}},~\tilde{\beta}_{j}+\Phi^{-1}(1-\frac{\alpha}{2})\sigma_{\varepsilon}\sqrt{\frac{e_{j}^T\widehat{\Theta}\Sigma_{N}\widehat{\Theta}^Te_{j}}{N}}\right]
\end{equation}
is an asymptotically valid $(1-\alpha)$-confidence interval for $\beta_{j}^{*}$. Here, $\Phi(x)$ is the cumulative distribution function of the standard normal distribution. 

Next, we consider the case when the standard deviation of the error is unknown. In this situation, we need an estimate of $\sigma_{\varepsilon}$. In particular, we obtain the estimate $\hat{\sigma}_{\varepsilon}$ from the consistent estimate of the regression coefficients via 
\begin{equation}
\label{sigmahat}
\hat{\sigma}_{\varepsilon}=\sqrt{\frac{1}{N}\sum_{i=1}^{N}\left(y_{i}-X_{i}^T\hat{\beta}\right)^2}. 
\end{equation}
We then have the following result:  
\begin{corollary}
\label{CI2}
Consider the linear model $(\ref{linear})$ where $\beta^*\in \mathcal{S}(k, s_{1}, s_{2})$ and $\sigma_{\varepsilon}$ is unknown. Under the same conditions as Corollary~\ref{CI}, for $j\in [n]$, we have  
\begin{equation*}
\frac{\sqrt{N}(\tilde{\beta}_{j}-\beta^*_{j})}{\hat{\sigma}_{\varepsilon}\sqrt{e_{j}^T\widehat{\Theta}\Sigma_{N}\widehat{\Theta}^Te_{j}}}\rightarrow N(0, 1). 
\end{equation*}
\end{corollary}

The proof of Corollary~\ref{CI2} is contained in Appendix~\ref{CI2prf}. Therefore, Corollary~\ref{CI2} implies  
\begin{equation}
\label{confidenceinterval2}
\left[ \tilde{\beta}_{j}-\Phi^{-1}(1-\frac{\alpha}{2})\hat{\sigma}_{\varepsilon}\sqrt{\frac{e_{j}^T\widehat{\Theta}\Sigma_{N}\widehat{\Theta}^Te_{j}}{N}},~\tilde{\beta}_{j}+\Phi^{-1}(1-\frac{\alpha}{2})\hat{\sigma}_{\varepsilon}\sqrt{\frac{e_{j}^T\widehat{\Theta}\Sigma_{N}\widehat{\Theta}^Te_{j}}{N}}\right]
\end{equation}
is an asymptotically valid $(1-\alpha)$-confidence interval for $\beta_{j}^{*}$. 

We have focused on the problem of confidence interval construction. In other applications, we might be interested in hypothesis testing. In the sequel, we discuss two types of hypothesis tests which can be solved by the proposed one-step estimator. First, we consider the following two-sided test for $\beta^*_{j}$: 
\begin{equation}
\label{test1}
H_{0, j}: \beta_{j}^{*}=0,\quad\text{vs.}\quad H_{A,j}: \beta_{j}^*\ne 0. 
\end{equation}
Corollary~\ref{CI} and Corollary~\ref{CI2} have immediate consequences for the problem $(\ref{test1})$. Let 
\begin{equation*}
  Z_{j}=
  \begin{cases}
    \frac{\sqrt{N}\tilde{\beta}_{j}}{\sigma_{\varepsilon}\sqrt{e_{j}^T\widehat{\Theta}\Sigma_{N}\widehat{\Theta}^Te_{j}}} & \text{if}~\sigma_{\varepsilon}~\text{is known} \\
       \frac{\sqrt{N}\tilde{\beta}_{j}}{\hat{\sigma}_{\varepsilon}\sqrt{e_{j}^T\widehat{\Theta}\Sigma_{N}\widehat{\Theta}^Te_{j}}} & \text{if}~\sigma_{\varepsilon}~\text{is unknown}.          
  \end{cases}
\end{equation*}
Then we define the following decision rule of the Z-test with significance level $\alpha$ for $(\ref{test1})$: 
\begin{equation*}
  T_{j}=
  \begin{cases}
                                   0 & \text{if}~|Z_{j}|\le \Phi^{-1}(1-\frac{\alpha}{2}) \\
                                   1 & \text{if}~|Z_{j}|> \Phi^{-1}(1-\frac{\alpha}{2}) 
  \end{cases}
\end{equation*}
That is, given the value of $T_{j}$, we reject the null hypothesis if and only if $T_{j}= 1$. Corollary~\ref{CI} and Corollary~\ref{CI2} imply that the type I error of $T_{j}$, i.e., the probability of rejecting $H_{0, j}$ when $H_{0, j}$ is true, can be controlled by $\alpha$ asymptotically. 

Next, we consider another type of test which might be of primary interest in our graph-based setting. Let $(u, v)$ be the $j$-th edge of the underlying graph $\mathcal{G}=(\mathcal{V}, \mathcal{E})$. We are interested in the following test: 
\begin{equation}
\label{test2}
H_{0, j}: \beta_{u}^{*}=\beta^*_{v},\quad\text{vs.}\quad H_{A,j}: \beta_{u}^*\ne \beta_{v}^*. 
\end{equation}
In order to propose an appropriate test statistic for $(\ref{test2})$, we present a useful result based on Corollary~\ref{deltaknormal} below:   
\begin{corollary}
\label{CI3}
Consider the linear model $(\ref{linear})$ where $\beta^*\in \mathcal{S}(k, s_{1}, s_{2})$. Under the same conditions as Corollary~\ref{deltaknormal}, if $dM_{n}\sqrt{\frac{\log n}{N}}\rightarrow 0$ and $s_{2}\log n\sqrt{\frac{d}{N}}\rightarrow 0$, then for $j\in [p]$, we have 
\begin{equation*}
\frac{\sqrt{N}(F_{j}^T\tilde{\beta}-F_{j}^T\beta^*)}{\sigma_{\varepsilon}\sqrt{F_{j}^T\widehat{\Theta}\Sigma_{N}\widehat{\Theta}^TF_{j}}}\rightarrow N(0, 1),
\end{equation*}
and 
\begin{equation*}
\frac{\sqrt{N}(F_{j}^T\tilde{\beta}-F_{j}^T\beta^*)}{\hat{\sigma}_{\varepsilon}\sqrt{F_{j}^T\widehat{\Theta}\Sigma_{N}\widehat{\Theta}^TF_{j}}}\rightarrow N(0, 1),
\end{equation*}
where $F\in \mathbb{R}^{p\times n}$ is the oriented incidence matrix. 
\end{corollary}
The proof of Corollary~\ref{CI3} is contained in Appendix~\ref{CI3proof}. Corollary~\ref{CI3} suggests selecting  
\begin{equation*}
  Z_{j}^{\prime}=
  \begin{cases}
   \frac{\sqrt{N}F_{j}^T\tilde{\beta}}{\sigma_{\varepsilon}\sqrt{F_{j}^T\widehat{\Theta}\Sigma_{N}\widehat{\Theta}^TF_{j}}} & \text{if}~\sigma_{\varepsilon}~\text{is known} \\
     \frac{\sqrt{N}F_{j}^T\tilde{\beta}}{\hat{\sigma}_{\varepsilon}\sqrt{F_{j}^T\widehat{\Theta}\Sigma_{N}\widehat{\Theta}^TF_{j}}} & \text{if}~\sigma_{\varepsilon}~\text{is unknown}          
  \end{cases}
\end{equation*}
as the test statistic for $(\ref{test2})$. Therefore, let  
\begin{equation*}
  T_{j}^{\prime}=
  \begin{cases}
                                   0 & \text{if}~|Z_{j}^{\prime}|\le \Phi^{-1}(1-\frac{\alpha}{2}) \\
                                   1 & \text{if}~|Z_{j}^{\prime}|> \Phi^{-1}(1-\frac{\alpha}{2}). 
  \end{cases}
\end{equation*}
We reject the null hypothesis of $(\ref{test2})$ if and only if $T_{j}^{\prime}=1$. 


\section{Simulations}
\label{simulationsection}

We now describe a variety of simulation results to assess the performance of our proposed methods. In all simulation studies, we solved both the optimization problems $(\ref{tf-sl})$ and $(\ref{CLIME})$ via the ADMM algorithms \cite{boyd2011admm}, which were implemented in the ADMM R package \cite{admmpackage} and the flare R package \cite{flarepackage}, respectively.   


\subsection{Simulation 1}
\label{simu1section}

In the first simulation study, our main interest was to compare the $\ell_{2}$-estimation error of our Graph-Piecewise-Polynomial-Lasso with other methods mentioned in the paper, including the Lasso, Smooth-Lasso and Spline-Lasso. We considered the situation where the underlying graph was a path graph with 250 nodes, i.e., $n=250$. Then we generated four different scenarios of $\beta^*$ described in the following:  
\begin{enumerate}
\item[(a)] Scenario 1: for $1\le j\le 250$,  
\begin{equation*}
\beta^*_{j}=
     \begin{cases}
       -1 & j\in [101, 110]\\
       1 & j \in [111, 120] \\
       -2 & j \in [121, 130]\\
       2  &  j \in [131, 140] \\ 
       1.5 &  j \in [141, 150]  \\
       0  &  \text{otherwise}.  
     \end{cases}
\end{equation*}
\item[(b)] Scenario 2: for $1\le j\le 250$, 
\begin{equation*}
\beta^*_{j}=
\begin{cases}
\frac{1}{5}\left|(j \bmod 25)-10\right|-1   & j \in [1, 10]\cup [50, 60]\cup [100, 110]\cup [150, 160]\cup [200, 210] \\
0   & \text{otherwise}. 
\end{cases}
\end{equation*}
\item[(c)] Scenario 3: for $1\le j\le 250$,  
\begin{equation*}
\beta_{j}^*=
\begin{cases} 
\frac{1}{50}\left( (x\bmod 50)-10  \right)^2-1   &  j \in [5, 15]\cup [105, 115]\cup [205, 215]  \\
-\frac{1}{50}\left( (x\bmod 50)-10  \right)^2+1   &  j \in [55, 65]\cup [155, 165] \\
0  &  \text{otherwise}. 
\end{cases}
\end{equation*}
\item[(d)] Scenario 4: for $1\le j\le 250$, 
\begin{equation*}
\beta_{j}^*=
\begin{cases} 
\sin(\frac{j}{10})+\cos(\frac{j}{3})   &  j \in [1, 10]\cup [50, 60]\cup [100, 110]\cup [150, 160]\cup [200, 210]  \\
0  &  \text{otherwise}. 
\end{cases}
\end{equation*}
\end{enumerate}
Scenarios 1, 2, and 3 correspond to subfigures (a), (b) and (c) in Figure~\ref{piecewisepolypathgraph}, and $\beta^*$ in Scenario 4 is a general smooth and sparse vector, which can be see in the left panel of Figure~\ref{scenario4truebeta}. Next, we generated each row of the design matrix $X$ from $N(0, I_{n\times n})$ and each $\varepsilon_{i}$ from $N(0, 0.1)$. Finally, the response vector $y$ was generated via the linear model in $(\ref{linear})$. 

In Scenarios 1, 2, and 3, we set $k=0, 1$, and $2$, respectively. In Scenario 4, we chose the value of $k$ by cross-validation. The tuning parameters of each method were also chosen via the 5-fold cross-validation procedure, which minimized the cross-validated prediction error. For each scenario, we considered three sample sizes for training data: $N=100$, $N=150$, and $N=200$. We repeated the simulation 50 times. Table~\ref{simu1results} shows the simulation results.  Our approach outperformed the other three methods in all scenarios across all sampling schemes except for $(N, n)=(100, 250)$ in Scenario 3, where the Spline-Lasso was the best.    

\begin{figure}[t]
\begin{subfigure}{.5\textwidth}
  \centering
  \includegraphics[width=.8\linewidth]{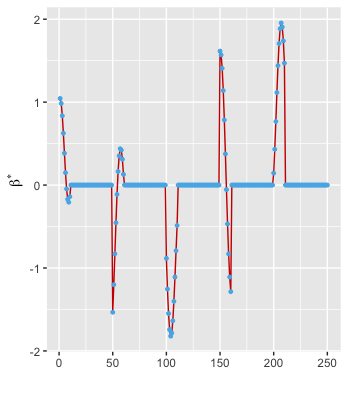}  
  \caption{Path graph with 250 nodes}
\end{subfigure}
\begin{subfigure}{.5\textwidth}
 \centering
  \includegraphics[width=.88\linewidth]{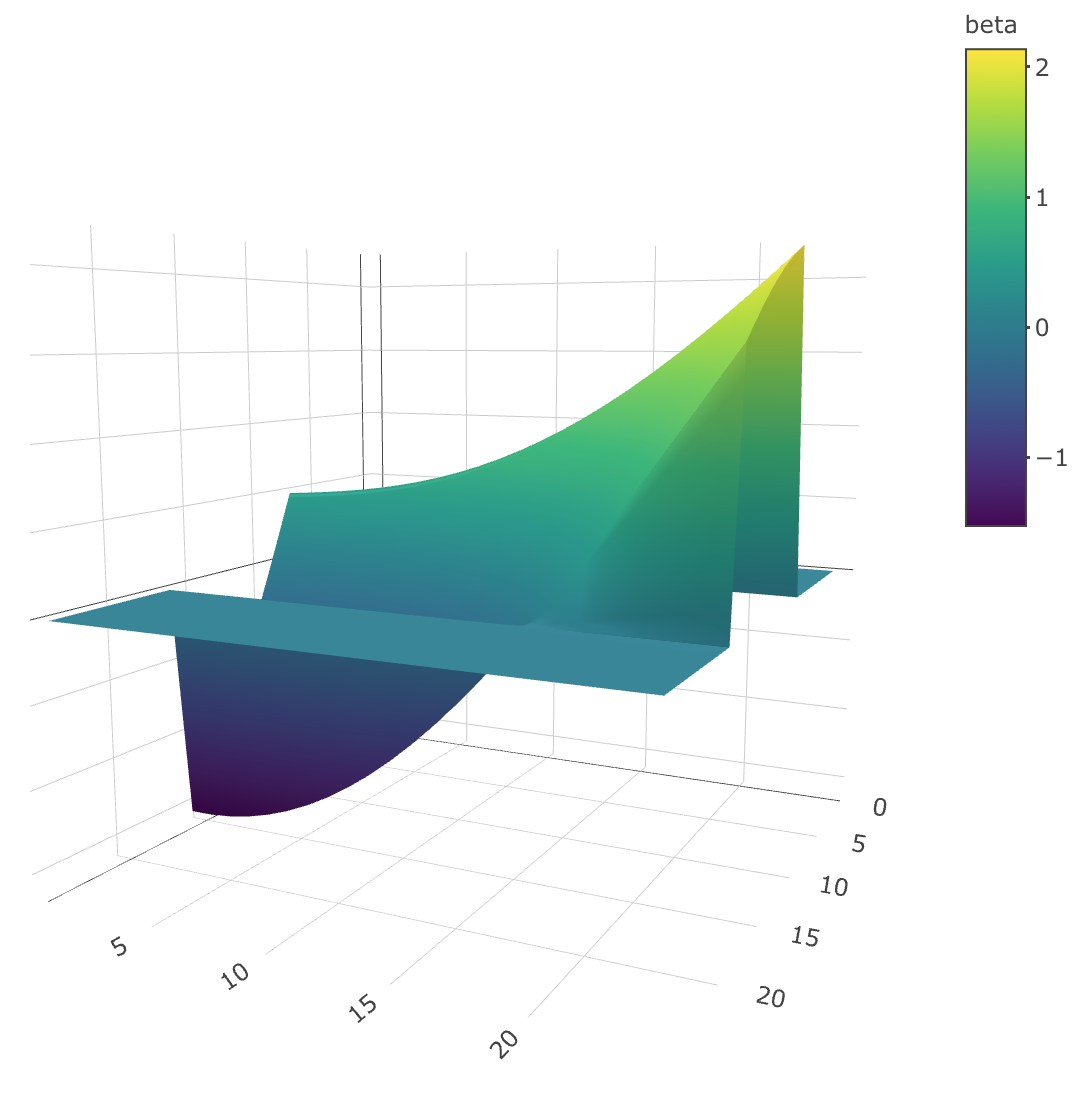}  
  \caption{2d grid graph with 25 rows and 25 columns}
\end{subfigure}
\caption{Simultaneously sparse and general smooth regression coefficients constructed in Scenario 4 of Section~\ref{simu1section} and Section~\ref{simu2section}. }
\label{scenario4truebeta}
\end{figure}

\begin{table}
\caption {Averages (standard errors) of $\ell_{2}$ estimation error in Simulation 1. The minimal averages are in bold. } 
\label{simu1results}  
\centering
\begin{tabular}{llllllllll}
\hline \hline 
   & \multicolumn{2}{c}{$(N, n)=(100, 250)$} &  & \multicolumn{2}{c}{$(N, n)=(150, 250)$} &   & \multicolumn{2}{c}{$(N, n)=(200, 250)$} \\ 
 &   \multicolumn{2}{c}{$\|\hat{\beta}-\beta^*\|_{2}$}   &            &       \multicolumn{2}{c}{$\|\hat{\beta}-\beta^*\|_{2}$}       &     &    \multicolumn{2}{c}{$\|\hat{\beta}-\beta^*\|_{2}$}  \\    
 \hline
     &    &    &   \text{Scenario 1}  &   &   &    \\ \hline 
 \text{Our approach}   &  \multicolumn{2}{c}{{\bf 0.767} (0.115)}       &        &  \multicolumn{2}{c}{{\bf 0.377} (0.008)}     &   & \multicolumn{2}{c}{{\bf 0.328} (0.004)}                         \\
 \text{Lasso}   &  \multicolumn{2}{c}{8.257 (0.129)}      &        &  \multicolumn{2}{c}{1.296 (0.083)}    &   & \multicolumn{2}{c}{0.551 (0.010)}                         \\
  \text{Smooth-Lasso}   & \multicolumn{2}{c}{4.671 (0.146)}          &        &  \multicolumn{2}{c}{1.289 (0.053)}    &    & \multicolumn{2}{c}{0.575 (0.010)}                      \\
   \text{Spline-Lasso}   & \multicolumn{2}{c}{3.498 (0.038)}   &   & \multicolumn{2}{c}{2.624 (0.024)}     &       &  \multicolumn{2}{c}{2.376 (0.017)}             \\
      \hline
      &    &    &   \text{Scenario 2}  &   &   &    \\ \hline 
 \text{Our approach}   &  \multicolumn{2}{c}{{\bf 0.896} (0.039)}       &        &  \multicolumn{2}{c}{{\bf 0.420} (0.008)}     &   & \multicolumn{2}{c}{{\bf 0.338} (0.005)}                         \\
 \text{Lasso}   &  \multicolumn{2}{c}{3.081 (0.056)}      &        &  \multicolumn{2}{c}{1.310 (0.038)}    &   & \multicolumn{2}{c}{0.474 (0.009)}                         \\
  \text{Smooth-Lasso}   & \multicolumn{2}{c}{2.007 (0.050)}          &        &  \multicolumn{2}{c}{0.795 (0.023)}    &    & \multicolumn{2}{c}{0.469 (0.009)}                      \\
   \text{Spline-Lasso}   & \multicolumn{2}{c}{1.922 (0.019)}   &   & \multicolumn{2}{c}{1.639 (0.013)}     &       &  \multicolumn{2}{c}{1.409 (0.013)}             \\
      \hline
   &    &    &   \text{Scenario 3}  &   &   &    \\ \hline 
 \text{Our approach}   &  \multicolumn{2}{c}{2.102 (0.135)}       &        &  \multicolumn{2}{c}{{\bf 0.574} (0.012)}     &   & \multicolumn{2}{c}{{\bf 0.374} (0.005)}                         \\
 \text{Lasso}   &  \multicolumn{2}{c}{4.761 (0.055)}      &        &  \multicolumn{2}{c}{1.607 (0.093)}    &   & \multicolumn{2}{c}{0.536 (0.012)}                         \\
  \text{Smooth-Lasso}   & \multicolumn{2}{c}{1.645 (0.064)}          &        &  \multicolumn{2}{c}{0.642 (0.013)}    &    & \multicolumn{2}{c}{0.439 (0.007)}                      \\
   \text{Spline-Lasso}   & \multicolumn{2}{c}{{\bf 0.900} (0.020)}   &   & \multicolumn{2}{c}{0.664 (0.007)}     &       &  \multicolumn{2}{c}{0.587 (0.005)}             \\
      \hline
 &    &    &   \text{Scenario 4}  &   &   &    \\ \hline 
 \text{Our approach}   &  \multicolumn{2}{c}{{\bf 1.097} (0.059)}       &        &  \multicolumn{2}{c}{{\bf 0.469} (0.008)}     &   & \multicolumn{2}{c}{{\bf 0.358} (0.006)}                         \\
 \text{Lasso}   &  \multicolumn{2}{c}{4.692 (0.124)}      &        &  \multicolumn{2}{c}{1.086 (0.049)}    &   & \multicolumn{2}{c}{0.583 (0.014)}                         \\
  \text{Smooth-Lasso}   & \multicolumn{2}{c}{2.331 (0.072)}          &        &  \multicolumn{2}{c}{0.840 (0.026)}    &    & \multicolumn{2}{c}{0.505 (0.010)}                      \\
   \text{Spline-Lasso}   & \multicolumn{2}{c}{1.863 (0.024)}   &   & \multicolumn{2}{c}{1.544 (0.011)}     &       &  \multicolumn{2}{c}{1.381 (0.016)}             \\ 
      \hline \hline
\end{tabular}
\end{table}

\subsection{Simulation 2}
\label{simu2section}
The main goal of our second simulation study was similar to the one in the first simulation study, but we considered the situation where the underlying graph was a 2d grid graph with 25 rows and 25 columns. Therefore, the Smooth-Lasso and the Spline-Lasso were replaced by their corresponding variants in this simulation. We first generated following four different scenarios of $B^{*}\in \mathbb{R}^{25\times 25}$ and then obtained $\beta^*\in \mathbb{R}^{625}$ via stacking the columns of $B^{*}$ on top of one another:
\begin{enumerate}
\item[(a)] Scenario 1: for $1\le i\le 25$ and $1\le j\le 25$, 
\begin{equation*}
B^{*}_{ij}=  \begin{cases}
       0.5 & (i, j)\in [9, 13]\times [13, 17]\\
       -1 & (i, j) \in [9, 13]\times [9, 12] \\
       +1 & (i, j) \in [14, 17]\times [9, 12] \\
       -0.5  &  (i, j) \in [14, 17]\times [13, 17] \\ 
       0  &  \text{otherwise}.  
     \end{cases}
\end{equation*}
\item[(b)] Scenario 2: for $1\le i\le 25$ and $1\le j\le 25$, 
\begin{equation*}
B^{*}_{ij}=  \begin{cases}
       0.1(i+j)-2.6  & (i, j)\in [9, 13]\times [13, 17]\\
       2.6-0.1(i+j) & (i, j) \in [9, 13]\times [9, 12] \\
       0.1(j-i) & (i, j) \in [14, 17]\times [9, 17] \\
       0  &  \text{otherwise}.  
     \end{cases}
\end{equation*}

\item[(c)] Scenario 3: for $1\le i\le 25$ and $1\le j\le 25$, 
\begin{equation*}
B^{*}_{ij}=  \begin{cases}
         0.7(0.1j-0.7)^2 & (i, j)\in [9, 13]\times [1, 12]\\
         0.7(0.1j-1.9)^2 & (i, j) \in [9, 13]\times [13, 25] \\
         -0.7(0.1j-0.7)^2  & (i, j) \in [14, 17]\times [1, 12] \\
           -0.7(0.1j-1.9)^2  & (i, j) \in [14, 17]\times [13, 25] \\         
       0  &  \text{otherwise}.  
     \end{cases}
\end{equation*}
\item[(d)] Scenario 4: for $1\le i\le 25$ and $1\le j\le 25$, 
\begin{equation*}
B^{*}_{ij}=  \begin{cases}
        \sin{\left(\frac{0.1j-1.3}{8}\right)}-\cos{\left(\frac{0.1i-1.3}{10}\right)}+2\sin{\left(\frac{0.1j-1.3}{2}-(0.1i-1.3)\right)} \\~~~~~~~~~~~~~~~~~~ -\cos{\left(0.1(i+j)-2.6\right)}+2 & (i, j)\in [9, 17]\times [1, 25]\\      
       0  &  \text{otherwise}.  
     \end{cases}
\end{equation*}
\end{enumerate}
Scenarios 1, 2, and 3 correspond to subfigures (a), (b) and (c) in Figure~\ref{piecewisepoly2dgraph}. Scenario 4 corresponds to the right panel of Figure~\ref{scenario4truebeta}. In the remaining steps, we followed the same procedure in Section~\ref{simu1section} except that the sample sizes of training data were replaced by $N=250, 375$, and $500$. The results of our second simulation are summarized in Table~\ref{simu2results}. Overall, our approach had much better performance compared to the other three methods.  

\begin{table}
\caption {Averages (standard errors) of $\ell_{2}$ estimation error in Simulation 2. The minimal averages are in bold. } 
\label{simu2results}  
\centering
\begin{tabular}{llllllllll}
\hline \hline 
   & \multicolumn{2}{c}{$(N, n)=(250, 625)$} &  & \multicolumn{2}{c}{$(N, n)=(375, 625)$} &   & \multicolumn{2}{c}{$(N, n)=(500, 625)$} \\ 
 &   \multicolumn{2}{c}{$\|\hat{\beta}-\beta^*\|_{2}$}   &            &       \multicolumn{2}{c}{$\|\hat{\beta}-\beta^*\|_{2}$}       &     &    \multicolumn{2}{c}{$\|\hat{\beta}-\beta^*\|_{2}$}  \\    
 \hline
     &    &    &   \text{Scenario 1}  &   &   &    \\ \hline 
 \text{Our approach}   &  \multicolumn{2}{c}{{\bf 0.433} (0.007)}       &        &  \multicolumn{2}{c}{{\bf 0.345} (0.003)}     &   & \multicolumn{2}{c}{{\bf 0.364} (0.002)}                         \\
 \text{Lasso}   &  \multicolumn{2}{c}{3.145 (0.077)}      &        &  \multicolumn{2}{c}{0.538 (0.012)}    &   & \multicolumn{2}{c}{0.381 (0.003)}                         \\
  \text{Graph-Smooth-Lasso}   & \multicolumn{2}{c}{2.288 (0.063)}          &        &  \multicolumn{2}{c}{0.618 (0.016)}    &    & \multicolumn{2}{c}{0.384 (0.004)}                      \\
   \text{Graph-Spline-Lasso}   & \multicolumn{2}{c}{3.439 (0.017)}   &   & \multicolumn{2}{c}{3.191 (0.018)}     &       &  \multicolumn{2}{c}{2.990 (0.011)}             \\
      \hline
      &    &    &   \text{Scenario 2}  &   &   &    \\ \hline 
 \text{Our approach}   &  \multicolumn{2}{c}{{\bf 0.406} (0.007)}       &        &  \multicolumn{2}{c}{{\bf 0.319} (0.003)}     &   & \multicolumn{2}{c}{{\bf 0.290} (0.003)}                         \\
 \text{Lasso}   &  \multicolumn{2}{c}{0.907 (0.023)}      &        &  \multicolumn{2}{c}{0.445 (0.005)}    &   & \multicolumn{2}{c}{0.336 (0.005)}                         \\
  \text{Graph-Smooth-Lasso}   & \multicolumn{2}{c}{0.860 (0.020)}          &        &  \multicolumn{2}{c}{0.447 (0.005)}    &    & \multicolumn{2}{c}{0.339 (0.004)}                      \\
   \text{Graph-Spline-Lasso}   & \multicolumn{2}{c}{1.488 (0.010)}   &   & \multicolumn{2}{c}{1.365 (0.005)}     &       &  \multicolumn{2}{c}{1.311 (0.004)}             \\
      \hline
   &    &    &   \text{Scenario 3}  &   &   &    \\ \hline 
 \text{Our approach}   &  \multicolumn{2}{c}{ {\bf 0.735} (0.010)}       &        &  \multicolumn{2}{c}{{\bf 0.491} (0.005)}     &   & \multicolumn{2}{c}{0.455 (0.004)}                         \\
 \text{Lasso}   &  \multicolumn{2}{c}{1.449 (0.010)}      &        &  \multicolumn{2}{c}{0.749 (0.009)}    &   & \multicolumn{2}{c}{0.503 (0.005)}                         \\
  \text{Graph-Smooth-Lasso}   & \multicolumn{2}{c}{0.955 (0.011)}          &        &  \multicolumn{2}{c}{0.598 (0.006)}    &    & \multicolumn{2}{c}{{\bf 0.440} (0.004)}                      \\
   \text{Graph-Spline-Lasso}   & \multicolumn{2}{c}{0.775 (0.007)}   &   & \multicolumn{2}{c}{0.658 (0.004)}     &       &  \multicolumn{2}{c}{0.609 (0.002)}             \\
      \hline
 &    &    &   \text{Scenario 4}  &   &   &    \\ \hline 
 \text{Our approach}   &  \multicolumn{2}{c}{{\bf 3.603} (0.120)}       &        &  \multicolumn{2}{c}{{\bf 1.012} (0.022)}     &   & \multicolumn{2}{c}{{\bf 0.612} (0.006)}                         \\
 \text{Lasso}   &  \multicolumn{2}{c}{11.865 (0.059)}      &        &  \multicolumn{2}{c}{6.687 (0.091)}    &   & \multicolumn{2}{c}{1.718 (0.041)}                         \\
  \text{Graph-Smooth-Lasso}   & \multicolumn{2}{c}{5.960 (0.093)}          &        &  \multicolumn{2}{c}{2.516 (0.047)}    &    & \multicolumn{2}{c}{1.070 (0.021)}                      \\
   \text{Graph-Spline-Lasso}   & \multicolumn{2}{c}{3.622 (0.019)}   &   & \multicolumn{2}{c}{3.175 (0.013)}     &       &  \multicolumn{2}{c}{3.015 (0.012)}             \\ 
      \hline \hline
\end{tabular}
\end{table}

\subsection{Simulation 3}
We now shift our focus to the problem of statistical inference in the third simulation study. Our first task was to verify the theoretical results in Corollary~\ref{CI} and construct confidence intervals for $\beta_{1}^*$. We considered $\beta^*$ described in Scenario 1 of Section~\ref{simu1section} and $N=200$. The steps of the experiment are summarized below:  
\begin{enumerate}
\item We generated each row of $X$ from $N(0, I_{n\times n})$ and solved the optimization problem $(\ref{CLIME})$ with $\mu=0.05\sqrt{\frac{\log n}{N}}$.  
\item We generated $\varepsilon_{i}$ from $N(0, 0.1)$, then generated the response $y$ via the linear model $y=X\beta^*+\varepsilon$. 
\item We solved the optimization problem $(\ref{tf-sl})$ with $\lambda$ and $\lambda_{g}$ used in Simulation 1.
\item We took the first component as an example and calculated one realization of $\frac{\sqrt{N}(\tilde{\beta}_{1}-\beta^*_{1})}{\sigma_{\varepsilon}\sqrt{e_{1}^T\widehat{\Theta}\Sigma_{N}\widehat{\Theta}^Te_{1}}}$, where $\tilde{\beta}_{1}$ was computed via $(\ref{onestep3})$. We also constructed a $95\%$ confidence interval for $\beta^*_{1}$ by $(\ref{confidenceinterval1})$.  
\item We repeated the second, third, and fourth steps 200 times. 
\end{enumerate}
Panel (a) in Figure~\ref{qqplots} shows the Q-Q plot of $\frac{\sqrt{N}(\tilde{\beta}_{1}-\beta^*_{1})}{\sigma_{\varepsilon}\sqrt{e_{1}^T\widehat{\Theta}\Sigma_{N}\widehat{\Theta}^Te_{1}}}$. The scatter points are close to the 45-degree line, which confirms the normal sampling distribution in Corollary~\ref{CI}. Panel (b) of Figure~\ref{qqplots} shows the confidence interval coverage based on 200 trials. We also conducted a similar experiment to verify the results in Corollary~\ref{CI2}. In the new experiment, we chose $\mu=0.08\sqrt{\frac{\log n}{N}}$ in the first step, and calculated $\frac{\sqrt{N}(\tilde{\beta}_{1}-\beta^*_{1})}{\hat{\sigma}_{\varepsilon}\sqrt{e_{1}^T\widehat{\Theta}\Sigma_{N}\widehat{\Theta}^Te_{1}}}$ in the 4th step, where $\hat{\sigma}_{\varepsilon}$ is given in $(\ref{sigmahat})$. We also constructed a $95\%$ confidence interval by $(\ref{confidenceinterval2})$ in the 4th step. The corresponding Q-Q plot and the confidence intervals are displayed in Panel (c) and (d) of Figure~\ref{qqplots}, respectively. 

\begin{figure}[!t]
\begin{subfigure}{.5\textwidth}
  \centering
  \includegraphics[width=.8\linewidth]{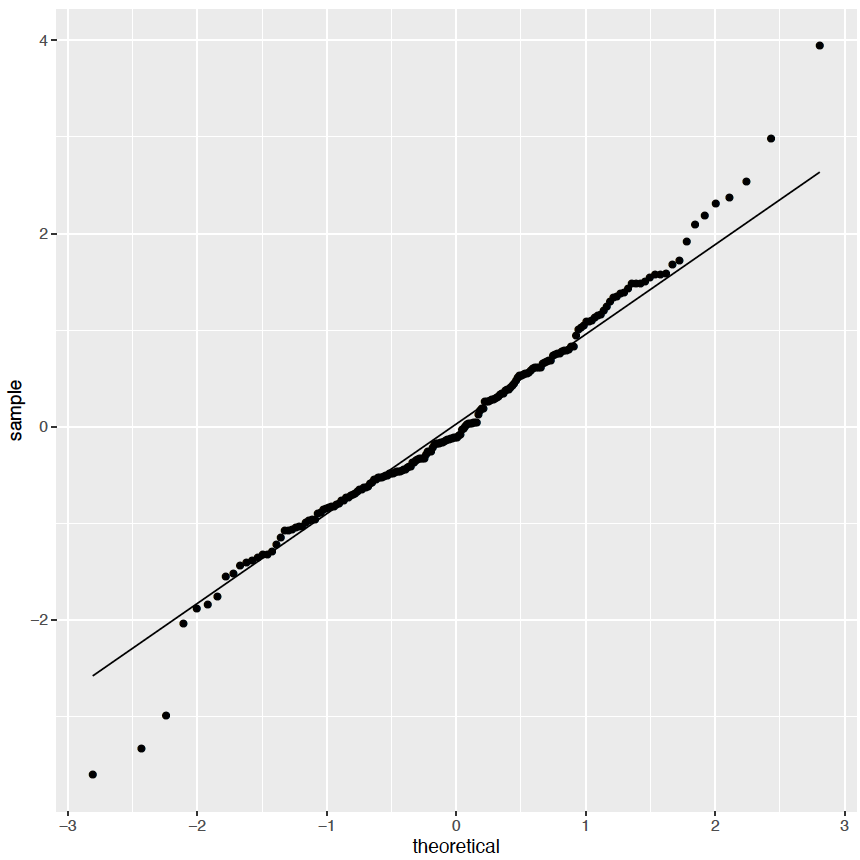}  
  \caption{}
\end{subfigure}
\begin{subfigure}{.5\textwidth}
 \centering
  \includegraphics[width=.8\linewidth]{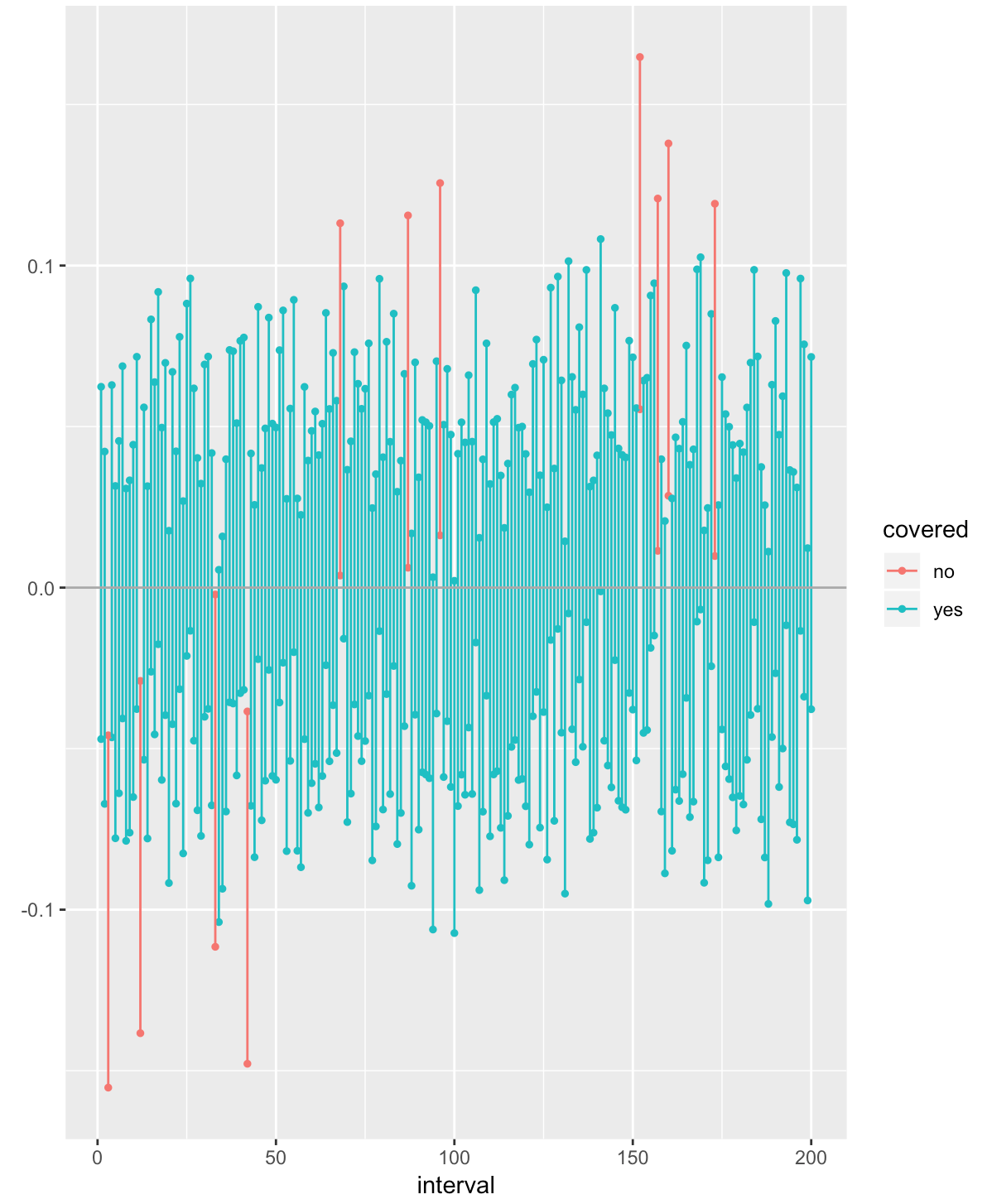}  
 \caption{}
\end{subfigure}
\begin{subfigure}{.5\textwidth}
  \centering
  \includegraphics[width=.8\linewidth]{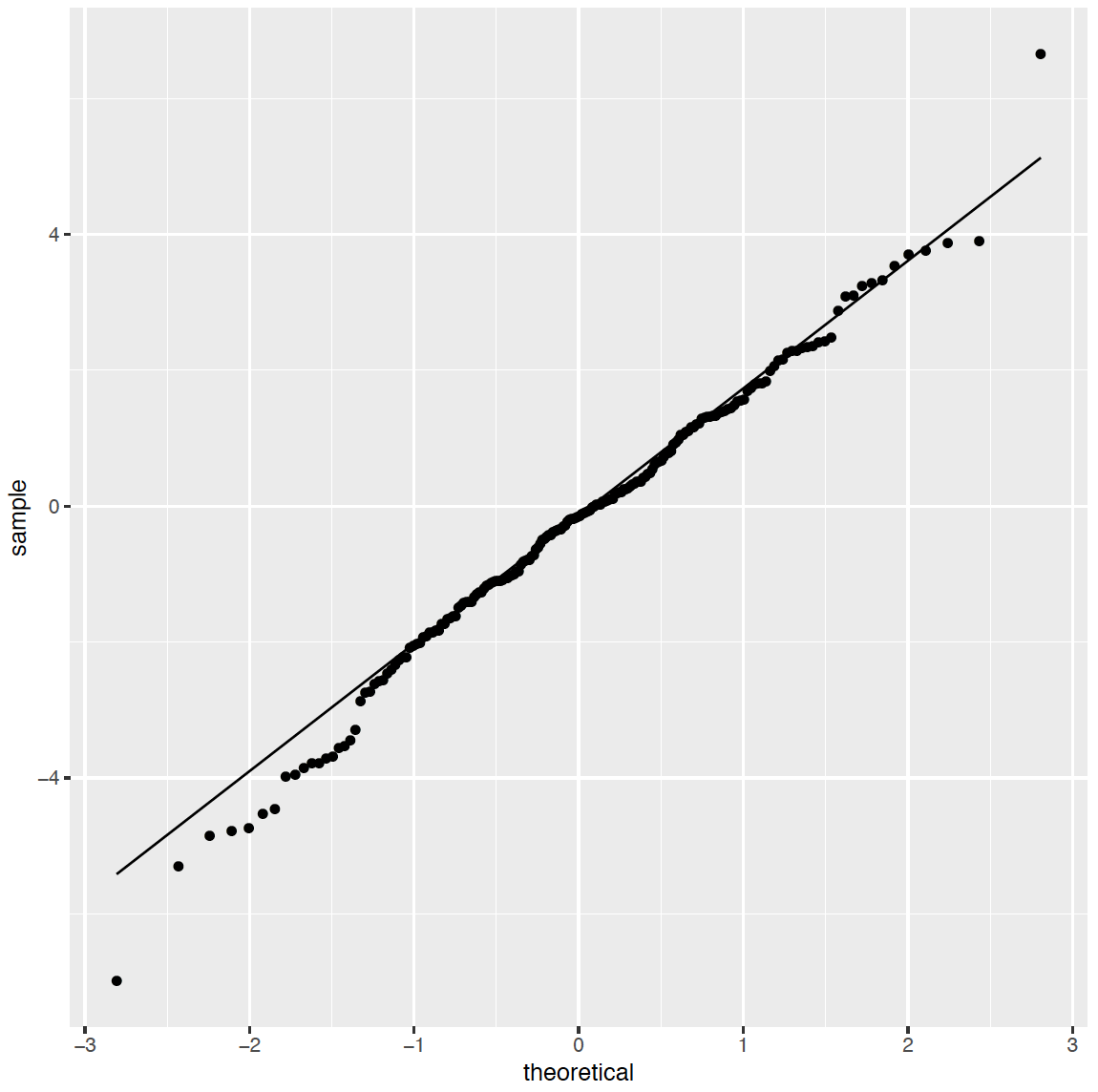}  
  \caption{}
\end{subfigure}
\begin{subfigure}{.5\textwidth}
 \centering
  \includegraphics[width=.8\linewidth]{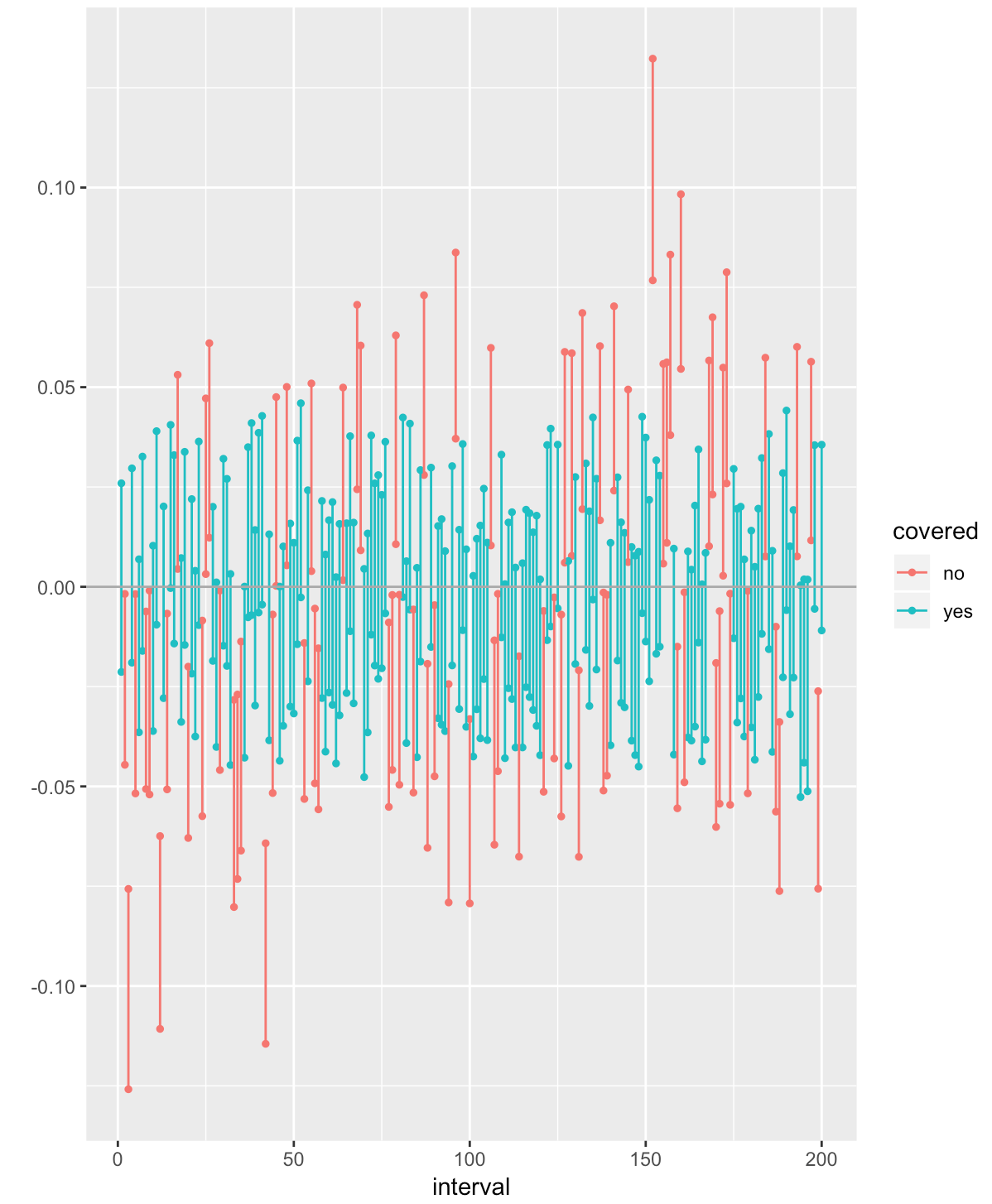}  
 \caption{}
\end{subfigure}
\caption{Q-Q plots and confidence intervals based on 200 trials. Panel (a) and Panel (c) are Q-Q plots of $\frac{\sqrt{N}(\tilde{\beta}_{1}-\beta^*_{1})}{\sigma_{\varepsilon}\sqrt{e_{1}^T\widehat{\Theta}\Sigma_{N}\widehat{\Theta}^Te_{1}}}$ and $\frac{\sqrt{N}(\tilde{\beta}_{1}-\beta^*_{1})}{\hat{\sigma}_{\varepsilon}\sqrt{e_{1}^T\widehat{\Theta}\Sigma_{N}\widehat{\Theta}^Te_{1}}}$, respectively. Panel (b) and Panel (d) show $95\%$ confidence intervals for $\beta^*_{1}$ constructed by $(\ref{confidenceinterval1})$ and $(\ref{confidenceinterval2})$, respectively. The empirical coverage in Panel (b) was $94.5\%$ and the empirical coverage in Panel (d) was $62.5\%$. }
\label{qqplots}
\end{figure}
 
Finally, we focused on the hypothesis testing problem. We considered one instance of $(\ref{test2})$: $H_{0}: \beta_{1}^*=\beta_{2}^*$ vs.\ $H_{A}: \beta_{1}^*\ne \beta_{2}^*$. Our goal was to check the validity of the Type I error of our proposed method. We took the setting where $\sigma_{\varepsilon}$ is known as an example. The first three steps of the procedure were same as those in the first experiment. In the 4th step, we calculated the test statistic $\frac{\sqrt{N}(\tilde{\beta}_{2}-\tilde{\beta}_{1})}{\sigma_{\varepsilon}\sqrt{F_{j}^T\widehat{\Theta}\Sigma_{N}\widehat{\Theta}^TF_{j}}}$, where $F_{j}=(-1, +1, 0,..., 0)$, and decided whether to reject $H_{0}$ at a $5\%$ significance level. The number of simulations was 200. The empirical Type I error was 0.04, which was close to the significance level.


\section{Application to an \emph{Arabidopsis thaliana} microarray dataset}
\label{realdatasection}

One motivation of our proposed method comes from the analysis of gene expression data, where genes within a same cluster have similar patterns. In this section, we report the performance of our approach to analyze a microarray dataset which was related to the isoprenoid biosynthesis in \emph{Arabidopsis thaliana}. In the application, we focused on identifying genes which are associated with the isoprenoid gene called GGPPS11 among hundreds of candidates from 58 metabolic pathways. In order to use our approach, the Smooth-Lasso, and the Spline-Lasso efficiently, we constructed the underlying graph as a path graph. More specifically, we ordered the candidate genes from the same pathway into a path subgraph, then each subgraph was concatenated by the alphabetical order of names of pathways. Therefore, each row of our design matrix recorded the expression levels measured from these ordered genes and the corresponding response variable was the expression level of GGPPS11. All variables in our analysis were log-transformed, centered and standardized to the unit variance. Finally, the dataset we used after the data preprocessing step consisted of $118$ samples and $777$ candidate genes. A more detailed description of the real data experiment can be found in \cite{wille2004} and \cite{chakraborty2019}.    

First, we compared the prediction accuracy for the four mentioned methods. All tuning parameters were selected via the 5-fold cross-validation procedure introduced in Section~\ref{simu1section}. For our Graph-Piecewise-Polynomial-Lasso, we chose the order $k$ to be 0 after performing a similar cross-validation among the set $\{0, 1, 2, 3\}$. We randomly split the whole dataset into the training and testing sets, which included 92 and 26 samples, respectively. We used the training set to estimate the regression coefficients and then calculated the mean squared prediction error (MSE) for the testing set. For robustness, we repeated the above dataset partition, estimation, and prediction process 50 times. The results are presented in Table~\ref{realdataMSE} and Figure~\ref{mseboxplot}. Overall, our approach achieved smaller MSE than all other methods.        

 \begin{figure}[t]
 \begin{minipage}[b]{0.5\textwidth}
    \centering 
    \begin{tabular}{cccc}\hline\hline  
      Method & $Q_{1}$ & Median & $Q_{3}$ \\ \hline
      Our approach & {\bf 0.25} &{\bf 0.30}  &  {\bf 0.42} \\ \hline
      Lasso & 0.33 &0.38 & 0.52 \\ \hline
      Smooth-lasso & 0.31 &0.36 & 0.44 \\ \hline
       Spline-Lasso &  0.31 &0.36 & 0.43 \\ \hline\hline
      \end{tabular}
     \captionof{table}{The first quartile, the median and the third quartile of MSEs. The minimal ones are in bold. }
      \label{realdataMSE}
    \end{minipage}
\hfill 
  \begin{minipage}[b]{0.5\textwidth}
    \centering
    \includegraphics[width=1\linewidth]{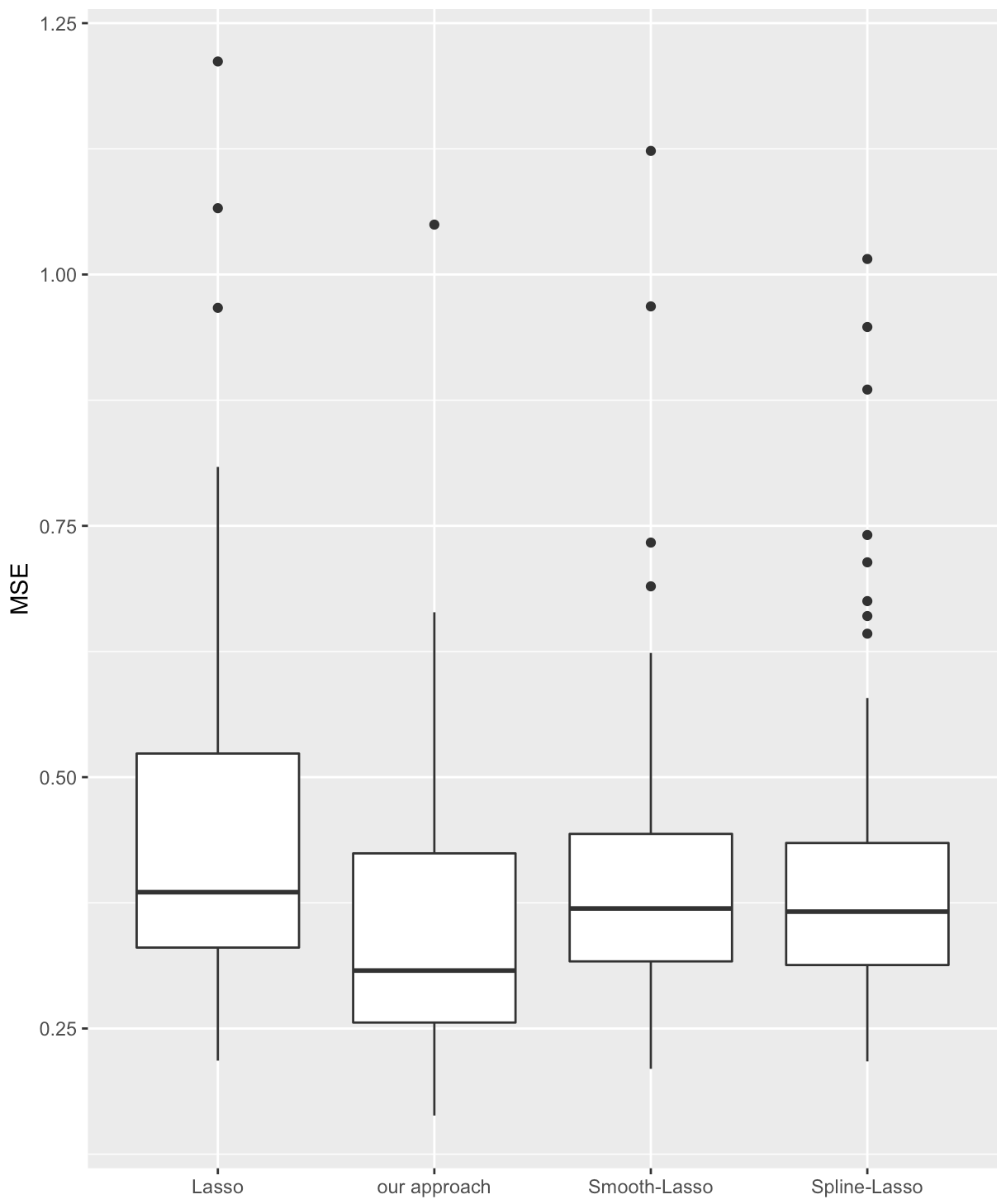}  
    \captionof{figure}{Boxplot of MSEs. }
    \label{mseboxplot}
  \end{minipage}
\end{figure}

We also applied our approach to the full dataset with the optimal tuning parameters chosen at the previous stage and analyzed the selected genes. Panel (a) of Figure~\ref{realdatacoef} in Appendix~\ref{realdataappendix} shows the estimated regression coefficients of 777 candidate genes across 58 pathways. Furthermore, we took the Purinemetabolism pathway as an example and plotted the corresponding coefficients in Panel (b) of Figure~\ref{realdatacoef}. Note that the estimated regression coefficients were piecewise constant between and within pathways, which could be very useful for other biological tasks such as the cluster analysis of genes. On the other hand, our proposed method selected 107 candidate genes which belong to 27 different pathways. In most cases, only a subset of genes within a given pathway was selected. Pathways which had top 5 percentages of selected genes included Morphinemetabolism, Tocopherolbiosynthesis, Chorismatemetabolism, Histidinemetabolism, and Flavonoidmetabolism. These findings were consistent with those reported in \cite{wille2004}. See Table~\ref{realdatapathwaytable} in Appendix~\ref{realdataappendix} for a complete summary of selected genes. 


\section{Discussion}
\label{discussion}

We have developed a flexible approach to estimate and infer graph-based regression coefficients in high-dimensional linear models. In the paper, we assume the order $k$ of the Graph-Piecewise-Polynomial-Lasso is known for ease of presentation, but in practice, we could select the best $k$ through the cross-validation procedure as we did in the simulation study and the real data analysis. From a practical point of view, this is one significant benefit of our approach in the sense that we are able to estimate regression coefficients with any complex structure by tuning $k$. In contrast, other existing methods such as the fused Lasso are designed for only one particular structure.           

We have established rigorous upper bounds on the estimation error and the prediction error for our approach. We mention one open question that is not addressed by the theory in the current paper. Recall that $S_{1}$ and $S_{2}$ are the support sets of $\Delta^{(k+1)}\beta^*$ and $\beta^*$, respectively. Furthermore, given an optimal solution $\hat{\beta}$ from $(\ref{tf-sl})$, we also denote the support sets of $\Delta^{(k+1)}\hat{\beta}$ and $\hat{\beta}$ by $\widehat{S}_{1}$ and $\widehat{S}_{2}$. Then in terms of the Graph-Piecewise-Polynomial-Lasso, it is interesting to ask the following question in our context: when are the support sets $\widehat{S}_{1}$ and $\widehat{S}_{2}$ exactly equal to the true support sets $S_{1}$ and $S_{2}$? We refer to this property as \emph{variable selection and change-point detection consistency}. We have attempted to explore this property via a routine application of the primal-dual witness type arguments \cite{wainwright2009sharp}, but have had no success. We suspect that this is because of the potential interactions between specifying the support of $\beta^*$ and the support of $\Delta^{(k+1)}\beta^*$. 

Finally, our paper suggests several directions for future research. Our current work considers the piecewise polynomial structure over the unweighted graph.  
Similar piecewise polynomial structure over a weighted graph could be defined using a weighted version of the oriented incidence matrix in Definition~\ref{diffoperator} and the same recursion in Definition~\ref{diffoperator2}. It would also be helpful to generalize the linear model to more general settings, such as generalized linear models.

%
%

\nocite{*}
\bibliographystyle{plain}
\bibliography{refs}

%
%

\appendix

\counterwithin{figure}{section}
\counterwithin{table}{section}


\section{Proofs of theorems}
\label{proofsection}
\setcounter{equation}{0}
\renewcommand{\theequation}{A.\arabic{equation}}

In this section, we provide proofs of Theorem~\ref{fixedconthm}, Theorem~\ref{probconthm}, Theorem~\ref{weaklypiecewisepolycorollary}, Theorem~\ref{weaklypiecewisepolyrandomthm}, Theorem~\ref{onestepfixthm} and Theorem~\ref{onesteprandomthm} established in the paper.  

\subsection{Proof of Theorem~\ref{fixedconthm}}
\label{fixedconthmprf}

We start with a supporting lemma which concerns the geometry of $D(\hat{\beta}-\beta^*)$. 

\begin{lemma}
\label{conelemma}
If the tuning parameter $\lambda$ satisfies the condition that 
\begin{equation*}
\lambda\ge \frac{2}{N}\|\varepsilon^TXD^{+} \|_{\infty},
\end{equation*}
then $D(\hat{\beta}-\beta^*)$ is in the cone $\mathbb{C}$ defined in Condition~\ref{lowerRE}. 
\end{lemma}

The proof of Lemma~\ref{conelemma} is deferred to Appendix~\ref{conelemmaproof}. We now prove Theorem~\ref{fixedconthm} in the following. 

\begin{proof}
We first show Part (a). Let $\Delta=\hat{\beta}-\beta^*$. Using a similar argument with Lemma~\ref{conelemma}, we can show that if $\lambda\ge\frac{2}{N}\|\varepsilon^TXD^{+}\|_{\infty}$, then 
\begin{equation}
\label{eq1thm2}
\begin{split}
\frac{1}{N}\|X\Delta\|_{2}^{2}\le 3\lambda\|(D\Delta)_{S}\|_{1}=3\lambda_{g}\|\Delta^{(k+1)}_{S_{1}}\Delta\|_{1}+3\lambda\|\Delta_{S_{2}}  \|_{1} \\
\le 3\lambda_{g}\sqrt{|S_{1}|}\|\Delta^{(k+1)}\Delta\|_{2}+3\lambda\sqrt{|S_{2}|}\|\Delta\|_{2}\le \left( 3\lambda_{g}\sqrt{(2d)^{k+1}|S_{1}|}+3\lambda\sqrt{|S_{2}|}  \right)\|\Delta\|_{2},
\end{split}
\end{equation}
where $S$, $S_{1}$ and $S_{2}$ are defined in Condition~\ref{lowerRE}, and the last inequality follows from Lemma~\ref{eigenvaluelem}. Furthermore, by Condition~\ref{lowerRE}, Lemma~\ref{conelemma} and Lemma~\ref{eigenvaluelem}, we have 
\begin{equation}
\label{eq2thm2}
\frac{1}{N}\|X\Delta\|^{2}_{2}=\frac{1}{N}\|XD^{+}D\Delta\|_{2}^{2}\ge \eta_{\gamma}\| D\Delta \|_{2}^{2}\ge \eta_{\gamma}\|\Delta\|_{2}^{2}. 
\end{equation}
Therefore, combining $(\text{\ref{eq1thm2}})$ and $(\text{\ref{eq2thm2}})$, we have 
\begin{equation*}
\|\Delta\|_{2}\le \frac{3\lambda_{g}\sqrt{(2d)^{k+1}s_{1}}+3\lambda\sqrt{s_{2}}}{\eta_{\gamma}},
\end{equation*}
where $s_{1}=|S_{1}|$ and $s_{2}=|S_{2}|$. Furthermore, we have 
\begin{equation*}
\begin{split}
\frac{1}{N}\|X\Delta\|_{2}^{2}\le \frac{\left( 3\lambda_{g}\sqrt{(2d)^{k+1}s_{1}}+3\lambda\sqrt{s_{2}}  \right)^{2}}{\eta_{\gamma}}. 
\end{split}
\end{equation*}
Hence we obtain the result in Part (a). 

Next, we show Part (b). We have 
\begin{equation*}
D^{+}=(D^TD)^{-1}D^T=\begin{bmatrix}\frac{\lambda_{g}}{\lambda}\left(\frac{\lambda_{g}^2}{\lambda^2}L^{k+1}+I_{n}\right)^{-1}(\Delta^{(k+1)})^T & \left(\frac{\lambda_{g}^2}{\lambda^2}L^{k+1}+I_{n}\right)^{-1} \end{bmatrix}. 
\end{equation*}
So by the definition of matrix $\ell_{1}$ norm, we have 
\begin{equation*}
\begin{split}
\opnorm{D^{+}}_{1}=\max\left\{   \opnorm{\frac{\lambda_{g}}{\lambda}\left(\frac{\lambda_{g}^2}{\lambda^2}L^{k+1}+I_{n}\right)^{-1}(\Delta^{(k+1)})^T}_{1}, \opnorm{\left(\frac{\lambda_{g}^2}{\lambda^2}L^{k+1}+I_{n}\right)^{-1}}_{1} \right\} \\
\le \max\left\{   \opnorm{\left(\frac{\lambda_{g}^2}{\lambda^2}L^{k+1}+I_{n}\right)^{-1}}_{1}\opnorm{\frac{\lambda_{g}}{\lambda}(\Delta^{(k+1)})^T}_{1}, \opnorm{\left(\frac{\lambda_{g}^2}{\lambda^2}L^{k+1}+I_{n}\right)^{-1}}_{1} \right\}.
\end{split}
\end{equation*}
Therefore, for odd $k$, we have 
\begin{equation*}
\opnorm{\frac{\lambda_{g}}{\lambda}(\Delta^{(k+1)})^T}_{1}\le \frac{\lambda_{g}}{\lambda}\opnorm{L}^{\frac{k+1}{2}}_{1}\le \frac{\lambda_{g}}{\lambda}\left(\opnorm{M}_{1}+\opnorm{A}_{1} \right)^{\frac{k+1}{2}}\le \frac{\lambda_{g}}{\lambda}(2d)^{\frac{k+1}{2}}, 
\end{equation*}
where $M$ and $A$ are degree matrix and adjacency matrix of the underlying graph, respectively. For even $k$, we have
\begin{equation*}
\opnorm{\frac{\lambda_{g}}{\lambda}(\Delta^{(k+1)})^T}_{1}\le \frac{\lambda_{g}}{\lambda}\opnorm{L}^{\frac{k}{2}}_{1}\opnorm{F^T}_{1}\le \frac{2\lambda_{g}}{\lambda}(2d)^{\frac{k}{2}}, 
\end{equation*}
where $F$ is the oriented incidence matrix of the underlying graph. Furthermore, by Lemma~\ref{l1matrixnorm}, when $\lambda_{g}^2/\lambda^2<1/(2d)^{k+1}$, we have
\begin{equation*}
\opnorm{\left(\frac{\lambda_{g}^2}{\lambda^2}L^{k+1}+I_{n}\right)^{-1}}_{1} \le \frac{1}{1-\frac{\lambda_{g}^2}{\lambda^2}\opnorm{L}^{k+1}_{1}}\le \frac{1}{1-\frac{\lambda_{g}^2}{\lambda^2}\left(2d \right)^{k+1}}. 
\end{equation*}
Combining above analysis, we have 
\begin{equation*}
\opnorm{D^{+}}_{1}\le \frac{1}{1-\frac{\lambda_{g}^2}{\lambda^2}(2d)^{k+1}}.  
\end{equation*}
Therefore, 
\begin{equation*}
\begin{split}
\|\Delta\|_{1}=\|D^{+}D\Delta\|_{1}\le \opnorm{D^{+}}_{1}\|D\Delta\|_{1}\le \frac{4\| (D\Delta)_{S} \|_{1}}{1-\frac{\lambda_{g}^2}{\lambda^2}(2d)^{k+1}}\\
\le \frac{4}{1-\frac{\lambda_{g}^2}{\lambda^2}(2d)^{k+1}}\left(\frac{\lambda_{g}}{\lambda}\sqrt{(2d)^{k+1}s_{1}}+\sqrt{s_{2}}   \right)\|\Delta\|_{2}\\ \le \frac{12}{1-\frac{\lambda_{g}^2}{\lambda^2}(2d)^{k+1}}\frac{\left(  
\lambda_{g}\sqrt{(2d)^{k+1}s_{1}}+\lambda\sqrt{s_{2}}\right)^{2}}{\lambda\eta_{\gamma}}, 
\end{split}
\end{equation*}
which yields the result in Part (b). Therefore, the proof is complete.

\end{proof}


\subsection{Proof of Theorem~\ref{probconthm}}
\label{probconthmprf}

\begin{proof}

If we have
\begin{equation*}
\eta_{\gamma}=\frac{\lambda_{1}(\Sigma_{x})}{2(\nu+1)},~\lambda\asymp \sigma_{\varepsilon}\sqrt{\frac{\log n}{N}},~\lambda_{g}=\lambda\sqrt{\frac{\nu}{(2d)^{k+1}}} 
\end{equation*}
and $s_{2}/s_{1}\ge \nu$ for a constant $0\le \nu<1$, then
\begin{equation*}
\frac{3\lambda_{g}\sqrt{(2d)^{k+1}s_{1}}+3\lambda\sqrt{s_{2}}}{\eta_{\gamma}}\asymp \sigma_{\varepsilon}\sqrt{\frac{s_{2}\log n}{N}},
\end{equation*}
\begin{equation*}
\frac{\left(3\lambda_{g}\sqrt{(2d)^{k+1}s_{1}}+3\lambda\sqrt{s_{2}}\right)^2}{\eta_{\gamma}}\asymp \sigma_{\varepsilon}^2\frac{s_{2}\log n}{N}, 
\end{equation*}
and 
\begin{equation*}
\frac{12\lambda\left(\gamma\sqrt{(2d)^{k+1}s_{1}}+\sqrt{s_{2}}\right)^2}{\eta_{\gamma}\left(1-\gamma^2(2d)^{k+1}\right)}\asymp \sigma_{\varepsilon}s_{2}\sqrt{\frac{\log n}{N}}. 
\end{equation*}
Therefore, applying Theorem~\ref{fixedconthm}, Lemma~\ref{eta1rand} and Lemma~\ref{tuninglem}, we obtain the desired results in the theorem.

%

\end{proof}


\subsection{Proof of Theorem~\ref{weaklypiecewisepolycorollary}}
\label{weaklypiecewisepolycorollary-proof}

\begin{proof}
We first show Part (a). For $\eta_{1}>0$ and $\eta_{2}>0$, let 
\begin{equation*}
S_{\eta_{1}}=\left\{ i\in [m], \left|(\Delta_{i}^{(k+1)})^T\beta^*\right|>\eta_{1}\right\},\quad S_{\eta_{2}}=\left\{ i\in [n], \left|\beta^*_{i}\right|>\eta_{2}\right\}.  
\end{equation*}
Then we have $|S_{\eta_{1}}|\le R_{1}\eta_{1}^{-q_{1}}$ and $|S_{\eta_{2}}|\le R_{2}\eta_{2}^{-q_{2}}$. Furthermore, we have 
\begin{equation*}
\left\|\Delta_{S_{\eta_{1}}^c}^{(k+1)}\beta^*\right\|_{1}\le \left(\sum_{i\in S_{\eta_{1}}^c}|(\Delta_{i}^{(k+1)})^T\beta^* |^{q_{1}} \right)\eta_{1}^{1-q_{1}}\le R_{1}\eta_{1}^{1-q_{1}}, 
\end{equation*}
and 
\begin{equation*}
\left\|\beta^*_{S_{\eta_{2}}^{c}}\right\|_{1}\le \left(\sum_{i\in S_{\eta_{2}}^c}|\beta_{i}^* |^{q_{2}} \right)\eta_{2}^{1-q_{2}}\le R_{2}\eta_{2}^{1-q_{2}}. 
\end{equation*}
Therefore, letting $S=S_{\eta_{1}}\cup\left\{ m+i, i\in S_{\eta_{2}} \right\}$, we have
\begin{equation*}
|S|=|S_{\eta_{1}}|+|S_{\eta_{2}}|\le R_{1}\eta_{1}^{-q_{1}}+R_{2}\eta_{2}^{-q_{2}},
\end{equation*}
and 
\begin{equation*}
\|D_{S^c}\beta^*\|_{1}=\frac{\lambda_{g}}{\lambda}\left\|\Delta^{(k+1)}_{S_{\eta_{1}}^{c}}\beta^*\right\|_{1}+\left\|\beta^*_{S_{\eta_{2}}^{c}}\right\|_{1}\le \frac{\lambda_{g}}{\lambda}R_{1}\eta_{1}^{1-q_{1}}
+R_{2}\eta_{2}^{1-q_{2}}. 
\end{equation*}
Therefore, by Part (a) of Lemma~\ref{modelmisspecified-fixed}, we have 
\begin{equation*}
\|\hat{\beta}-\beta^*\|_{2}^{2}\le \frac{144\lambda^2\left(R_{1}\eta_{1}^{-q_{1}}+R_{2}\eta_{2}^{-q_{2}}\right)}{(\eta^{\prime}_{\gamma})^2}+\frac{32}{\eta^{\prime}_{\gamma}}\left[\lambda_{g}R_{1}\eta_{1}^{1-q_{1}}+\lambda R_{2}\eta_{2}^{1-q_{2}}+4\tau(N, n)\left(\frac{\lambda_{g}}{\lambda}R_{1}\eta_{1}^{1-q_{1}}
+R_{2}\eta_{2}^{1-q_{2}} \right)^2\right], 
\end{equation*}
and 
\begin{equation*}
\begin{split}
\frac{1}{N}\|X(\hat{\beta}-\beta^*)\|_{2}^{2}\le 4\lambda\left( \frac{\lambda_{g}}{\lambda}R_{1}\eta_{1}^{1-q_{1}}+R_{2}\eta_{2}^{1-q_{2}}\right)+3\lambda\sqrt{R_{1}\eta_{1}^{-q_{1}}+R_{2}\eta_{2}^{-q_{2}}}\times \\
\sqrt{\frac{144\lambda^2\left(R_{1}\eta_{1}^{-q_{1}}+R_{2}\eta_{2}^{-q_{2}}\right)}{(\eta^{\prime}_{\gamma})^2}+\frac{32}{\eta^{\prime}_{\gamma}}
\left[\lambda_{g}R_{1}\eta_{1}^{1-q_{1}}+\lambda R_{2}\eta_{2}^{1-q_{2}}+4\tau(N, n)\left(\frac{\lambda_{g}}{\lambda}R_{1}\eta_{1}^{1-q_{1}}
+R_{2}\eta_{2}^{1-q_{2}} \right)^2\right]}. 
\end{split}
\end{equation*}
Hence we obtain the results in Part (a). 

Next, we show Part (b). Applying Part (b) of Lemma~\ref{modelmisspecified-fixed}, we have 
\begin{equation*}
\begin{split}
\|\hat{\beta}-\beta^*\|^{2}_{1}\le 128\left(\frac{\lambda_{g}}{\lambda}R_{1}\eta_{1}^{1-q_{1}}+R_{2}\eta_{2}^{1-q_{2}}\right)^2+128\left(R_{1}\eta_{1}^{-q_{1}}+R_{2}\eta_{2}^{-q_{2}}\right) \times \\
\left\{\frac{144\lambda^2\left(R_{1}\eta_{1}^{-q_{1}}+R_{2}\eta_{2}^{-q_{2}}\right)}{(\eta^{\prime}_{\gamma})^2}+\frac{32}{\eta^{\prime}_{\gamma}}\left[\lambda_{g}R_{1}\eta_{1}^{1-q_{1}}+\lambda R_{2}\eta_{2}^{1-q_{2}}+4\tau(N, n)\left(\frac{\lambda_{g}}{\lambda}R_{1}\eta_{1}^{1-q_{1}}
+R_{2}\eta_{2}^{1-q_{2}} \right)^2\right]\right\}.
\end{split}
\end{equation*}
Therefore, the proof is complete. 

\end{proof}


\subsection{Proof of Theorem~\ref{weaklypiecewisepolyrandomthm}}
\label{weaklypiecewisepolyrandomthmprf}

\begin{proof}
We first show Part (a). Let $\eta_{1}=\eta_{2}=\sqrt{\frac{\log n}{N}}$. For $\lambda\asymp \sigma_{\varepsilon}\sqrt{\frac{\log n}{N}}$, $\lambda_{g}=2^{-(1+k/2)}d^{-(k+1)/2}\lambda$, $\tau(N, n)\asymp\log n/N$ and $\eta^{\prime}_{\gamma}=\frac{1}{3}\lambda_{1}(\Sigma_{x})$, we have 
\begin{equation*}
\RNum{1}=\frac{144\lambda^2\left(R_{1}\eta_{1}^{-q_{1}}+R_{2}\eta_{2}^{-q_{2}}\right)}{(\eta^{\prime}_{\gamma})^2}\asymp \sigma_{\varepsilon}^{2}\left( R_{1}\left(\frac{\log n}{N}  \right)^{1-\frac{q_1}{2}}+R_{2}\left( \frac{\log n}{N} \right)^{1-\frac{q_{2}}{2}}  \right).
\end{equation*}
Furthermore, if 
\begin{equation*}
R_{1}\left( \frac{\log n}{N} \right)^{1-\frac{q_1}{2}}+R_{2}\left( \frac{\log n}{N} \right)^{1-\frac{q_2}{2}}\lesssim 1,
\end{equation*}
then we have  
\begin{equation*}
\begin{split}
\RNum{2}=\frac{32}{\eta^{\prime}_{\gamma}}\left[\lambda_{g}R_{1}\eta_{1}^{1-q_{1}}+\lambda R_{2}\eta_{2}^{1-q_{2}}+4\tau(N, n)\left(\frac{\lambda_{g}}{\lambda}R_{1}\eta_{1}^{1-q_{1}}+R_{2}\eta_{2}^{1-q_{2}} \right)^2\right] \\
\asymp R_{1}\left( \frac{\log n}{N} \right)^{1-\frac{q_{1}}{2}}+R_{2}\left( \frac{\log n}{N}  \right)^{1-\frac{q_{2}}{2}}.
\end{split}
\end{equation*}
Therefore, by Part (a) of Theorem~\ref{weaklypiecewisepolycorollary}, we have 
\begin{equation*}
\|\hat{\beta}-\beta^*\|_{2}^{2}\le \mathcal{O}_{\mathbb{P}}\left( \sigma_{\varepsilon}^{2}\left( R_{1}\left(\frac{\log n}{N}  \right)^{1-\frac{q_1}{2}}+R_{2}\left( \frac{\log n}{N} \right)^{1-\frac{q_{2}}{2}}  \right)  \right),
\end{equation*}
and 
\begin{equation*}
\frac{1}{N}\|X(\hat{\beta}-\beta^*)\|_{2}^{2}\le  \mathcal{O}_{\mathbb{P}}\left( \sigma_{\varepsilon}^{2}\left( R_{1}\left(\frac{\log n}{N}  \right)^{1-\frac{q_1}{2}}+R_{2}\left( \frac{\log n}{N} \right)^{1-\frac{q_{2}}{2}}  \right)  \right). 
\end{equation*}
Hence we obtain the results in Part (a). 

A similar argument leads to the result in Part (b). Therefore, the proof is complete.

\end{proof}


\subsection{Proof of Theorem~\ref{onestepfixthm}}
\label{onestepfixthmprf}

\begin{proof}
By the definition of the one-step estimator and the linear model $(\ref{linear})$, we have
\begin{equation*}
\begin{split}
\sqrt{N}(\tilde{\beta}-\beta^*)=\sqrt{N}\left[\hat{\beta}-\beta^*+\frac{1}{N}\widehat{\Theta}X^TX(\beta^*-\hat{\beta})+\frac{1}{N}\widehat{\Theta}X^T\varepsilon\right]
=\Psi-e,
\end{split}
\end{equation*}
where
\begin{equation*}
\Psi=\frac{1}{\sqrt{N}}\widehat{\Theta}X^T\varepsilon\sim N(0, \sigma^{2}_{\varepsilon}\widehat{\Theta}\Sigma_{N}(\widehat{\Theta})^T),\quad e=\sqrt{N}(\widehat{\Theta}\Sigma_{N}-I_{n})(\hat{\beta}-\beta^*).
\end{equation*}
Furthermore, we have
\begin{equation*}
\begin{split}
\|e\|_{\infty}\le \sqrt{N}\|\widehat{\Theta}\Sigma_{N}-I_{n}\|_{\infty}\|\hat{\beta}-\beta^*\|_{1}\le \sqrt{N}\mu\|\hat{\beta}-\beta^*\|_{1}.
\end{split}
\end{equation*}
Therefore, the proof is complete. 

%


\end{proof}


\subsection{Proof of Theorem~\ref{onesteprandomthm}}
\label{onesteprandomthmprf}

We start with a supporting lemma which concerns the consistency of $\widehat{\Theta}$ obtained from the CLIME method. 

\begin{lemma}
\label{consistencyTheta}
If $\|\Theta_{x}\Sigma_{N}-I_{n}\|_{\infty}\le \mu$, then we have $\|\widehat{\Theta}-\Theta_{x}\|_{\infty}\le 2M_{n}\mu$. \footnote{The upper bound in Lemma~\ref{consistencyTheta} is sharper than the one in Theorem 4 of \cite{clime} if $M_{n}$ is allowed to increase as $n$ increases. } 

\end{lemma}

The proof of Lemma~\ref{consistencyTheta} is deferred to Appendix~\ref{consistencyThetaprf}. We now prove Theorem~\ref{onesteprandomthm} in the following.  

\begin{proof}
If 
\begin{equation*}
\lambda\asymp\sigma_{\varepsilon}\sqrt{\frac{\log n}{N}},\quad\mu\asymp\sqrt{\frac{\log n}{N}},\quad\gamma=\frac{\lambda_{g}}{\lambda}=\sqrt{\frac{\nu}{(2d)^{k+1}}},\quad\eta_{\gamma}=\frac{\lambda_{1}(\Sigma_{x})}{2(\nu+1)},  
\end{equation*}
and $s_{2}/s_{1}\ge \nu$ where $0\le \nu<1$, then by Theorem~\ref{probconthm} and Theorem~\ref{onestepfixthm}, we have
\begin{equation*}
\|e\|_{\infty}\le \mathcal{O}_{\mathbb{P}}\left(\frac{\sigma_{\varepsilon}s_{2}\log n}{\sqrt{N}}\right).  
\end{equation*}

Next, we bound $\|\widehat{\Theta}\Sigma_{N}(\widehat{\Theta})^T-\Theta_{x}\|_{\infty}$. We have
\begin{equation*}
\widehat{\Theta}\Sigma_{N}(\widehat{\Theta})^T-\Theta_{x}=(\widehat{\Theta}\Sigma_{N}-I_{n})(\widehat{\Theta})^T+(\widehat{\Theta})^T-\Theta_{x}. 
\end{equation*}
Therefore, 
\begin{equation*}
\begin{split}
\|\widehat{\Theta}\Sigma_{N}(\widehat{\Theta})^T-\Theta_{x}\|_{\infty}\le \|\widehat{\Theta}\Sigma_{N}-I_{n}\|_{\infty}\opnorm{\widehat{\Theta}}_{\infty}+\|\widehat{\Theta}-\Theta_{x}\|_{\infty}. 
\end{split}
\end{equation*}
We define the event $\mathcal{E}=\{ X: \|\Theta_{x}\Sigma_{N}-I_{n}\|_{\infty}\le \mu\}$. Then by Lemma~\ref{consistencyTheta}, on the event $\mathcal{E}$, we have 
\begin{equation*}
\| \widehat{\Theta}\Sigma_{N}(\widehat{\Theta})^T-\Theta_{x}\|_{\infty}\le 3M_{n}\mu. 
\end{equation*}
Finally, if $\mu\asymp \sqrt{\frac{\log n}{N}}$, then by Lemma~\ref{tuningmulemma}, we have
\begin{equation*}
\left\| \widehat{\Theta}\Sigma_{N}(\widehat{\Theta})^T-\Theta_{x}\right\|_{\infty}\le \mathcal{O}_{\mathbb{P}}\left(M_{n}\sqrt{\frac{\log n}{N}}\right). 
\end{equation*}
Hence the proof is complete.

\end{proof}


\section{Proofs of propositions}
\setcounter{equation}{0}
\renewcommand{\theequation}{B.\arabic{equation}}

In this section, we provide proofs of Proposition~\ref{structure} and Proposition~\ref{bayes} established in the paper. 

\subsection{Proof of Proposition~\ref{structure}}
\label{structureproof}

\begin{proof}
We first consider even $k$. Note we have  
\begin{equation*}
\Delta^{(k+1)}_{-\widehat{S}_{1}}=F_{-\widehat{S}_{1}}L^{\frac{k}{2}}. 
\end{equation*}
By Lemma~\ref{graph-oriented}, we have $\text{rank}\left(L^{k/2}\right)=n-1$, so $\text{null}\left(L^{k/2}\right)=\text{span}\left(\mathbbm{1}_{n}\right)$. Therefore, we have 
\begin{equation*}
\text{span}({\mathbbm{1}_{n}})\subset \text{null}\left(\Delta^{(k+1)}_{-\widehat{S}_{1}}\right). 
\end{equation*}
Thus 
\begin{equation*}
\text{null}\left(\Delta^{(k+1)}_{-\widehat{S}_{1}}\right)=\text{span}(\mathbbm{1}_{n})+\text{span}(\mathbbm{1}_{n})^{\perp}\cap \text{null}\left(\Delta^{(k+1)}_{-\widehat{S}_{1}}\right).
\end{equation*}
Furthermore, since $L^{k/2}+\mathbbm{1}_{n}\mathbbm{1}_{n}^T$ is positive definite, so 
\begin{equation*}
\left\{ (u, v)\in \mathbb{R}^{n}\times \mathbb{R}^{n}; \mathbbm{1}_{n}^Tu=0, v=L^{\frac{k}{2}}u \right\}=\left\{ (u, v)\in \mathbb{R}^{n}\times \mathbb{R}^{n};  \mathbbm{1}_{n}^Tu=0, u=\left(L^{\frac{k}{2}}+\mathbbm{1}_{n}\mathbbm{1}_{n}^T\right)^{-1}v \right\}. 
\end{equation*}
On the other hand, $\text{null}\left(F_{-\widehat{S}_{1}}\right)=\text{span}\left(\mathbbm{1}_{C_{1}}, ..., \mathbbm{1}_{C_{j}}\right)$. Therefore,  
\begin{equation*}
\text{span}(\mathbbm{1}_{n})^{\perp}\cap \text{null}\left(\Delta^{(k+1)}_{-\widehat{S}_{1}}\right)= \text{span}(\mathbbm{1}_{n})^{\perp}\cap\left(L^{\frac{k}{2}}+\mathbbm{1}_{n}\mathbbm{1}_{n}^T\right)^{-1}\text{span}\left(\mathbbm{1}_{C_{1}},...,\mathbbm{1}_{C_{j}} \right). 
\end{equation*}
Thus we have
\begin{equation*}
\text{null}\left(\Delta^{(k+1)}_{-\widehat{S}_{1}}\right)=\text{span}(\mathbbm{1}_{n})+ \text{span}(\mathbbm{1}_{n})^{\perp}\cap\left(L^{\frac{k}{2}}+\mathbbm{1}_{n}\mathbbm{1}_{n}^T\right)^{-1}\text{span}(\mathbbm{1}_{C_{1}},...,\mathbbm{1}_{C_{j}} ). 
\end{equation*}

Next, we consider the odd $k$. Using a similar argument as the even case, we have
\begin{equation*}
\begin{split}
\text{null}\left(\Delta^{(k+1)}_{-\widehat{S}_{1}}\right)=\text{span}(\mathbbm{1}_{n})+\text{span}(\mathbbm{1}_{n})^{\perp}\cap \text{null}\left(\Delta^{(k+1)}_{-\widehat{S}_{1}}\right) \\
= \text{span}(\mathbbm{1}_{n})+\text{span}(\mathbbm{1}_{n})^{\perp}\cap \left\{ u\in \mathbb{R}^{n}; u=\left(L^{\frac{k+1}{2}}+\mathbbm{1}_{n}\mathbbm{1}_{n}^T\right)^{-1}v, v_{-\widehat{S}_{1}}=0  \right\}. 
\end{split}
\end{equation*}
Thus we prove the proposition. 
\end{proof}


\subsection{Proof of Proposition~\ref{bayes}}
\label{bayesproof}

\begin{proof}
We have the basic identity 
\begin{equation*}
\frac{a}{2}\exp{(-a|z|)}=\int_{0}^{\infty}\frac{1}{\sqrt{2\pi t}}\exp{\left(-\frac{z^2}{2t}\right)}\frac{a^2}{2}\exp{\left(-\frac{a^2t}{2}\right)}dt, 
\end{equation*}
where $a>0$. Therefore, for $j\in [n]$ and $i\in [m]$, we have 
\begin{equation*}
\frac{\lambda_{1}}{2\sigma_{\varepsilon}}\exp{\left(-\frac{\lambda_{1}}{\sigma_{\varepsilon}}\left|\beta_{j}\right|  \right)}=\int_{0}^{\infty}\frac{1}{\sqrt{2\pi \sigma_{\varepsilon}^{2}\tau_{j}^2}}\exp{\left(-\frac{\beta_{j}^2}{2\sigma^2_{\varepsilon}\tau_{j}^2}\right)}\frac{\lambda_{1}^2}{2}\exp{\left(-\frac{\lambda_{1}^2\tau_{j}^2}{2}\right)}d\tau_{j}^{2},  
\end{equation*}
and 
\begin{equation*}
\frac{\lambda_{2}}{2\sigma_{\varepsilon}}\exp{\left(-\frac{\lambda_{2}}{\sigma_{\varepsilon}}\left|\left(\Delta_{i}^{(k+1)}\right)^T\beta\right|  \right)}=\int_{0}^{\infty}\frac{1}{\sqrt{2\pi \sigma_{\varepsilon}^2\omega_{i}^2}}\exp{\left(-\frac{((\Delta_{i}^{(k+1)})^T\beta)^2}{2\sigma_{\varepsilon}^2\omega_{i}^2}\right)}\frac{\lambda_{2}^2}{2}\exp{\left(-\frac{\lambda_{2}^2\omega_{i}^2}{2}\right)}d\omega_{i}^{2}. 
\end{equation*}
Hence, we have 
\begin{equation*}
\begin{split}
\pi(\beta)=\int_{0}^{\infty}\dots\int_{0}^{\infty} f\left(\beta|\tau_{1}^2,..., \tau_{n}^2, \omega_{1}^2,..., \omega_{m}^2\right)\pi\left(\tau_{1}^2,..., \tau_{n}^2, \omega_{1}^2,..., \omega_{m}^2\right)d\tau_{1}^2\dots d\omega_{m}^2 \\ 
\propto\exp{\left(-\frac{\lambda_{1}}{\sigma_{\varepsilon}}\|\beta\|_{1}-\frac{\lambda_{2}}{\sigma_{\varepsilon}}\|\Delta^{(k+1)}\beta\|_{1}\right)}.
\end{split}
\end{equation*}
Therefore, we have 
\begin{equation*}
\begin{split}
f(\beta|X, y)\propto f(y|X, \beta)\pi(\beta)\propto \exp{\left(-\frac{1}{2\sigma_{\varepsilon}^2}\|y-X\beta\|_{2}^{2}-\frac{\lambda_{1}}{\sigma_{\varepsilon}}\|\beta\|_{1}-\frac{\lambda_{2}}{\sigma_{\varepsilon}}\|\Delta^{(k+1)}\beta\|_{1}\right)} \\
=\exp{\left(-\frac{N}{\sigma_{\varepsilon}^2}\left(\frac{1}{2N}\|y-X\beta\|_{2}^{2}+\lambda\|\beta\|_{1}+\lambda_{g}\|\Delta^{(k+1)}\beta\|_{1}\right)\right)}.  
\end{split}
\end{equation*}
Hence we obtain the desired result in the proposition.

\end{proof}


\section{Proofs of corollaries}
\setcounter{equation}{0}
\renewcommand{\theequation}{C.\arabic{equation}}

In this section, we provide proofs of Corollary~\ref{deltaknormal}, Corollary~\ref{CI}, Corollary~\ref{CI2} and Corollary~\ref{CI3} established in the paper.


\subsection{Proof of Corollary~\ref{deltaknormal}}
\label{deltaknormalproof}

\begin{proof}
By Theorem~\ref{onestepfixthm}, we have 
\begin{equation*}
\sqrt{N}(\Delta^{(k+1)}\tilde{\beta}-\Delta^{(k+1)}\beta^*)=\Psi^{(k+1)}-e^{(k+1)},
\end{equation*}
where 
\begin{equation*}
\Psi^{(k+1)}=\frac{1}{\sqrt{N}}\Delta^{(k+1)}\widehat{\Theta}X^T\varepsilon,\quad e^{(k+1)}=\sqrt{N}\Delta^{(k+1)}(\widehat{\Theta}\Sigma_{N}-I_{n})(\hat{\beta}-\beta^*). 
\end{equation*}
Therefore, we have
\begin{equation*}
\Psi^{(k+1)}|X\sim N(0, \sigma^2_{\varepsilon}\Delta^{(k+1)}\widehat{\Theta}\Sigma_{N}(\widehat{\Theta})^T(\Delta^{(k+1)})^T),
\end{equation*}
and 
\begin{equation*}
\|e^{(k+1)}\|_{\infty}\le \sqrt{N}\opnorm{\Delta^{(k+1)}}_{\infty}\left\|(\widehat{\Theta}\Sigma_{N}-I)(\hat{\beta}-\beta^*)\right\|_{\infty}\le\mathcal{O}_{\mathbb{P}}\left( \sigma_{\varepsilon}(2d)^{\frac{k+1}{2}}\frac{s_{2}\log n}{\sqrt{N}}\right).
\end{equation*}
Furthermore, we have   
\begin{equation*}
\begin{split}
\| \Delta^{(k+1)}\widehat{\Theta}\Sigma_{N}(\widehat{\Theta})^T(\Delta^{(k+1)})^T-\Delta^{(k+1)}\Theta_{x}(\Delta^{(k+1)})^T\|_{\infty} \\
\le \opnorm{\Delta^{(k+1)}}_{\infty}^{2}\|\widehat{\Theta}\Sigma_{N}(\widehat{\Theta})^T-\Theta_{x}\|_{\infty} \le \mathcal{O}_{\mathbb{P}}\left( (2d)^{k+1}M_{n}\sqrt{\frac{\log n}{N}}  \right).  
\end{split}
\end{equation*}
Therefore, the proof is complete. 
\end{proof}


\subsection{Proof of Corollary~\ref{CI}}
\label{CIprf}

\begin{proof}
By Theorem~\ref{onesteprandomthm}, for $j\in [n]$, we have
\begin{equation*}
\frac{\sqrt{N}(\tilde{\beta}_{j}-\beta^*_{j})}{\sigma_{\varepsilon}\sqrt{e_{j}^T\widehat{\Theta}\Sigma_{N}\widehat{\Theta}^Te_{j}}}=\frac{\frac{1}{\sqrt{N}}e_{j}^T\widehat{\Theta}X^T\varepsilon}{\sigma_{\varepsilon}\sqrt{e_{j}^T\widehat{\Theta}\Sigma_{N}\widehat{\Theta}^Te_{j}}}+\frac{\sqrt{N}e_{j}^T(\widehat{\Theta}\Sigma_{N}-I_{n})(\beta^*-\hat{\beta})}{\sigma_{\varepsilon}\sqrt{e_{j}^T\widehat{\Theta}\Sigma_{N}\widehat{\Theta}^Te_{j}}}. 
\end{equation*}

First, we claim that $Z=\frac{\frac{1}{\sqrt{N}}e_{j}^T\widehat{\Theta}X^T\varepsilon}{\sigma_{\varepsilon}\sqrt{e_{j}^T\widehat{\Theta}\Sigma_{N}\widehat{\Theta}^Te_{j}}}\sim N(0, 1)$.  
To see this, the characteristic function of $Z$ is 
\begin{equation*}
\begin{split}
\mathbb{E}\left(e^{itZ}\right)=\mathbb{E}\left[\exp{\left(it\frac{\frac{1}{\sqrt{N}}e_{j}^T\widehat{\Theta}X^T\varepsilon}{\sigma_{\varepsilon}\sqrt{e_{j}^T\widehat{\Theta}\Sigma_{N}\widehat{\Theta}^Te_{j}}}\right)}\right]=\mathbb{E}_{X}\left[\mathbb{E}_{\varepsilon}\left[ \exp{\left(it\frac{\frac{1}{\sqrt{N}}e_{j}^T\widehat{\Theta}X^T\varepsilon}{\sigma_{\varepsilon}\sqrt{e_{j}^T\widehat{\Theta}\Sigma_{N}\widehat{\Theta}^Te_{j}}}\right)}\Bigg|X  \right]\right] \\
=\mathbb{E}_{X}\left[\exp{\left(-\frac{t^2}{2}\right)}\Bigg| X  \right]=\exp{\left( -\frac{t^2}{2} \right)}. 
\end{split}
\end{equation*}
Hence we have $Z\sim N(0, 1)$. 

Next, for $x\in \mathbb{R}$ and $\delta=\frac{s_{2}\log n}{\sqrt{N}}$, we have 
\begin{equation*}
\begin{split}
\mathbb{P}\left(\frac{\sqrt{N}(\tilde{\beta}_{j}-\beta^*_{j})}{\sigma_{\varepsilon}\sqrt{e_{j}^T\widehat{\Theta}\Sigma_{N}\widehat{\Theta}^Te_{j}}}\le x  \right)= \mathbb{P}\left(Z+\frac{\sqrt{N}e_{j}^T(\widehat{\Theta}\Sigma_{N}-I_{n})(\beta^*-\hat{\beta})}{\sigma_{\varepsilon}\sqrt{e_{j}^T\widehat{\Theta}\Sigma_{N}\widehat{\Theta}^Te_{j}}}\le x   \right) \\
\le \mathbb{P}\left(Z\le x+\delta\right)+\mathbb{P}\left( \frac{\sqrt{N}|e_{j}^T(\widehat{\Theta}\Sigma_{N}-I_{n})(\beta^*-\hat{\beta})|}{\sigma_{\varepsilon}\sqrt{e_{j}^T\widehat{\Theta}\Sigma_{N}\widehat{\Theta}^Te_{j}}}\ge \delta  \right) \\
\le \Phi(x+\delta)+\mathbb{P}\left(\sqrt{N}|e_{j}^T(\widehat{\Theta}\Sigma_{N}-I_{n})(\beta^*-\hat{\beta})|\ge \sigma_{\varepsilon}\sqrt{\frac{1}{2}\lambda_{1}(\Theta_{x})}\delta  \right)+\mathbb{P}\left(e_{j}^T\widehat{\Theta}\Sigma_{N}\widehat{\Theta}^Te_{j}\le \frac{1}{2}\lambda_{1}(\Theta_{x})  \right),
\end{split}
\end{equation*}
where $\Phi(x)$ is the cumulative distribution function of $N(0, 1)$. Similarly, we have 
\begin{equation*}
\begin{split}
\mathbb{P}\left(\frac{\sqrt{N}(\tilde{\beta}_{j}-\beta^*_{j})}{\sigma_{\varepsilon}\sqrt{e_{j}^T\widehat{\Theta}\Sigma_{N}\widehat{\Theta}^Te_{j}}}\le x  \right)= \mathbb{P}\left(Z+\frac{\sqrt{N}e_{j}^T(\widehat{\Theta}\Sigma_{N}-I_{n})(\beta^*-\hat{\beta})}{\sigma_{\varepsilon}\sqrt{e_{j}^T\widehat{\Theta}\Sigma_{N}\widehat{\Theta}^Te_{j}}}\le x   \right) \\
\ge \mathbb{P}\left(Z\le x-\delta\right)-\mathbb{P}\left( \frac{\sqrt{N}|e_{j}^T(\widehat{\Theta}\Sigma_{N}-I_{n})(\beta^*-\hat{\beta})|}{\sigma_{\varepsilon}\sqrt{e_{j}^T\widehat{\Theta}\Sigma_{N}\widehat{\Theta}^Te_{j}}}\ge \delta  \right) \\
\ge \Phi(x-\delta)-\mathbb{P}\left(\sqrt{N}|e_{j}^T(\widehat{\Theta}\Sigma_{N}-I_{n})(\beta^*-\hat{\beta})|\ge \sigma_{\varepsilon}\sqrt{\frac{1}{2}\lambda_{1}(\Theta_{x})}\delta  \right)-\mathbb{P}\left(e_{j}^T\widehat{\Theta}\Sigma_{N}\widehat{\Theta}^Te_{j}\le \frac{1}{2}\lambda_{1}(\Theta_{x})  \right). 
\end{split}
\end{equation*}
Furthermore, using similar arguments as the proof of Theorem~\ref{onesteprandomthm}, we have 
\begin{equation*}
\mathbb{P}\left(\sqrt{N}|e_{j}^T(\widehat{\Theta}\Sigma_{N}-I_{n})(\beta^*-\hat{\beta})|\ge \sigma_{\varepsilon}\sqrt{\frac{1}{2}\lambda_{1}(\Theta_{x})}\delta  \right)\le c\exp{(c's_{2}\log n-c''N)+2\exp{(-\log n)}},  
\end{equation*}
where $c>0$, $c'>0$ and $c''>0$ are constants, and 
\begin{equation*}
\mathbb{P}\left( e_{j}^T\widehat{\Theta}\Sigma_{N}\widehat{\Theta}^Te_{j} \le \frac{1}{2}\lambda_{1}(\Theta_{x})\right)\le 2\exp{(-\log n)}, 
\end{equation*}
provided that $M_{n}\sqrt{\frac{\log n}{N}}\rightarrow 0$. Therefore, we have 
\begin{equation*}
\begin{split}
\mathbb{P}\left(\frac{\sqrt{N}(\tilde{\beta}_{j}-\beta^*_{j})}{\sigma_{\varepsilon}\sqrt{e_{j}^T\widehat{\Theta}\Sigma_{N}\widehat{\Theta}^Te_{j}}}\le x  \right)
\le \Phi(x+\delta)+c\exp{(c' s_{2}\log n-c''N)}+4\exp{(-\log n)}, 
\end{split}
\end{equation*}
and 
\begin{equation*}
\begin{split}
\mathbb{P}\left(\frac{\sqrt{N}(\tilde{\beta}_{j}-\beta^*_{j})}{\sigma_{\varepsilon}\sqrt{e_{j}^T\widehat{\Theta}\Sigma_{N}\widehat{\Theta}^Te_{j}}}\le x  \right)
\ge \Phi(x+\delta)-c\exp{(c's_{2}\log n-c''N)}-4\exp{(-\log n)}. 
\end{split}
\end{equation*}
Hence, combining the above analysis, if $M_{n}\sqrt{\frac{\log n}{N}}\rightarrow 0$ and $\delta\rightarrow 0$, then we have 
\begin{equation*}
\mathbb{P}\left(\frac{\sqrt{N}(\tilde{\beta}_{j}-\beta^*_{j})}{\sigma_{\varepsilon}\sqrt{e_{j}^T\widehat{\Theta}\Sigma_{N}\widehat{\Theta}^Te_{j}}}\le x  \right)\rightarrow \Phi(x), 
\end{equation*}
implying the result in the corollary. 

\end{proof}


\subsection{Proof of Corollary~\ref{CI2}}
\label{CI2prf}
We start with a supporting lemma which concerns the consistency of $\hat{\sigma}_{\varepsilon}$ defined in $(\ref{sigmahat})$. 

\begin{lemma}
\label{sigmaconsis}
Under the same conditions with Theorem~\ref{probconthm}, we have 
\begin{equation*}
\left| \frac{\hat{\sigma}_{\varepsilon}}{\sigma_{\varepsilon}}-1  \right|\le \mathcal{O}_{\mathbb{P}}\left(\frac{s_{2}\log n}{N}+\sqrt{\frac{\log N}{N}}\right). 
\end{equation*}


\end{lemma}

The proof of Lemma~\ref{sigmaconsis} is deferred to Appendix~\ref{sigmaconsisprf}. We now prove Corollary~\ref{CI2} in the following. 

\begin{proof}
Using a similar argument to the proof of Corollary~\ref{CI}, for $x\in \mathbb{R}$, $0<\delta$ and $0<\zeta<1$, we have 
\begin{equation*}
\begin{split}
\mathbb{P}\left(\frac{\sqrt{N}(\tilde{\beta}_{j}-\beta^*_{j})}{\hat{\sigma}_{\varepsilon}\sqrt{e_{j}^T\widehat{\Theta}\Sigma_{N}\widehat{\Theta}^Te_{j}}}\le \frac{x}{1+\zeta}  \right)= \mathbb{P}\left(\frac{\sigma_{\varepsilon}}{\hat{\sigma}_{\varepsilon}}Z+\frac{\sqrt{N}e_{j}^T(\widehat{\Theta}\Sigma_{N}-I_{n})(\beta^*-\hat{\beta})}{\hat{\sigma}_{\varepsilon}\sqrt{e_{j}^T\widehat{\Theta}\Sigma_{N}\widehat{\Theta}^Te_{j}}}\le \frac{x}{1+\zeta}   \right) \\
\le \mathbb{P}\left(\frac{\sigma_{\varepsilon}}{\hat{\sigma}_{\varepsilon}}Z\le \frac{x+\delta}{1+\zeta}\right)+\mathbb{P}\left( \frac{\sqrt{N}|e_{j}^T(\widehat{\Theta}\Sigma_{N}-I_{n})(\beta^*-\hat{\beta})|}{\hat{\sigma}_{\varepsilon}\sqrt{e_{j}^T\widehat{\Theta}\Sigma_{N}\widehat{\Theta}^Te_{j}}}\ge \frac{\delta}{1+\zeta}  \right) \\
\le \mathbb{P}\left(Z\le x+\delta\right)+\mathbb{P}\left(\left|\frac{\hat{\sigma}_{\varepsilon}}{\sigma_{\varepsilon}}-1\right|\ge \zeta   \right)
+\mathbb{P}\left(\sqrt{N}|e_{j}^T(\widehat{\Theta}\Sigma_{N}-I_{n})(\beta^*-\hat{\beta})|\ge \frac{\sigma_{\varepsilon}}{2}\sqrt{\frac{1}{2}\lambda_{1}(\Theta_{x})}\frac{\delta}{1+\zeta}  \right) \\
+\mathbb{P}\left( \frac{\hat{\sigma}_{\varepsilon}}{\sigma_{\varepsilon}}\le\frac{1}{2}\right)+\mathbb{P}\left(e_{j}^T\widehat{\Theta}\Sigma_{N}\widehat{\Theta}^Te_{j}\le \frac{1}{2}\lambda_{1}(\Theta_{x})  \right),
\end{split}
\end{equation*}
and 
\begin{equation*}
\begin{split}
\mathbb{P}\left(\frac{\sqrt{N}(\tilde{\beta}_{j}-\beta^*_{j})}{\hat{\sigma}_{\varepsilon}\sqrt{e_{j}^T\widehat{\Theta}\Sigma_{N}\widehat{\Theta}^Te_{j}}}\le \frac{x}{1+\zeta}  \right)= \mathbb{P}\left(\frac{\sigma_{\varepsilon}}{\hat{\sigma}_{\varepsilon}}Z+\frac{\sqrt{N}e_{j}^T(\widehat{\Theta}\Sigma_{N}-I_{n})(\beta^*-\hat{\beta})}{\hat{\sigma}_{\varepsilon}\sqrt{e_{j}^T\widehat{\Theta}\Sigma_{N}\widehat{\Theta}^Te_{j}}}\le \frac{x}{1+\zeta}   \right) \\
\ge \mathbb{P}\left(\frac{\sigma_{\varepsilon}}{\hat{\sigma}_{\varepsilon}}Z\le \frac{x-\delta}{1+\zeta}\right)-\mathbb{P}\left( \frac{\sqrt{N}|e_{j}^T(\widehat{\Theta}\Sigma_{N}-I_{n})(\beta^*-\hat{\beta})|}{\hat{\sigma}_{\varepsilon}\sqrt{e_{j}^T\widehat{\Theta}\Sigma_{N}\widehat{\Theta}^Te_{j}}}\ge \frac{\delta}{1+\zeta}  \right) \\
\ge \mathbb{P}\left(Z\le \frac{(x-\delta)(1-\zeta)}{1+\zeta} \right)-\mathbb{P}\left( \left| \frac{\hat{\sigma}_{\varepsilon}}{\sigma_{\varepsilon}}-1  \right|\ge \zeta  \right)-\mathbb{P}\left(\sqrt{N}|e_{j}^T(\widehat{\Theta}\Sigma_{N}-I_{n})(\beta^*-\hat{\beta})|\ge \frac{\sigma_{\varepsilon}}{2}\sqrt{\frac{1}{2}\lambda_{1}(\Theta_{x})}\frac{\delta}{1+\zeta}  \right) \\
-\mathbb{P}\left(\frac{\hat{\sigma}_{\varepsilon}}{\sigma_{\varepsilon}}\le \frac{1}{2}\right)-\mathbb{P}\left(e_{j}^T\widehat{\Theta}\Sigma_{N}\widehat{\Theta}^Te_{j}\le \frac{1}{2}\lambda_{1}(\Theta_{x})  \right). 
\end{split}
\end{equation*}
Letting $\delta=\frac{s_{2}\log n}{N}$ and $\gamma=\frac{s_{2}\log n}{N}+\sqrt{\frac{\log N}{N}}$, and applying Lemma~\ref{sigmaconsis} and Theorem~\ref{onesteprandomthm}, under conditions given in the corollary, we have 
\begin{equation*}
\mathbb{P}\left(\frac{\sqrt{N}(\tilde{\beta}_{j}-\beta^*_{j})}{\hat{\sigma}_{\varepsilon}\sqrt{e_{j}^T\widehat{\Theta}\Sigma_{N}\widehat{\Theta}^Te_{j}}}\le x  \right)\rightarrow \Phi(x). 
\end{equation*}
Therefore, we prove the result in the corollary. 

\end{proof}


\subsection{Proof of Corollary~\ref{CI3}}
\label{CI3proof}

\begin{proof}
Using similar arguments as the proof of Corollary~\ref{CI} and Corollary~\ref{CI2} and applying Corollary~\ref{deltaknormal}, if $dM_{n}\sqrt{\frac{\log n}{N}}\rightarrow 0$ and $s_{2}\log n\sqrt{\frac{d}{N}}\rightarrow 0$, then for $x\in \mathbb{R}$, we have 
\begin{equation*}
\mathbb{P}\left(\frac{\sqrt{N}(F_{j}^T\tilde{\beta}-F_{j}^T\beta^*)}{\sigma_{\varepsilon}\sqrt{F_{j}^T\widehat{\Theta}\Sigma_{N}\widehat{\Theta}^TF_{j}}}\le x  \right)\rightarrow \Phi(x),
\end{equation*}
and
\begin{equation*}
\mathbb{P}\left(\frac{\sqrt{N}(F_{j}^T\tilde{\beta}-F_{j}^T\beta^*)}{\hat{\sigma}_{\varepsilon}\sqrt{F_{j}^T\widehat{\Theta}\Sigma_{N}\widehat{\Theta}^TF_{j}}}\le x  \right)\rightarrow \Phi(x),
\end{equation*}
implying the results in the corollary.


\end{proof}


\section{Proofs of lemmas}
\setcounter{equation}{0}
\renewcommand{\theequation}{D.\arabic{equation}}

In this section, we provide proofs of various lemmas established in the paper. 

\subsection{Proof of Lemma~\ref{eta1rand}}
\label{eta1randprf}

\begin{proof}
Let $\mathbb{L}_{0}(s)=\{v\in \mathbb{R}^{m+n}: v\in\mathbb{B}_{0}(s)\cap \mathbb{B}_{2}(1)\}$ and 
$\mathbb{L}_{1}(s)=\{v\in \mathbb{R}^{m+n}: \|v\|_{1}\le 4\sqrt{s}\|v\|_{2} \}$. For $v\in \mathbb{R}^{m+n}$, we have
\begin{equation*}
v^T(D^{+})^T\Sigma_{x}D^{+}v\ge \lambda_{1}(\Sigma_{x})\sigma^{2}_{\min}(D^{+})\|v\|_{2}^{2}=\frac{\lambda_{1}(\Sigma_{x})}{\sigma^{2}_{\max}(D)}\|v\|_{2}^{2}\ge \frac{\lambda_{1}(\Sigma_{x})}{\nu+1}\|v\|_{2}^{2},  
\end{equation*}
where the last inequality follows from Lemma~\ref{eigenvaluelem} and $\gamma^2=\left(\lambda_{g}/\lambda\right)^2=\nu(2d)^{-(k+1)}$. Let $\delta=\lambda_{1}(\Sigma_{x})(\nu+1)^{-1}$ and $s=s_{1}+s_{2}$. Then it boils down to proving 
\begin{equation}
\label{lowerREeq}
|v^T(D^{+})^T(\frac{X^TX}{N}-\Sigma_{x})D^{+}v|\le \frac{\delta}{150},\quad\forall~v\in \mathbb{L}_{0}(2s)
\end{equation}
with high probability. Indeed, if $\text{(\ref{lowerREeq})}$ holds, then by Part (a) of Lemma~\ref{REgeneral}, we have
\begin{equation*}
\frac{1}{N}\|XD^{+}v\|_{2}^{2} \ge \frac{\delta}{2}\|v\|_{2}^2,\quad\forall~v\in \mathbb{L}_{1}(s). 
\end{equation*}
Therefore, letting $\eta_{\gamma}=\frac{1}{2}\lambda_{1}(\Sigma_{x})\left(\nu+1\right)^{-1}$ yields the desired restricted eigenvalue condition. 

Next, we use a discretization argument to show $\text{(\ref{lowerREeq})}$. For an index subset $I\subset [m+n]$ with $|I|\le 2s$, we define 
\begin{equation*}
S_{I}=\left\{q\in \real^{m+n}; \|v\|_{2}\le 1, \text{support}(v)\subset I\right\}, 
\end{equation*}
then $\mathbb{L}_{0}(2s)=\cup_{|I|\le 2s}S_{I}$. For a fixed $S_{I}$, let $\mathcal{A}_{I}$ be a $\frac{1}{3}$-cover of $S_{I}$ where $|\mathcal{A}_{I}|\le 9^{2s}$. Then for $v\in S_{I}$,  there exists $a_{v}\in \mathcal{A}_{I}$, such that $\|\Delta_{v}\|_{2}=\|v-a_{v}\|_{2}\le \frac{1}{3}$. We also write 
\begin{equation*}
M=(D^{+})^T(\frac{X^TX}{N}-\Sigma_{x})D^{+}
\end{equation*}
and $\Psi(v)=v^TMv$. For $v\in S_{I}$, we have
\begin{equation*}
\begin{split}
    |\Psi(v)|=|(\Delta_{v}+a_{v})^{T}M(\Delta_{v}+a_{v})|\le |\Delta_{v}^{T}M\Delta_{v}|+2|a_{v}^{T}M\Delta_{v}|+|a_{v}^{T}Ma_{v}| \\
    \le \frac{1}{9}\sup_{v\in S_{I}}|\Psi(v)|+\frac{2}{3} \sup_{v\in S_{I}}|\Psi(v)|+\sup_{v\in \mathcal{A}_I}|\Psi(v)|= \frac{7}{9}\sup_{v\in S_{I}}|\Psi(v)|+\sup_{v\in \mathcal{A}_I}|\Psi(v)|. 
\end{split}
\end{equation*}
Therefore, we have 
\begin{equation*}
    \sup_{v\in S_{I}}|\Psi(v)|\le \frac{9}{2}\sup_{v\in\mathcal{A}_I}|\Psi(v)|.
\end{equation*}
Applying Lemma~\ref{LemBern}, taking union bounds and letting $t=\frac{\delta}{675}$ , we have
\begin{equation*}
\begin{split}
\mathbb{P}\left(\sup_{v\in \mathbb{L}_{0}(2s)}\left|\Psi(v)\right|\ge \frac{\delta}{150} \right)\le 2\exp{\left(cs\log n-c^{\prime}N\right)},
\end{split}
\end{equation*}
which yields the result in the lemma. Hence the proof is complete.

\end{proof}


\subsection{Proof of Lemma~\ref{tuninglem}}
\label{tuninglemprf}

\begin{proof}




For $i\in [m+n]$, we define the event $\mathcal{E}_{i}=\left\{\|XD^{+}e_{i}\|_{2}^2\le 2\lambda_{n}(\Sigma_{x})N\right\}$. Then by Lemma~\ref{LemBern}, we have 
\begin{equation*}
\mathbb{P}\left(\mathcal{E}_{i}\right)\ge 1-\exp{\left(-cN\right)}, 
\end{equation*}
where $c>0$ is a constant. 
Conditioning on $\mathcal{E}_{i}$, for $t>0$, we have 
\begin{equation*}
    \mathbb{P}\left(|\varepsilon^TXD^{+}e_{i}|>t \mid   \mathcal{E}_{i}\right)\le 2\exp{\left[-\frac{t^{2}}{4N\sigma_{\varepsilon}^{2}\lambda_{n}(\Sigma_{x}) } \right] }.
\end{equation*}
Therefore, 
\begin{equation*}
    \mathbb{P}\left(|\varepsilon^TXD^{+}e_{i}|>t  \right)\le 2\exp{\left[-\frac{t^{2}}{4N\sigma_{\varepsilon}^{2}\lambda_{n}(\Sigma_{x}) } \right] }+\exp{(-cN)}. 
\end{equation*}
Finally, applying a union bound and letting $t=2\sqrt{2}\sigma_{\varepsilon}\sqrt{\lambda_{n}(\Sigma_{x})}\sqrt{N\log (m+n)}$ yield
\begin{equation*}
    \mathbb{P}\left(\|\varepsilon^{T}XD^{+}\|_{\infty}>t\right)\le \frac{2}{m+n}+\exp{\left[\log (m+n)-cN\right]}.
\end{equation*}
Therefore, letting $\lambda\asymp\sigma_{\varepsilon}\sqrt{\frac{\log n}{N}}$, we have
\begin{equation*}
\mathbb{P}\left(\lambda\ge \frac{2}{N}\|\varepsilon^TXD^{+}\|_{\infty}\right)\ge 1-2\exp{(-\log n)}-\exp{\left(c\log n-c^{\prime}N\right)}. 
\end{equation*}
Hence the proof is complete.



\end{proof}


\subsection{Proof of Lemma~\ref{modelmisspecified-fixed}}
\label{modelmisspecified-fixed-proof}

\begin{proof}
We first show Part (a). For a general vector $\beta^*\in \mathbb{R}^{n}$ and any subset $S\subset [m+n]$ with $|S|\le \frac{\eta_{\gamma}^{\prime}}{64\tau(N, n)}$, applying similar arguments as the proof of Theorem~\ref{fixedconthm}, when $\lambda\ge \frac{2}{N}\|\varepsilon^TXD^{+}\|_{\infty}$, we have 
\begin{equation}
\label{thm7-eq1}
0\le \frac{1}{2N}\|X\Delta\|_{2}^{2}\le \frac{\lambda}{2}\left(4\|D_{S^c}\beta^{*}\|_{1}+3\|D_{S}\Delta\|_{1}-\|D_{S^c}\Delta\|_{1}\right).
\end{equation}
Hence,  
\begin{equation*}
\|D_{S^c}\Delta\|_{1}\le 4\|D_{S^c}\beta^*\|_{1}+3\|D_{S}\Delta\|_{1}, 
\end{equation*}
implying that 
\begin{equation}
\label{thm7-eq2}
\begin{split}
\|D\Delta\|_{1}^{2}\le (4\|D_{S^c}\beta^*\|_{1}+4\|D_{S}\Delta\|_{1})^2  \\ \le (4\| D_{S^c}\beta^* \|_{1}+4\sqrt{|S|}\| D\Delta \|_{2})^2
\le 32\|D_{S^c}\beta^*\|_{1}^{2}+32|S|\|D\Delta\|_{2}^{2}. 
\end{split}
\end{equation}
Combining Condition~\ref{lowerRE2} and $\text{(\ref{thm7-eq1})}$, we have
\begin{equation*}
\eta^{\prime}_{\gamma}\| D\Delta \|_{2}^{2}-\tau(N, n)\| D\Delta \|_{1}^{2}\le \frac{1}{N}\| X\Delta \|_{2}^{2}\le \lambda\left(4\|D_{S^c}\beta^*\|_{1}+3\sqrt{|S|}\|D\Delta\|_{2}\right). 
\end{equation*}
Therefore, by $\text{(\ref{thm7-eq2})}$, we obtain 
\begin{equation*}
\left(\eta^{\prime}_{\gamma}-32\tau(N, n)|S|\right)\|D\Delta\|_{2}^{2}-32\tau(N, n)\| D_{S^c}\beta^*\|_{1}^{2}\le \lambda\left(4\|D_{S^c}\beta^*\|_{1}+3\sqrt{|S|}\|D\Delta\|_{2}\right). 
\end{equation*}
Hence,  
\begin{equation}
\label{thm7-eq3}
\frac{\eta^{\prime}_{\gamma}}{2}\|D\Delta\|_{2}^{2}-32\tau(N, n)\| D_{S^c}\beta^*\|_{1}^{2}\le \lambda\left(4\|D_{S^c}\beta^*\|_{1}+3\sqrt{|S|}\|D\Delta\|_{2}\right). 
\end{equation}
Then we split the remainder of the analysis into two cases. In the first case, we suppose 
\begin{equation*}
\frac{\eta^{\prime}_{\gamma}}{4}\|D\Delta\|_{2}^{2}\ge 32\tau(N, n)\|D_{S^c}\beta^*\|_{1}^{2}.  
\end{equation*}
Then by $\text{(\ref{thm7-eq3})}$, we have 
\begin{equation*}
\frac{\eta^{\prime}_{\gamma}}{4}\|D\Delta\|_{2}^{2}\le \lambda\left(4\|D_{S^c}\beta^*\|_{1}+3\sqrt{|S|}\|D\Delta\|_{2}\right). 
\end{equation*}
Using Young's inequality, we have 
\begin{equation}
\label{thm7-eq4}
\|D\Delta\|_{2}^{2}\le \frac{144\lambda^2|S|}{(\eta^{\prime}_{\gamma})^2}+\frac{32\lambda}{\eta^{\prime}_{\gamma}}\|D_{S^c}\beta^*\|_{1}. 
\end{equation}
In the second case, we have 
\begin{equation*}
\frac{\eta^{\prime}_{\gamma}}{4}\|D\Delta\|_{2}^{2}<32\tau(N, n)\| D_{S^c}\beta^{*}\|_{1}^{2}, 
\end{equation*}
implying that 
\begin{equation}
\label{thm7-eq5}
\|D\Delta\|_{2}^{2}\le \frac{128\tau(N, n)}{\eta^{\prime}_{\gamma}}\|D_{S^c}\beta^*\|_{1}^{2}. 
\end{equation}
Taking into account both cases, we combine $\text{(\ref{thm7-eq5})}$ with the earlier inequality $\text{(\ref{thm7-eq4})}$, then obtain
\begin{equation*}
\|D\Delta\|_{2}^{2}\le \frac{144\lambda^2|S|}{(\eta^{\prime}_{\gamma})^2}+\frac{32\lambda}{\eta^{\prime}_{\gamma}}\|D_{S^c}\beta^*\|_{1}+\frac{128\tau(N, n)}{\eta^{\prime}_{\gamma}}\|D_{S^c}\beta^*\|_{1}^{2}.
\end{equation*}
Therefore, we have 
\begin{equation*}
\|\Delta\|_{2}^{2}\le \frac{144\lambda^2|S|}{(\eta^{\prime}_{\gamma})^2}+\frac{32\lambda}{\eta^{\prime}_{\gamma}}\|D_{S^c}\beta^*\|_{1}+\frac{128\tau(N, n)}{\eta^{\prime}_{\gamma}}\|D_{S^c}\beta^*\|_{1}^{2}, 
\end{equation*}
and 
\begin{equation*}
\begin{split}
\frac{1}{N}\|X\Delta\|_{2}^{2}
\le 4\lambda\|D_{S^c}\beta^*\|_{1}+3\lambda\sqrt{|S| \left(\frac{144\lambda^2|S|}{(\eta^{\prime}_{\gamma})^2}+\frac{32\lambda}{\eta^{\prime}}\|D_{S^c}\beta^*\|_{1}+\frac{128\tau(N, n)}{\eta^{\prime}_{\gamma}}\|D_{S^c}\beta^*\|_{1}^{2}\right)}. 
\end{split}
\end{equation*}
Therefore, we obtain the results in Part (a). 

Next, we show Part (b). If $\lambda_{g}=2^{-(1+k/2)}d^{-(k+1)/2}\lambda$, then we have 
\begin{equation*}
\|\Delta\|^{2}_{1}\le 4\|D\Delta\|_{1}^{2}\le 128\|D_{S^c}\beta^*\|_{1}^{2}+128|S|\left(\frac{144\lambda^2|S|}{(\eta^{\prime}_{\gamma})^2}+\frac{32\lambda}{\eta^{\prime}_{\gamma}}\|D_{S^c}\beta^*\|_{1}+\frac{128\tau(N, n)}{\eta^{\prime}_{\gamma}}\|D_{S^c}\beta^*\|_{1}^{2}\right),
\end{equation*}
which yields the result in Part (b). Therefore, the proof is complete. 

\end{proof}


\subsection{Proof of Lemma~\ref{lowerRE2random}}
\label{lowerRE2random-proof}

\begin{proof}
Using Part (b) of Lemma~\ref{REgeneral} and a similar argument as the proof of Lemma~\ref{eta1rand}, for $s\ge 1$, we have
\begin{equation*}
\frac{1}{N}\|XD^{+}v\|_{2}^{2}\ge \frac{1}{3}\lambda_{1}(\Sigma_{x})\|v\|_{2}^{2}-\frac{\lambda_{1}(\Sigma_{x})}{48s}\|v\|_{1}^{2}\quad\forall~v\in \mathbb{R}^{m+n}
\end{equation*}
with probability at least $1-2\exp{(cs\log n-c^{\prime}N)}$ where $c>0$ and $c^{\prime}>0$ are constants. 

Finally, let $s\asymp N/\log n$, then we have 
\begin{equation*}
\frac{1}{N}\|XD^{+}v\|_{2}^{2}\ge \frac{1}{3}\lambda_{1}(\Sigma_{x})\|v\|_{2}^{2}-\frac{c\log n}{N}\|v\|_{1}^{2}\quad\forall~v\in \mathbb{R}^{m+n}
\end{equation*}
with probability at least $1-2\exp{(-c^{\prime}N)}$. Therefore the proof is complete.


\end{proof}


\subsection{Proof of Lemma~\ref{tuningmulemma}}
\label{tuningmulemmaprf}

\begin{proof}
Let $\Gamma=\Theta_{x}\Sigma_{N}-I_{n}$. Then the $(j, k)$-th entry of $\Gamma$ is $\Gamma^{jk}=\frac{1}{N}\sum_{i=1}^{N}\Gamma_{i}^{jk}$
where $\Gamma_{i}^{jk}=e_{j}^T\Theta_{x}X_{i}X^{T}_{i}e_{k}-e_{j}^Te_{k}$. Furthermore, we have $\mathbb{E}(\Gamma_{i}^{jk})=0$, 
and 
\begin{equation*}
\begin{split}
\|\Gamma_{i}^{jk}\|_{\psi_{1}}\le 2\| e_{j}^T\Theta_{x}X_{i}X^T_{i}e_{k} \|_{\psi_{1}}\le 4\| e^T_{j}\Theta_{x}X_{i}  \|_{\psi_{2}}\| X_{i}^Te_{k} \|_{\psi_{2}}\le 4c\lambda_{n}(\Theta_{x})\sigma_{x}^2,
\end{split}
\end{equation*}
where $c>0$ is a constant. Then applying Lemma~\ref{bernstein-type} and letting $K=4c\lambda_{n}(\Theta_{x})\sigma_{x}^2$, for $t>0$, we have 
\begin{equation*}
\mathbb{P}\left(|\Gamma^{jk}|\ge t\right)\le 2\exp{\left[-C_{b}N\min\left(\frac{t^2}{K^2}, \frac{t}{K}\right)\right]}. 
\end{equation*}
Taking a union bound, we have
\begin{equation*}
\mathbb{P}\left(\|\Gamma\|_{\infty}\ge t\right)\le 2n^2\exp{\left[-C_{b}N\min\left(\frac{t^2}{K^2}, \frac{t}{K}\right)\right]}. 
\end{equation*}
Finally, letting $t=c\sqrt{\frac{\log n}{N}}$ and $N\gtrsim \log n$, we have
\begin{equation*}
\mathbb{P}\left(\|\Gamma\|_{\infty}\ge c\sqrt{\frac{\log n}{N}}\right)\le 2\exp{(-c^{\prime}\log n)},
\end{equation*}
implying the desired result in the lemma.


\end{proof}


\subsection{Proof of Lemma~\ref{conelemma}}
\label{conelemmaproof}

\begin{proof}
By the optimality of $\hat{\beta}$, we have
\begin{equation*}
    \frac{1}{2N}\|y-X\hat{\beta}\|_{2}^{2}+\lambda \|D\hat{\beta}\|_{1}\le \frac{1}{2N}\|y-X\beta^{*}\|_{2}^{2}+\lambda \|D\beta^{*}\|_{1}.
\end{equation*}
Let $\Delta=\hat{\beta}-\beta^*$. Then rearranging terms, we have
\begin{equation*}
\begin{split}
\frac{1}{2N}||X\Delta||_{2}^{2}\le \frac{1}{N}\varepsilon^{T}X\Delta+\lambda\left(\|D\beta^{*}\|_{1}-\|D\hat{\beta}\|_{1} \right) \\
= \frac{1}{N}\varepsilon^{T}X\Delta+\lambda\left(\|(D\beta^{*})_{S}\|_{1}-\|(D\hat{\beta})_{S}\|_{1}-\|(D\hat{\beta})_{S^c}\|_{1}  \right) \\
    \le \frac{1}{N}\varepsilon^{T}X\Delta+\lambda\left(\|D_{S}\Delta\|_{1}-\|D_{S^c}\Delta\|_{1}   \right),
\end{split}
\end{equation*}
where the last inequality follows from the triangle inequality and the definition of $S$. Furthermore, by Holder's inequality and the fact that $D^{+}D=I_{n}$, we have
\begin{equation*}
  \varepsilon^{T}X\Delta\le \|\varepsilon^T XD^{+}\|_{\infty}\|D\Delta\|_{1}.
\end{equation*}
Therefore, 
\begin{equation*}
  0\le\frac{1}{2N}\|X\Delta\|_{2}^{2} \le \frac{1}{N}\|\varepsilon^T XD^{+}\|_{\infty}\|D\Delta\|_{1}+\lambda\left(\|D_{S}\Delta\|_{1}-\|D_{S^c}\Delta\|_{1}  \right). 
\end{equation*}
So when $\lambda$ satisfies the condition in the lemma, we have
\begin{equation*}
    0\le \frac{\lambda}{2}\left(\|(D\Delta)_{S}\|_{1}+\|(D\Delta)_{S^c}\|_{1}\right)+\lambda\left(\|(D\Delta)_{S}\|_{1}-\|(D\Delta)_{S^c}\|_{1}   \right).
\end{equation*}
Thus, we have $\|(D\Delta)_{S^c}\|_{1}\le 3\|(D\Delta)_{S}\|_{1}$. Therefore we prove the result. 
\end{proof}


\subsection{Proof of Lemma~\ref{consistencyTheta}}
\label{consistencyThetaprf}

\begin{proof}
 We have
\begin{equation*}
\begin{split}
\|\widehat{\Theta}-\Theta_{x}\|_{\infty}=\|\widehat{\Theta}\left(I_{n}-\Sigma_{N}\Theta_{x}\right)+(\widehat{\Theta}\Sigma_{N}-I_{n})\Theta_{x}\|_{\infty} \\
\le \|\widehat{\Theta}\left(I_{n}-\Sigma_{N}\Theta_{x}\right)\|_{\infty}+\|(\widehat{\Theta}\Sigma_{N}-I_{n})\Theta_{x}\|_{\infty} \\
\le \opnorm{\widehat{\Theta}}_{\infty}\|I_{n}-\Sigma_{N}\Theta_{x}\|_{\infty}+\|\widehat{\Theta}\Sigma_{N}-I_{n}\|_{\infty}\opnorm{\Theta_{x}}_{1} \\
\le \opnorm{\widehat{\Theta}}_{\infty}\|I_{n}-\Sigma_{N}\Theta_{x}\|_{\infty}+M_{n}\mu.
\end{split}
\end{equation*}
If $\|\Theta_{x}\Sigma_{N}-I_{n}\|_{\infty}\le \mu$, then by the optimality of $\widehat{\Theta}$ and the feasibility of $\Theta_{x}$, we have $\|\widehat{\Theta}_{j}\|_{1}\le \|(\Theta_{x})_{j}\|_{1}$ for $1\le j\le n$ where $\widehat{\Theta}^{T}_{j}$ and $(\Theta_{x})^{T}_{j}$ are $j$-th row of $\widehat{\Theta}$ and $\Theta_{x}$ respectively. Therefore, we have
$\opnorm{\widehat{\Theta}}_{\infty}\le \opnorm{\Theta_{x}}_{\infty}\le M_{n}$. Hence,  we obtain $\|\widehat{\Theta}-\Theta_{x}\|_{\infty}\le 2M_{n}\mu$, implying the result in the lemma. 
\end{proof}


\subsection{Proof of Lemma~\ref{sigmaconsis}}
\label{sigmaconsisprf}

\begin{proof}
We begin by writing 
\begin{equation*}
\begin{split}
\left| \frac{\hat{\sigma}_{\varepsilon}}{\sigma_{\varepsilon}}-1\right|<\left| \frac{\hat{\sigma}^{2}_{\varepsilon}}{\sigma^2_{\varepsilon}}-1\right|=\frac{1}{\sigma^{2}_{\varepsilon}}\left|\frac{1}{N} \sum_{i=1}^{N}\left(y_{i}-X_{i}^{T} \hat{\beta}\right)^{2}-\sigma^{2}_{\varepsilon}\right| \\
\le \frac{1}{\sigma^{2}_{\varepsilon}}\left[  \left|\frac{1}{N} \sum_{i=1}^{N}\left(y_{i}-X_{i}^{T} \hat{\beta}\right)^{2}-\frac{1}{N} \sum_{i=1}^{N}\left(y_{i}-X_{i}^{T} \beta^{*}\right)^{2}\right|+\left|\frac{1}{N} \sum_{i=1}^{N}\left(y_{i}-X_{i}^{T} \beta^{*}\right)^{2}-\mathbb{E}(\varepsilon_{i}^{2})\right| \right]. 
\end{split}
\end{equation*}
Therefore, it suffices to bound
\begin{equation*}
\left|\frac{1}{N} \sum_{i=1}^{N}\left(y_{i}-X_{i}^{T} \hat{\beta}\right)^{2}-\frac{1}{N} \sum_{i=1}^{N}\left(y_{i}-X_{i}^{T} \beta^{*}\right)^{2}\right|
\end{equation*}
 and 
\begin{equation*}
\left|\frac{1}{N} \sum_{i=1}^{N}\left(y_{i}-X_{i}^{T} \beta^{*}\right)^{2}-\mathbb{E}(\varepsilon_{i}^{2})\right|
\end{equation*}

First, by Lemma~\ref{bernstein-type}, we have 
\begin{equation*}
\left|\frac{1}{N}\sum_{i=1}^{N}(y_{i}-X_{i}^T\beta^*)^2-\mathbb{E}(\varepsilon^2_{i})\right| \le \mathcal{O}_{\mathbb{P}}\left(\sqrt{\frac{\log N}{N}}\right). 
\end{equation*}

Next, we bound the first term, we have 
\begin{equation*}
\begin{split}
\left|\frac{1}{N} \sum_{i=1}^{N}\left(y_{i}-X_{i}^{T} \hat{\beta}\right)^{2}-\frac{1}{N} \sum_{i=1}^{N}\left(y_{i}-X_{i}^{T} \beta^{*}\right)^{2}\right| 
=\left|\frac{1}{N} \sum_{i=1}^{N}\left(\left(X_{i}^{T}\left(\beta^{*}-\hat{\beta}\right)+\varepsilon_{i}\right)^{2}-\varepsilon_{i}^{2}\right)\right| \\ 
\le \frac{1}{N}\left|\sum_{i=1}^{N}\left(X_{i}^{T}\left(\hat{\beta}-\beta^{*}\right)\right)^{2}\right|+\frac{2}{N}\left|\sum_{i=1}^{N}\left(X_{i}^{T}\left(\beta^*-\hat{\beta}\right)\right) \varepsilon_{i}\right| 
\le\left|\left(\hat{\beta}-\beta^{*}\right)^{T}\Sigma_{N}\left(\hat{\beta}-\beta^{*}\right)\right|+2\left\|\frac{X^{T} \varepsilon}{N}\right\|_{\infty}\left\|\hat{\beta}-\beta^{*}\right\|_{1} \\
\le \left|\left(\hat{\beta}-\beta^{*}\right)^{T}\left(\Sigma_{N}-\Sigma_{x}\right)\left(\hat{\beta}-\beta^{*}\right)\right|+\left|(\hat{\beta}-\beta^*)^T\Sigma_{x}(\hat{\beta}-\beta^*)\right|+2\left\|\frac{X^{T} \varepsilon}{N}\right\|_{\infty}\left\|\hat{\beta}-\beta^{*}\right\|_{1}. 
\end{split}
\end{equation*}
By Theorem~\ref{probconthm}, we have 
\begin{equation}
\label{l8eq1}
\left|(\hat{\beta}-\beta^*)^T\Sigma_{x}(\hat{\beta}-\beta^*)\right|\le \mathcal{O}_{\mathbb{P}}\left( \frac{s_{2}\log n}{N}\right). 
\end{equation}
Using similar arguments as the proof of Lemma~\ref{eta1rand}, Theorem~\ref{fixedconthm} and Theorem~\ref{probconthm}, we have   
\begin{equation}
\label{l8eq2}
\left|(\hat{\beta}-\beta^*)^T(\Sigma_{N}-\Sigma)(\hat{\beta}-\beta^*)\right|\le \mathcal{O}_{\mathbb{P}}\left( \frac{s_{2}\log n}{N} \right). 
\end{equation}
 Furthermore, using similar arguments as the proof of Lemma~\ref{tuninglem} and Theorem~\ref{probconthm}, we have 
\begin{equation}
\label{l8eq3}
\left\|\frac{X^T\varepsilon}{N}\right\|_{\infty}\left\|\hat{\beta}-\beta^*\right\|_{1}\le \mathcal{O}_{\mathbb{P}}\left( \frac{s_{2}\log n}{N}  \right).
\end{equation}
Therefore, combining $(\text{\ref{l8eq1}})$, $(\text{\ref{l8eq2}})$ and $(\text{\ref{l8eq3}})$, we have 
\begin{equation*}
\begin{split}
\left|\frac{1}{N} \sum_{i=1}^{N}\left(y_{i}-X_{i}^{T} \hat{\beta}\right)^{2}-\frac{1}{N} \sum_{i=1}^{N}\left(y_{i}-X_{i}^{T} \beta^{*}\right)^{2}\right|\le \mathcal{O}_{\mathbb{P}}\left(\frac{s_{2}\log n}{N}\right). 
\end{split}
\end{equation*}

Altogether, we conclude that 
\begin{equation*}
\left| \frac{\hat{\sigma}_{\varepsilon}}{\sigma_{\varepsilon}}-1  \right|\le \mathcal{O}_{\mathbb{P}}\left(\frac{s_{2}\log n}{N}+\sqrt{\frac{\log N}{N}}\right). 
\end{equation*}
Hence the proof is complete.

\end{proof}


\section{Supplementary lemmas}
\label{supplementary-appendix}
\setcounter{equation}{0}
\renewcommand{\theequation}{E.\arabic{equation}}

In this section, we collect several useful results which are frequently used in our proofs. 

\subsection{Restricted eigenvalue condition}
\label{RE-appendix}

We start with a geometric lemma which shows how to bound the intersection of the $\ell_{1}$-ball with $\ell_{2}$-ball in terms of a simpler set. This result is a generalization of Lemma 11 in \cite{LohWai12}. 
\begin{lemma}
\label{LemBall}
For any integer $s \ge 1$ and any constant $c>0$, we have
\begin{equation*}
\Ball_1(c\sqrt{s}) \cap \Ball_2(1) \subseteq (1+c) \mathrm{cl}\left\{\mathrm{conv}\{\Ball_0(s) \cap \Ball_2(1)\}\right\}.
\end{equation*}
where ``cl" denotes the closure of a set, ``conv" denotes the convex hull. All these balls are in $\mathbb{R}^{n}$.
\end{lemma}

\begin{proof}
Without loss of generality, we assume $1\le s\le n$. The key idea is using a fact that $\phi_{A}(x)\le \phi_{B}(x)$ if and only if $A\subset B$ where $A, B\in \mathbb{R}^n$ are closed convex sets, and $\phi_{A}(x)=\sup_{\theta\in A}\theta^Tx$ and $\phi_{B}(x)=\sup_{\theta\in B}\theta^Tx$.

Let $A=\Ball_1(c\sqrt{s}) \cap \Ball_2(1)$, $B=(1+c) \mathrm{cl}\left\{\mathrm{conv}\{\Ball_0(s) \cap \Ball_2(1)\}\right\}$ and $x\in \mathbb{R}^{n}$. Denote the subset that indexes the top $s$ elements of $x$ in absolute value by S. Then we have $\|x_{S^c}\|_{\infty}\le |x_{j}|$ for all $j\in S$, and
\begin{equation*}
\|x_{S^c}\|_{\infty}\le \frac{1}{s}\|x_{S}\|_{1}\le \frac{1}{\sqrt{s}}\|x_{S}\|_{2}.
\end{equation*}

Furthermore, we have 
\begin{equation*}
\phi_{A}(x)=\sup_{\theta\in A}\theta^Tx=\sup_{\theta\in A}(\theta_{S}^Tx_{S}+\theta_{S^c}^Tx_{S^c})\le \sup_{\|\theta_{S}\|_{2}\le 1}\theta_{S}^Tx_{S}+\sup_{\|\theta_{S^c}\|_{1}\le c\sqrt{s}}\theta_{S^c}^Tx_{S^c} \le (1+c)\|x_{S}\|_{2}, 
\end{equation*}
and
\begin{equation*}
\phi_{B}(x)=\sup_{\theta\in B}\theta^Tx=(1+c)\max_{|U|=s}\sup_{\|\theta_{U}\|_{2}\le 1}\theta_{U}^Tx_{U}=(1+c)\|x_{S}\|_{2}
\end{equation*}
from which the lemma holds.
\end{proof}

Our next result builds on the above geometric lemma. 

\begin{lemma}
\label{REprelim}
Let $\mathbb{L}_{0}(s)=\mathbb{B}_{0}(s)\cap\mathbb{B}_{2}(1)$ and $\mathbb{L}_{1}(s)=\{v: \|v\|_{1}\le 4\sqrt{s}\|v\|_{2}\}$. For a symmetric matrix $\Gamma\in \mathbb{R}^{n\times n}$, parameters $s\ge 1$ and $\delta>0$, suppose we have the deviation condition that $|v^T\Gamma v|\le \delta$ for all $v\in \mathbb{L}_{0}(2s)$. Then,
\begin{equation}
\label{lemma18-eq1}
    |v^T\Gamma v|\le 75\delta \|v\|_{2}^{2}\quad\forall~v\in \mathbb{L}_{1}(s).
\end{equation}
Furthermore, we have
\begin{equation}
\label{lemma18-eq2}
|v^T\Gamma v|\le 75\delta\|v\|_{2}^{2}+\frac{75\delta}{16s}\|v\|_{1}^{2}\quad\forall~v\in \mathbb{R}^{n}. 
\end{equation}

\end{lemma}

\begin{proof}
We first show $\text{(\ref{lemma18-eq1})}$. It suffices to prove $|v^T\Gamma v|\le 75\delta$ for all $v\in \mathbb{L}_{1}(s)\cap \mathbb{B}_{2}(1)$. By Lemma~\ref{LemBall} and continuity, we could reduce the problem to proving $|v^T\Gamma v|\le 75\delta$ for all $v\in 5\mathrm{conv}\{\mathbb{L}_{0}(s)\}=\mathrm{conv}\{ \mathbb{B}_{0}(s)\cap \mathbb{B}_{2}(5)\}$. Consider the convex combination $v=\sum_{i}\alpha_{i}v_{i}$ where $\alpha_{i}\ge 0$ and $\sum_{i}\alpha_{i}=1$, and $\|v_{i}\|_{0}\le s$ and $\|v_{i}\|_{2}\le 5$ for each $i$. Then we have
\begin{equation*}
    \left|v^T\Gamma v\right|=\left|( \sum_{i}\alpha_{i}v_{i})^T\Gamma (\sum_{j}\alpha_{j}v_{j})\right|=\left|\sum_{i,j}\alpha_{i}\alpha_{j}(v_{i}^T\Gamma v_{j} )\right|.
\end{equation*}
Furthermore, $\frac{1}{5}v_{i}\in \mathbb{L}_{0}(s)\subset \mathbb{L}_{0}(2s), \frac{1}{10}(v_{i}+v_{j})\in \mathbb{L}_{0}(2s)$, so we have
\begin{equation*}
    \left|v_{i}^T\Gamma v_{j}\right|=\frac{1}{2}\left|(v_{i}+v_{j})^T\Gamma (v_{i}+v_{j})-v_{i}^T\Gamma v_{i}-v_{j}^T\Gamma v_{j}\right|\le \frac{1}{2}(100\delta+25\delta+25\delta)=75\delta
\end{equation*}
for all $i$ and $j$. Therefore, $|v^T\Gamma v|\le 75\delta$.

Next, we show $\text{(\ref{lemma18-eq2})}$. For $v\notin \mathbb{L}_{1}(s)$, let $u=4\sqrt{s}\frac{v}{\|v\|_{1}}$. Then we have $\|u\|_{2}<1$ and $\|u\|_{1}=4\sqrt{s}$, which 
implies that $u\in \mathbb{B}_{1}(4\sqrt{s})\cap\mathbb{B}_{2}(1)$. Hence for $v\notin\mathbb{L}_{1}(s)$, we have
\begin{equation*}
\frac{|v^T\Gamma v|}{\|v\|_{1}^{2}}\le \frac{1}{16s}\sup_{u\in \mathbb{B}_{1}(4\sqrt{s})\cap \mathbb{B}_{2}(1)}u^T\Gamma u. 
\end{equation*}
Using a similar argument with the previous one, we have 
\begin{equation}
\label{lemma18-eq3}
\frac{|v^T\Gamma v|}{\|v\|_{1}^{2}}\le \frac{75\delta}{16s}\quad\forall~v\notin \mathbb{L}_{1}(s).
\end{equation}
Therefore, combining $\text{(\ref{lemma18-eq1})}$ and $\text{(\ref{lemma18-eq3})}$, we have 
\begin{equation*}
|v^T\Gamma v|\le 75\delta\|v\|_{2}^{2}+\frac{75\delta}{16s}\|v\|_{1}^{2}\quad\forall~v\in \mathbb{R}^{n},
\end{equation*}
implying that $\text{(\ref{lemma18-eq2})}$ holds. 

\end{proof}

Then we have the following general result on restricted eigenvalue condition, which is a direct application of Lemma~\ref{REprelim}. 

\begin{lemma}
\label{REgeneral}
Let $\mathbb{L}_{0}(s)$ and $\mathbb{L}_{1}(s)$ be two sets defined in Lemma~\ref{REprelim}. Let $\widehat{\Gamma}\in \mathbb{R}^{n\times n}$ and $\Gamma\in \mathbb{R}^{n\times n}$ be two symmetric matrices. 
\begin{enumerate}
\item[(a)] If there exists a $\delta>0$, such that $v^T\Gamma v\ge \delta\|v\|_{2}^2$ for all $v\in \mathbb{L}_{1}(s)$ and 
\begin{equation*}
\left|v^T(\widehat{\Gamma}-\Gamma)v\right|\le \frac{\delta}{150}\quad\forall~v\in \mathbb{L}_{0}(2s),
\end{equation*}
then we have
\begin{equation}
\label{lemma19-eq1}
v^T\widehat{\Gamma}v\ge \frac{\delta}{2}\|v\|_{2}^{2}\quad\forall~v\in \mathbb{L}_{1}(s). 
\end{equation}

\item[(b)] If there exists a $\delta>0$, such that $v^T\Gamma v\ge \delta\|v\|_{2}^2$ for all $v\in \mathbb{R}^{n}$ and 
\begin{equation*}
\left|v^T(\widehat{\Gamma}-\Gamma)v\right|\le \frac{\delta}{150}\quad\forall~v\in \mathbb{L}_{0}(2s),
\end{equation*}
then we have
\begin{equation}
\label{lemma19-eq2}
v^T\widehat{\Gamma}v\ge \frac{\delta}{2}\|v\|_{2}^{2}-\frac{\delta}{32s}\|v\|_{1}^{2}\quad\forall~v\in \mathbb{R}^{n}.  
\end{equation}

\end{enumerate}
\end{lemma}

\begin{proof}
We first show $\text{(\ref{lemma19-eq1})}$. Since $\left|v^T(\widehat{\Gamma}-\Gamma)v\right|\le \frac{\delta}{150}$ for all $v\in \mathbb{L}_{0}(2s)$, so by Lemma~\ref{REprelim}, we have $\left|v^T(\widehat{\Gamma}-\Gamma)v\right|\le \frac{\delta}{2}\|v\|_{2}^2$ for all $v\in \mathbb{L}_{1}(s)$. Therefore, we have
\begin{equation*}
v^T\widehat{\Gamma}v\ge -\frac{\delta}{2}\|v\|_{2}^2+v^T\Gamma v\ge -\frac{\delta}{2}\|v\|_{2}^2+\delta\|v\|_{2}^2=\frac{\delta}{2}\|v\|_{2}^2.
\end{equation*}
Similarly, applying $\text{(\ref{lemma18-eq2})}$ in Lemma~\ref{REprelim} yields $\text{(\ref{lemma19-eq2})}$. 


\end{proof}

\subsection{Deviation bounds}
\label{mat-appendix}

We start with the following definitions on sub-exponential norm and sub-Gaussian norm.  
\begin{definition}
For a random variable $X$, the sub-exponential norm is defined as 
\begin{equation*}
\|X\|_{\psi_{1}}=\sup_{p\ge 1}p^{-1}(\mathbb{E}|X|^p)^{1/p},  
\end{equation*}
and the sub-Gaussian norm is defined as 
\begin{equation*}
\|X\|_{\psi_{2}}=\sup_{p\ge 1}p^{-1/2}(\mathbb{E}|X|^p)^{1/p}. 
\end{equation*}
\end{definition}

It is straightforward to show that for a $\sigma$-sub-Gaussian random variable $X$ defined in Definition~\ref{subgaussianrv}, we have 
$\|X\|_{\psi_{2}}\le \sigma\max{(e^{1/e}, \sqrt{2\pi})}$. Furthermore, for two sub-Gaussian random variables $X$ and $Y$, we have 
$\|XY\|_{\psi_{1}}\le 2\|X\|_{\psi_{2}}\|Y\|_{\psi_{2}}$. Next, we have a general result for sum of independent sub-exponential random variables cited from Proposition 5.16 in \cite{Vershynin11}. 

\begin{lemma}[Bernstein-type inequality]
\label{bernstein-type}
Let $X_{1},...,X_{N}$ be independent centered sub-exponential random variable, and $K=\max_{i}\|X_{i}\|_{\psi_{1}}$. Then for every
$a=(a_{1},...,a_{N})\in \mathbb{R}^{N}$ and for every $t\ge 0$, we have
\begin{equation*}
  \mathbb{P}\left(\sum_{i=1}^{N}a_{i}X_{i}\ge t\right)\le \exp{\left[-C_{b}\min{\left(\frac{t}{K\|a\|_{\infty}}, \frac{t^2}{K^2\|a\|_{2}^2}\right)}\right]},
\end{equation*}
and 
\begin{equation*}
  \mathbb{P}\left(\sum_{i=1}^{N}a_{i}X_{i}\le -t\right)\le \exp{\left[-C_{b}\min{\left(\frac{t}{K\|a\|_{\infty}}, \frac{t^2}{K^2\|a\|_{2}^2}\right)}\right]},
\end{equation*}
where $C_{b}>0$ is a universal constant. 
\end{lemma}

We now derive the following lemma for sub-Gaussian random matrix based on Lemma~\ref{bernstein-type}.

\begin{lemma}
\label{LemBern}
Assume $X\in \mathbb{R}^{N\times n}$ is a row-wise $(\sigma_{x}, \Sigma_{x})$-sub-Gaussian random matrix defined in Definition~\ref{randmat}. 
\begin{enumerate}
\item[(a)] For any fixed unit vector $v\in \real^{n}$ and $t>0$, we have
\begin{equation*}
    \mathbb{P}\left(v^T\frac{X^TX}{N}v-v^T\Sigma_{x}v\ge t\right)\le \exp{\left[-NC_{b}\min{\left(\frac{t^{2}}{16c_{b}^2\sigma_{x}^{4}}, \frac{t}{4c_{b}\sigma_{x}^2}\right)}\right]},
\end{equation*}
and
\begin{equation*}
    \mathbb{P}\left(v^T\frac{X^TX}{N}v-v^T\Sigma_{x}v\le -t\right)\le  \exp{\left[-NC_{b}\min{\left(\frac{t^{2}}{16c_{b}^2\sigma_{x}^{4}}, \frac{t}{4c_{b}\sigma_{x}^2}\right)}\right]},
\end{equation*}
where $c_{b}>0$ and $C_{b}>0$ are constants. 
\item[(b)] For $t>0$, we have
\begin{equation*}
\mathbb{P}\left( \opnorm{\frac{1}{N}X^TX-\Sigma_{x}}_{op}\ge 2t\right)\le 9^{n}\times 2\exp{\left[-NC_{b}\min{\left(\frac{t^2}{16c_{b}^2\sigma_{x}^4}, \frac{t}{4c_{b}\sigma_{x}^2}\right)}\right]}.
\end{equation*}

\item[(c)] We have
\begin{equation*}
\frac{1}{2}\lambda_{1}(\Sigma_{x})\le \lambda_{1}\left(\frac{X^TX}{N}\right)\le \lambda_{n}\left(\frac{X^TX}{N}\right)\le \frac{3}{2}\lambda_{n}(\Sigma_{x})
\end{equation*}
with probability at least 
\begin{equation*}
1-2\exp{\left[n\log 9-NC_{b}\min{\left(\frac{\lambda^2_{1}(\Sigma_{x})}{256c_{b}^2\sigma_{x}^4}, \frac{\lambda_{1}(\Sigma_{x})}{16c_{b}\sigma_{x}^2}\right)}\right]}.
\end{equation*}
\end{enumerate}
\end{lemma}

\begin{proof}
First, we show Part (a). Let $X_{i}^T$ be the i-th row of X. Since $X_{i}^Tv$ is $\sigma_{x}$-sub-Gaussian, so $\|X_{i}^Tv\|_{\psi_{2}}^{2}\le c_{b}\sigma_{x}^2$ where $c_{b}>0$ is a constant. Therefore, 
\begin{equation*}
    \|(X_{i}^Tv)^{2}\|_{\psi_{1}}\le 2\|X_{i}^Tv\|_{\psi_{2}}^{2}\le 2c_{b}\sigma_{x}^2.
\end{equation*}
Hence we have
\begin{equation*}
    \|(X_{i}^Tv)^{2}-\mathbb{E}(X_{i}^Tv)^{2}\|_{\psi_{1}}\le  2\|(X_{i}^Tv)^{2}\|_{\psi_{1}}  \le4c_{b}\sigma_{x}^2.
\end{equation*}
Applying Lemma~\ref{bernstein-type} and letting $K=4c_{b}\sigma_{x}^2$, we have
\begin{equation*}
    \mathbb{P}\left[\sum_{i=1}^{N}((X_{i}^Tv)^{2}-\mathbb{E}(X_{i}^Tv)^{2})\ge Nt\right]\le \exp{\left[-NC_{b}\min{\left(\frac{t^{2}}{16c_{b}^2\sigma_{x}^{4}}, \frac{t}{4c_{b}\sigma_{x}^2}\right)}\right]}, 
\end{equation*}
and
\begin{equation*}
    \mathbb{P}\left[\sum_{i=1}^{N}\left((X_{i}^Tv)^{2}-\mathbb{E}(X_{i}^Tv)^{2}\right)\le -Nt\right]\le \exp{\left[-NC_{b}\min{\left(\frac{t^{2}}{16c_{b}^2\sigma_{x}^4}, \frac{t}{4c_{b}\sigma_{x}^2}\right)}\right]},
\end{equation*}
implying the result in Part (a). 

Next, we show Part (b). It suffices to evaluate the operator norm on a $\frac{1}{4}$-net $\mathcal{N}$ of $\mathbb{S}^{n-1}$ since we have
  \begin{equation*}
  \opnorm{\frac{1}{N}X^TX-\Sigma_{x}}_{op}\le 2\max_{v\in \mathcal{N}}\left|v^T\left(\frac{1}{N}X^TX-\Sigma_{x}\right)v\right|.
  \end{equation*}
For any fixed $v\in \mathcal{N}\subset \mathbb{S}^{n-1}$, by Part (a), we have
  \begin{equation*}
  \mathbb{P}\left(\left| v^T\frac{1}{N}X^TXv-v^T\Sigma_{x}v  \right|\ge t\right)\le 2\exp{\left[-NC_{b}\min{\left(\frac{t^2}{16c_{b}^2\sigma_{x}^4}, \frac{t}{4c_{b}\sigma_{x}^2}\right)}\right]}.
  \end{equation*}
Therefore, taking a union bound over $\mathcal{N}$, we have
\begin{equation*}
\mathbb{P}\left( \opnorm{\frac{1}{N}X^TX-\Sigma_{x}}_{op}\ge 2t\right)\le 9^{n}\times 2\exp{\left[-NC_{b}\min{\left(\frac{t^2}{16c_{b}^2\sigma_{x}^4}, \frac{t}{4c_{b}\sigma_{x}^2}\right)}\right]}.
\end{equation*}
Hence we prove Part (b). 

Finally, we show Part (c). Let $t=\frac{1}{4}\lambda_{1}(\Sigma_{x})$. Then by Part (b), we have 
\begin{equation*}
\opnorm{\frac{1}{N}X^TX-\Sigma_{x}}_{op}\le \frac{\lambda_{1}(\Sigma_{x})}{2}
\end{equation*}
with probability at least 
\begin{equation*}
1-2\exp{\left[n\log 9-NC_{b}\min{\left(\frac{\lambda^2_{1}(\Sigma_{x})}{256c_{b}^2\sigma_{x}^4}, \frac{\lambda_{1}(\Sigma_{x})}{16c_{b}\sigma_{x}^2}\right)}\right]}. 
\end{equation*}
By Lemma~\ref{randmatdev1}, we have 
\begin{equation*}
\frac{1}{2}\lambda_{1}(\Sigma_{x})\le \lambda_{1}\left(\frac{X^TX}{N}\right)\le \lambda_{n}\left(\frac{X^TX}{N}\right)\le \frac{3}{2}\lambda_{n}(\Sigma_{x})
\end{equation*}
with at least the same probability. Hence the proof is complete. 


\end{proof}


\subsection{Other supporting lemmas}
\label{supporting-appendix}

\begin{lemma}
\label{graph-oriented}
If the undirected graph $\mathcal{G}=(\mathcal{V},\mathcal{E})$ has $r$ connected components and $|\mathcal{V}|=n$, then the rank of $F$ and the rank of $L$ are equal to  
$n-r$, where $F$ is the oriented incidence matrix and $L$ is the Laplacian matrix. Furthermore, for $k\ge 2$, the rank of $\Delta^{(k+1)}$ is also equal to $n-r$, where $\Delta^{(k+1)}$ is the graph difference operator of order $k+1$. 
\end{lemma}

\begin{proof}
Let $z$ be a vector such that $Fz=0$. Then for every $(i, j)\in \mathcal{E}$, we have $z_{i}=z_{j}$, which implies $z$ takes the same value on vertices of the same connected component. Therefore the dimension of the null space of $F$ is $r$. By rank-nullity theorem, we have the rank of $F$ is $n-r$. Furthermore, since $L=F^TF$, so the rank of $L$ is equal to $n-r$. For $k\ge 2$, applying the singular value decomposition of $F$ obtains the desired result for $\Delta^{(k+1)}$. Thus we prove the lemma. 
\end{proof}

\begin{lemma}
\label{eigenvaluelem}

The largest eigenvalue of Laplacian matrix $L$ satisfies $\lambda_{n}(L)\le 2d$ where $d$ is the maximum degree. Furthermore, for $D$ defined in $(\ref{Dmatrix})$, we have 
\begin{equation}
\label{singularvalueD}
1=\sigma_{\min}(D)\le \sigma_{\max}(D)\le \sqrt{\left(\frac{\lambda_{g}}{\lambda}  \right)^2(2d)^{k+1}+1}. 
\end{equation}
\end{lemma}

\begin{proof}
Since $L=M-A$ where $M\in \mathbb{R}^{n\times n}$ is the degree matrix and $A\in \mathbb{R}^{n\times n}$ is the adjacency matrix, so we have
$\lambda_{n}(L)\le \lambda_{n}(M)+\lambda_{n}(A)=d+\lambda_{n}(A)$ where $d$ is the maximum degree. Next, we bound $\lambda_{n}(A)$. Let $v$ be the eigenvector of $\lambda_{n}(A)$ and let $i$ be the node on which $v$ takes its maximum value. Without loss of generality, we assume $v_{i}>0$. Then 
\begin{equation*}
\lambda_{n}(A)v_{i}=A_{i}^Tv=\sum_{(i, j)\in \mathcal{E}}v_{j}\le dv_{i},
\end{equation*}
where $A_{i}^T$ is the $i$-th row of $A$.  Hence we have $\lambda_{n}(A)\le d$, which yields the result for $\lambda_{n}(L)$.  

Next, we show $(\text{\ref{singularvalueD}})$. For $v\in \mathbb{R}^{n}$ and $\|v\|_{2}=1$, we have 
\begin{equation*}
v^TD^TDv=\left( \frac{\lambda_{g}}{\lambda}\right)^{2}v^TL^{k+1}v+1. 
\end{equation*}
Therefore, we have $\sigma_{\min}(D)=1$. Furthermore, by the first result in this lemma, we have 
\begin{equation*}
v^TD^TDv\le \left(\frac{\lambda_{g}}{\lambda}  \right)^2(2d)^{k+1}+1, 
\end{equation*}
implying the desired result. 
\end{proof}


\begin{lemma}
\label{l1matrixnorm}
Let $A\in \mathbb{R}^{n\times n}$ be a positive semidefinite matrix and $\opnorm{A}_{1}<1$. Then we have
\begin{equation*}
\opnorm{(I_{n}+A)^{-1}}_{1}\le \frac{1}{1-\opnorm{A}_{1}}. 
\end{equation*}
\end{lemma}

\begin{proof}
Since $A$ is positive semidefinite, so $(I_{n}+A)^{-1}$ is well-defined. We have 
\begin{equation*}
(I_{n}+A)^{-1}=I_{n}-A(I_{n}+A)^{-1}, 
\end{equation*}
which implies 
\begin{equation*}
\opnorm{(I_{n}+A)^{-1}}_{1}\le 1+\opnorm{A}_{1}\opnorm{(I_{n}+A)^{-1}}_{1}.
\end{equation*}
Therefore, we have 
\begin{equation*}
\opnorm{(I_{n}+A)^{-1}}_{1}\le \frac{1}{1-\opnorm{A}_{1}}. 
\end{equation*}

\end{proof}

\begin{lemma}
\label{randmatdev1}
Let $X$ be a $N\times n$ matrix and $\Sigma_{x}\in \mathbb{R}^{n\times n}$ be a positive definite matrix. If
\begin{equation*}
\opnorm{\frac{1}{N}X^TX-\Sigma_{x}}_{op}\le \frac{1}{2}\lambda_{1}(\Sigma_{x}),
\end{equation*}
then we have
\begin{equation*}
\frac{1}{2}\lambda_{1}(\Sigma_{x})\le\lambda_{1}\left(\frac{X^TX}{N}\right)\le \lambda_{n}\left(\frac{X^TX}{N}\right)\le \frac{3}{2}\lambda_{n}(\Sigma_{x}).
\end{equation*}
\end{lemma}

\begin{proof}
For $v\in \mathbb{S}^{n-1}$, we have
\begin{equation*}
\left|v^T(\frac{1}{N}X^TX-\Sigma_{x})v\right|\le \frac{1}{2}\lambda_{1}(\Sigma_{x}).
\end{equation*}
Therefore,
\begin{equation*}
\lambda_{1}(\Sigma_{x})-\frac{1}{2}\lambda_{1}(\Sigma_{x})\le v^T\Sigma_{x}v-\frac{1}{2}\lambda_{1}(\Sigma_{x})\le v^T\frac{1}{N}X^TXv\le \frac{1}{2}\lambda_{1}(\Sigma_{x})+v^T\Sigma_{x}v\le \frac{1}{2}\lambda_{1}(\Sigma_{x})+\lambda_{n}(\Sigma_{x}),
\end{equation*}
which implies that our lemma holds. 
\end{proof}

\section{Supplementary simulation and real data analysis results}

In this section, we provide more results on simulation studies and the real data analysis conducted in the main paper.

\subsection{Simulations on a 2d grid graph}
\label{simuappendix}

We performed simulations to compare the performance of our approach with Lasso, Graph-Smooth-Lasso (\ref{gsmoothlasso}), and Graph-Spline-Lasso (\ref{gsplinelasso}) for structure recovery over a 2d grid graph. We set $N=250$, and followed the same procedure conducted in Section~\ref{simu2section} of the main paper to estimate three scenarios of $\beta^*$ plotted in Figure~\ref{piecewisepoly2dgraph} with the mentioned approaches, respectively. Figure~\ref{2dconstantestimation}, Figure~\ref{2dlinearestimation}, and Figure~\ref{2dquadraticestimation} present the corresponding results. In Figure~\ref{2dconstantestimation} and Figure~\ref{2dlinearestimation}, our approach visibly outperformed the other three methods. In Figure~\ref{2dquadraticestimation}, our approach had a similar performance with the Graph-Spline-Lasso.          

\begin{figure}[h]
\begin{subfigure}{.5\textwidth}
  \centering
  \includegraphics[width=.8\linewidth]{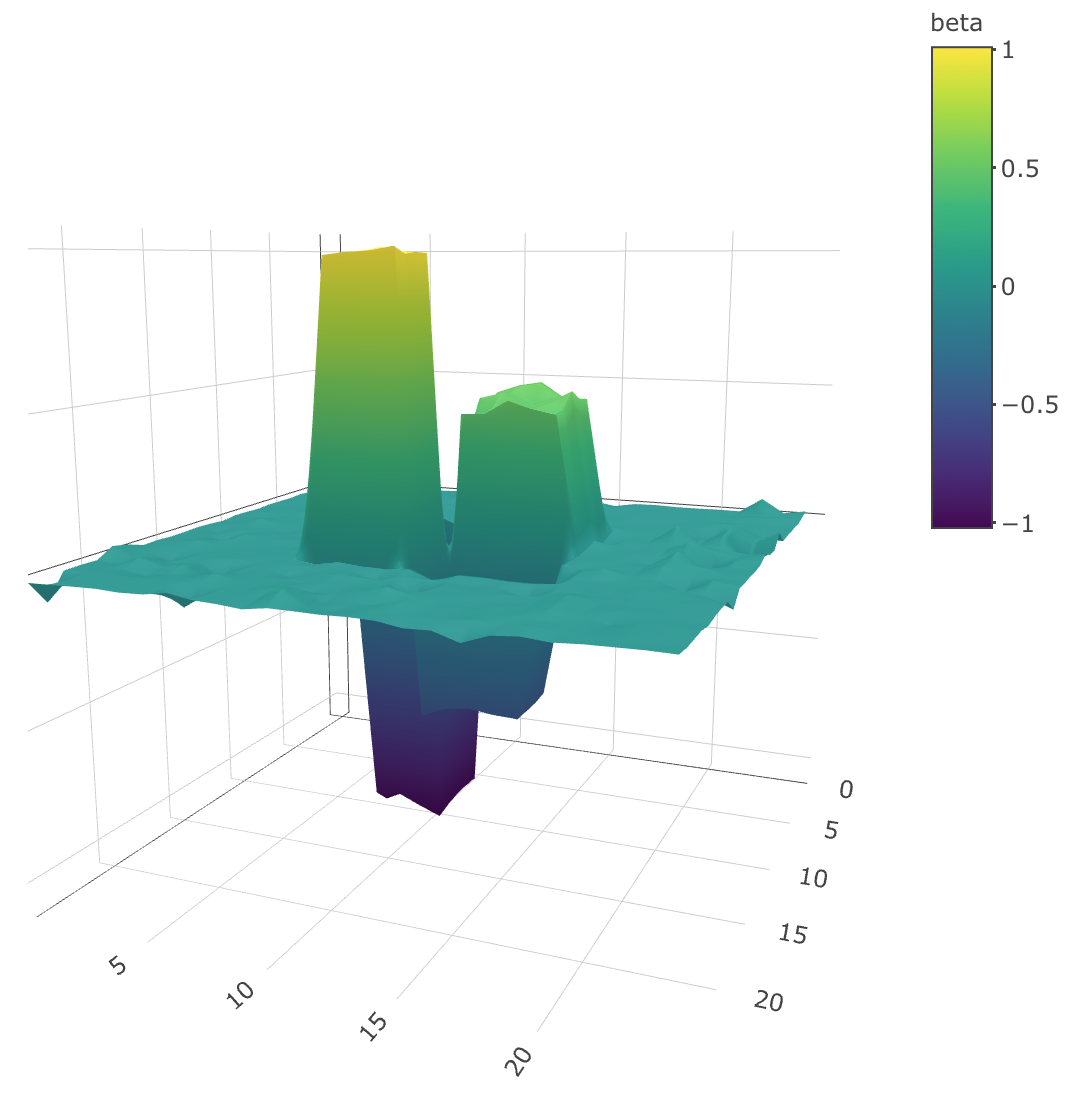}  
  \caption{Our approach}
\end{subfigure}
\begin{subfigure}{.5\textwidth}
 \centering
  \includegraphics[width=.8\linewidth]{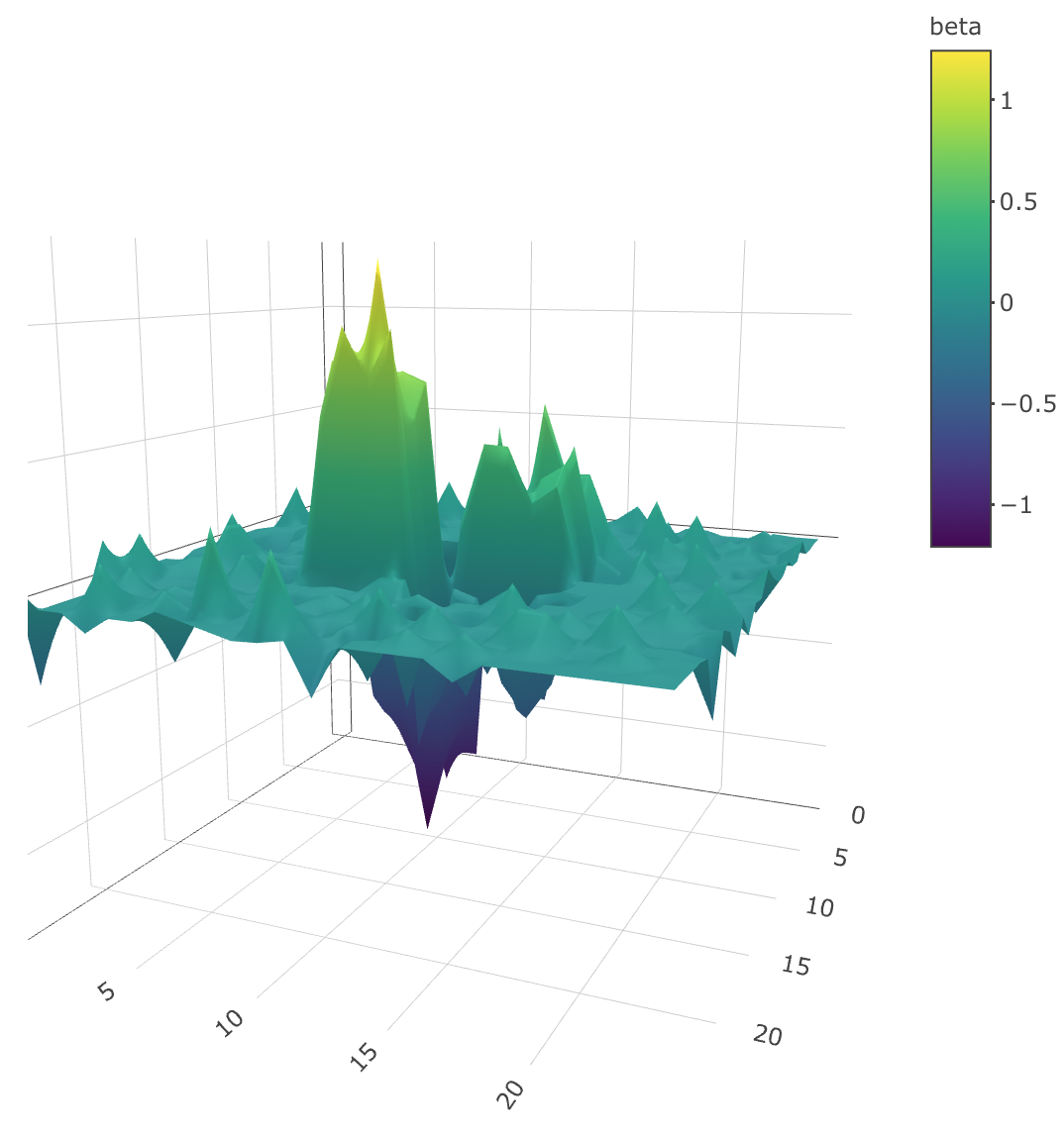}  
  \caption{Lasso}
\end{subfigure}
\begin{subfigure}{.5\textwidth}
  \centering
  \includegraphics[width=.8\linewidth]{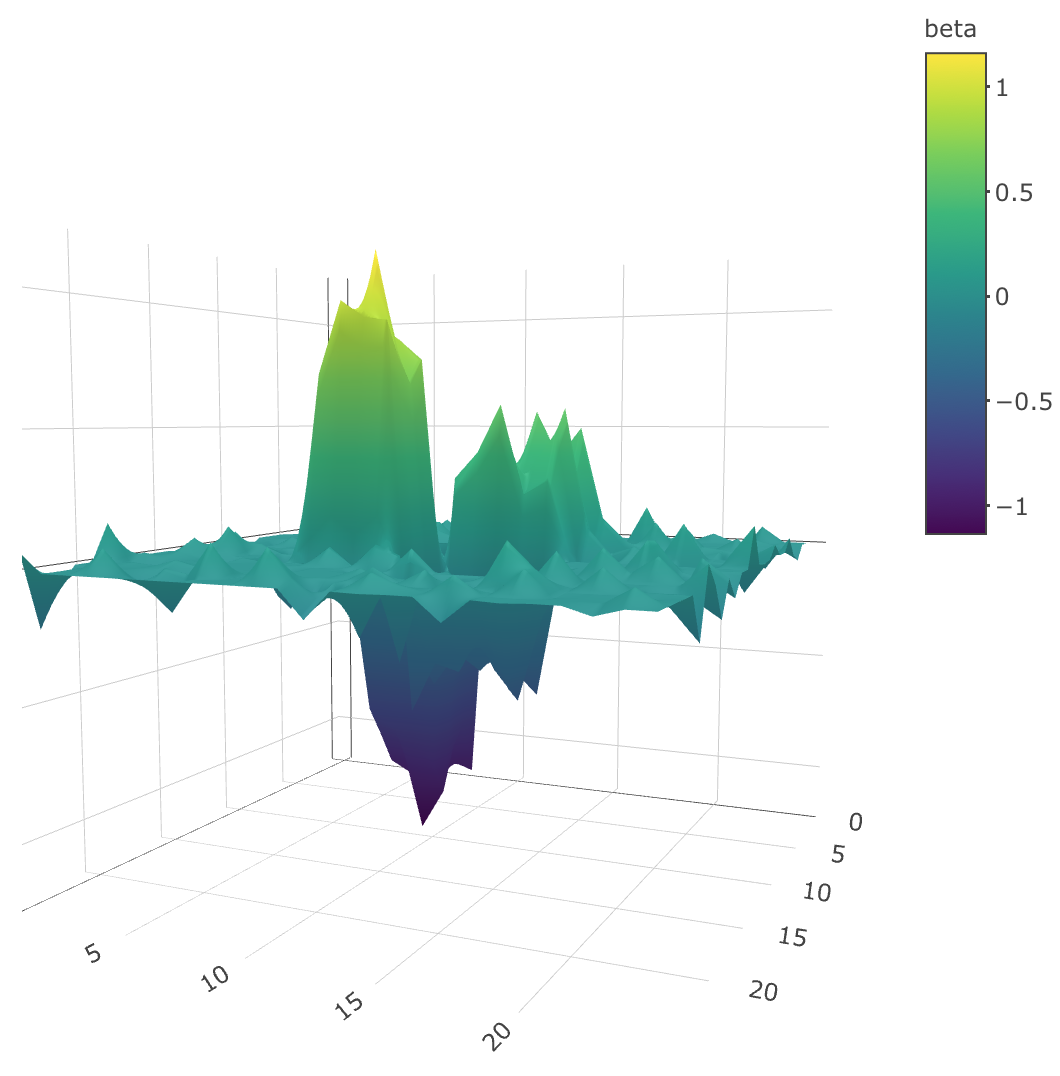}  
  \caption{Graph-Smooth-Lasso}
\end{subfigure}
\begin{subfigure}{.5\textwidth}
 \centering
  \includegraphics[width=.8\linewidth]{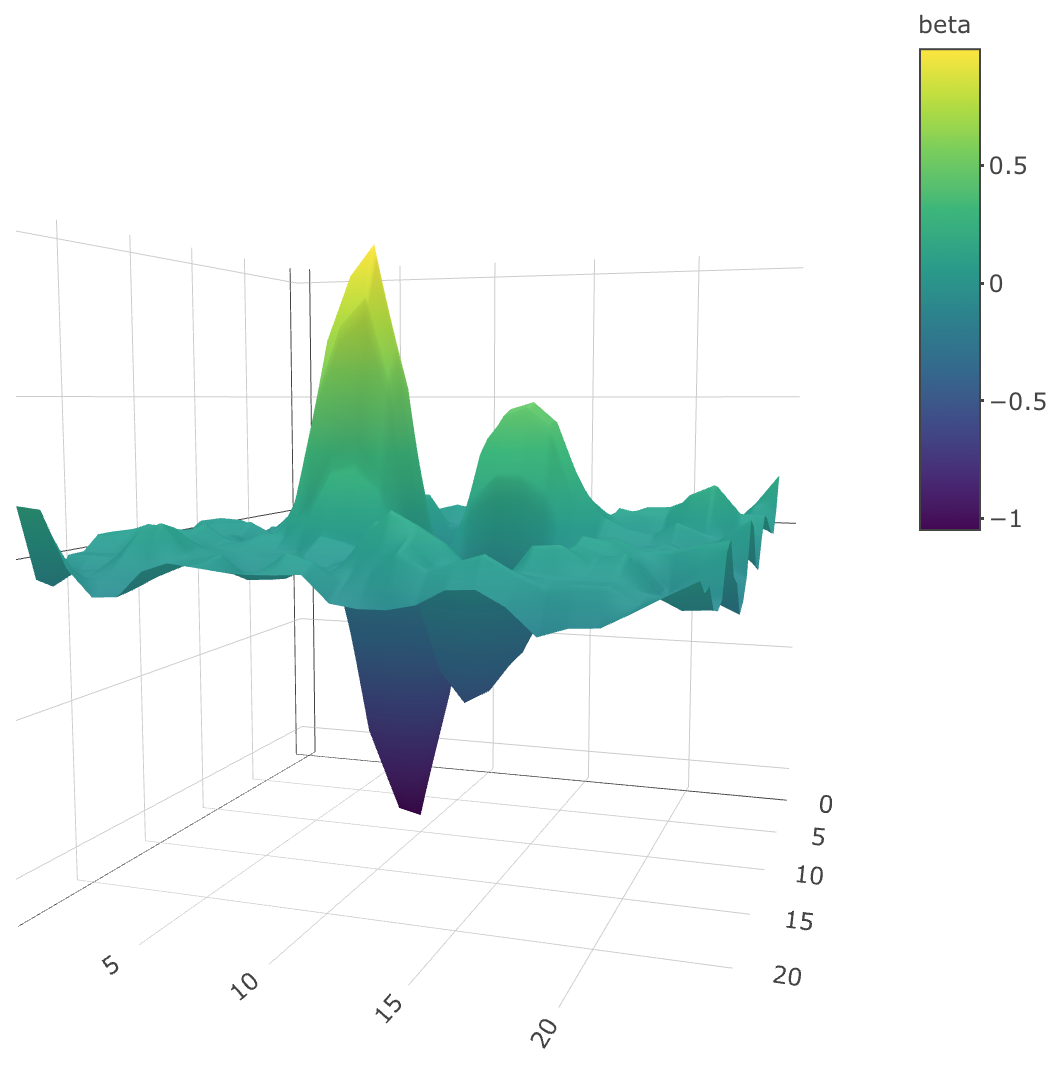}  
  \caption{Graph-Spline-Lasso}
\end{subfigure}
\caption{Estimation of $\beta^*$ plotted in (a) of Figure~\ref{piecewisepoly2dgraph} of the main paper.}
\label{2dconstantestimation}
\end{figure}

\begin{figure}[h]
\begin{subfigure}{.5\textwidth}
  \centering
  \includegraphics[width=.8\linewidth]{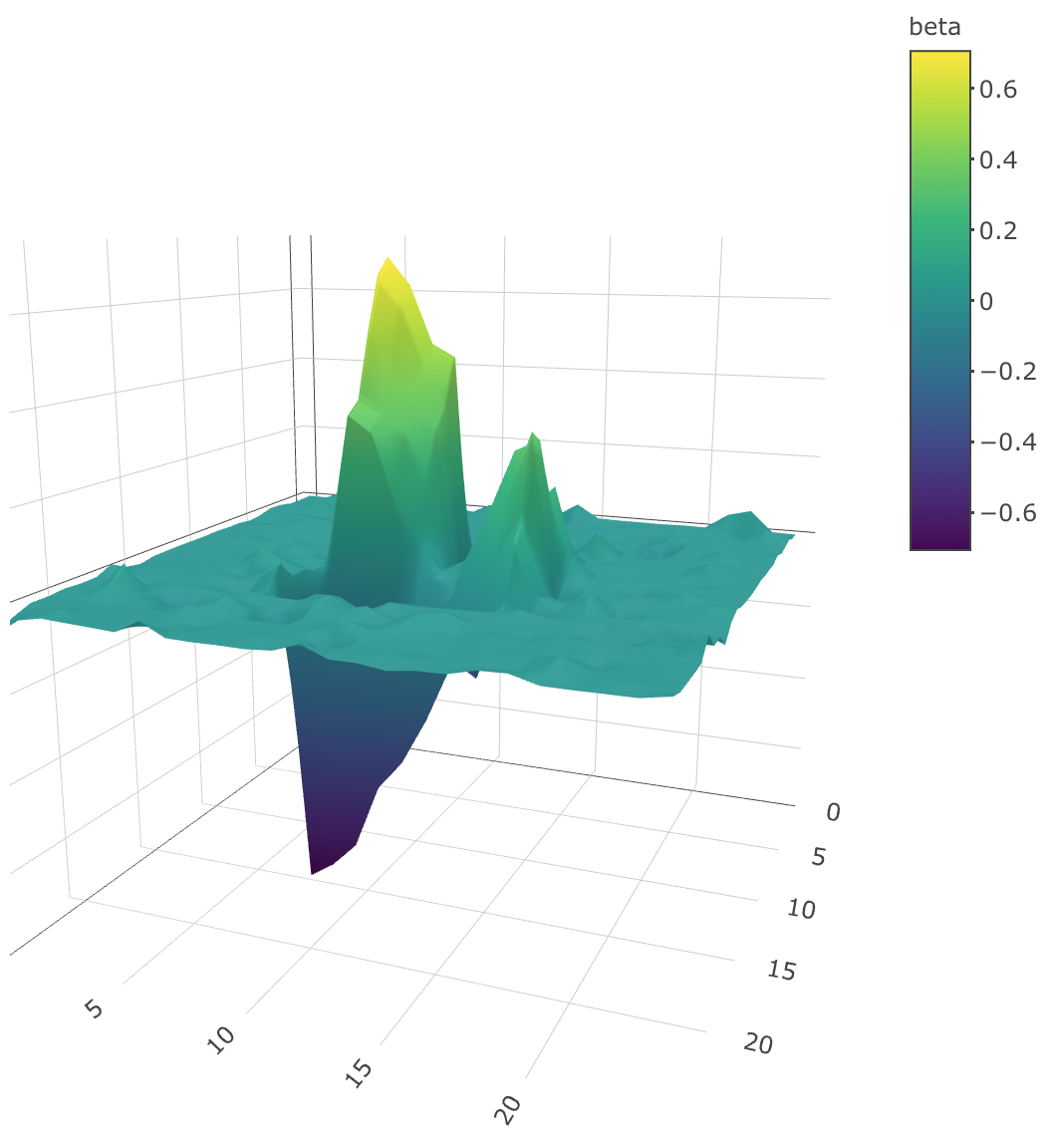}  
  \caption{Our approach}
\end{subfigure}
\begin{subfigure}{.5\textwidth}
 \centering
  \includegraphics[width=.8\linewidth]{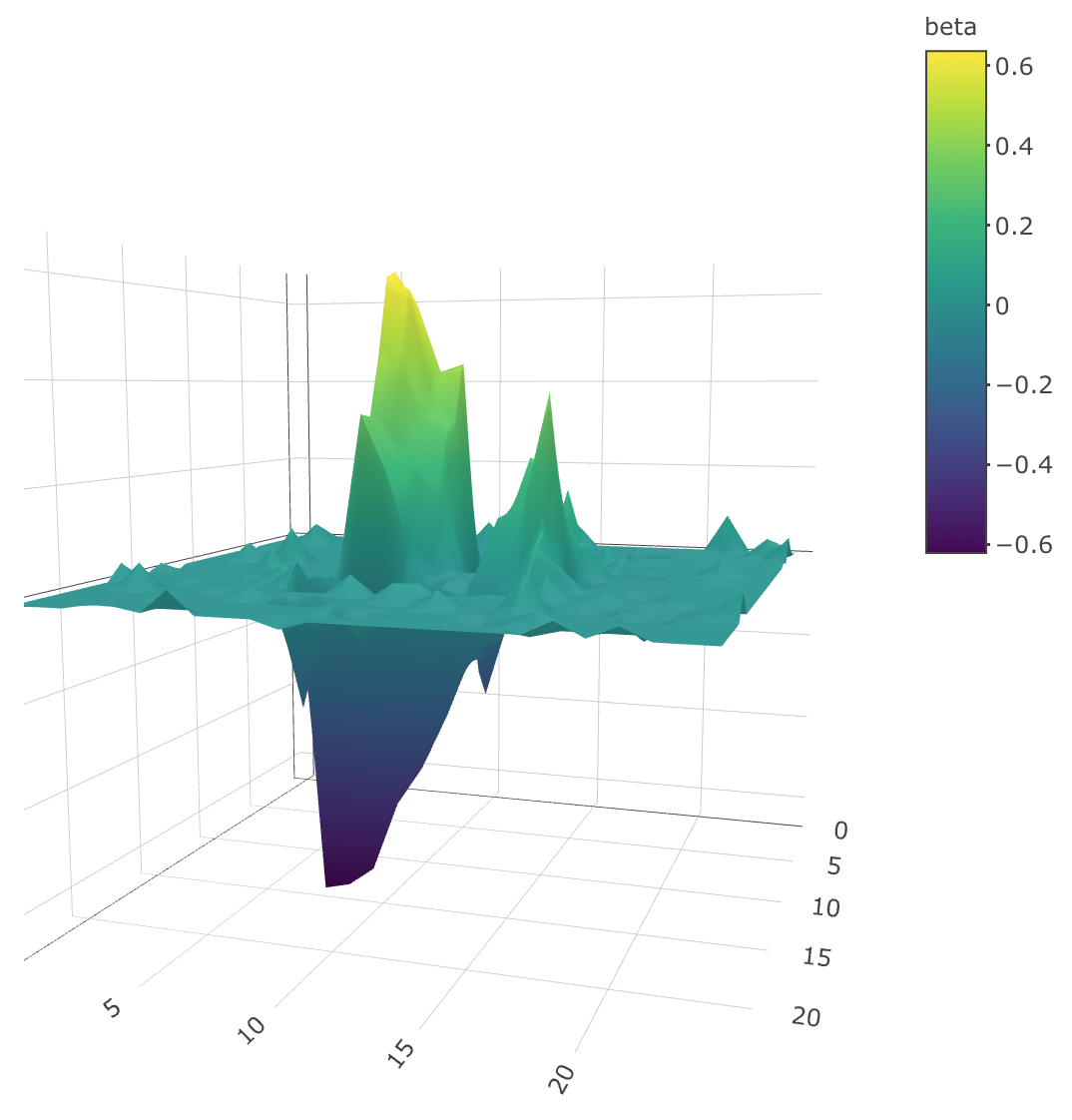}  
  \caption{Lasso}
\end{subfigure}
\begin{subfigure}{.5\textwidth}
  \centering
  \includegraphics[width=.8\linewidth]{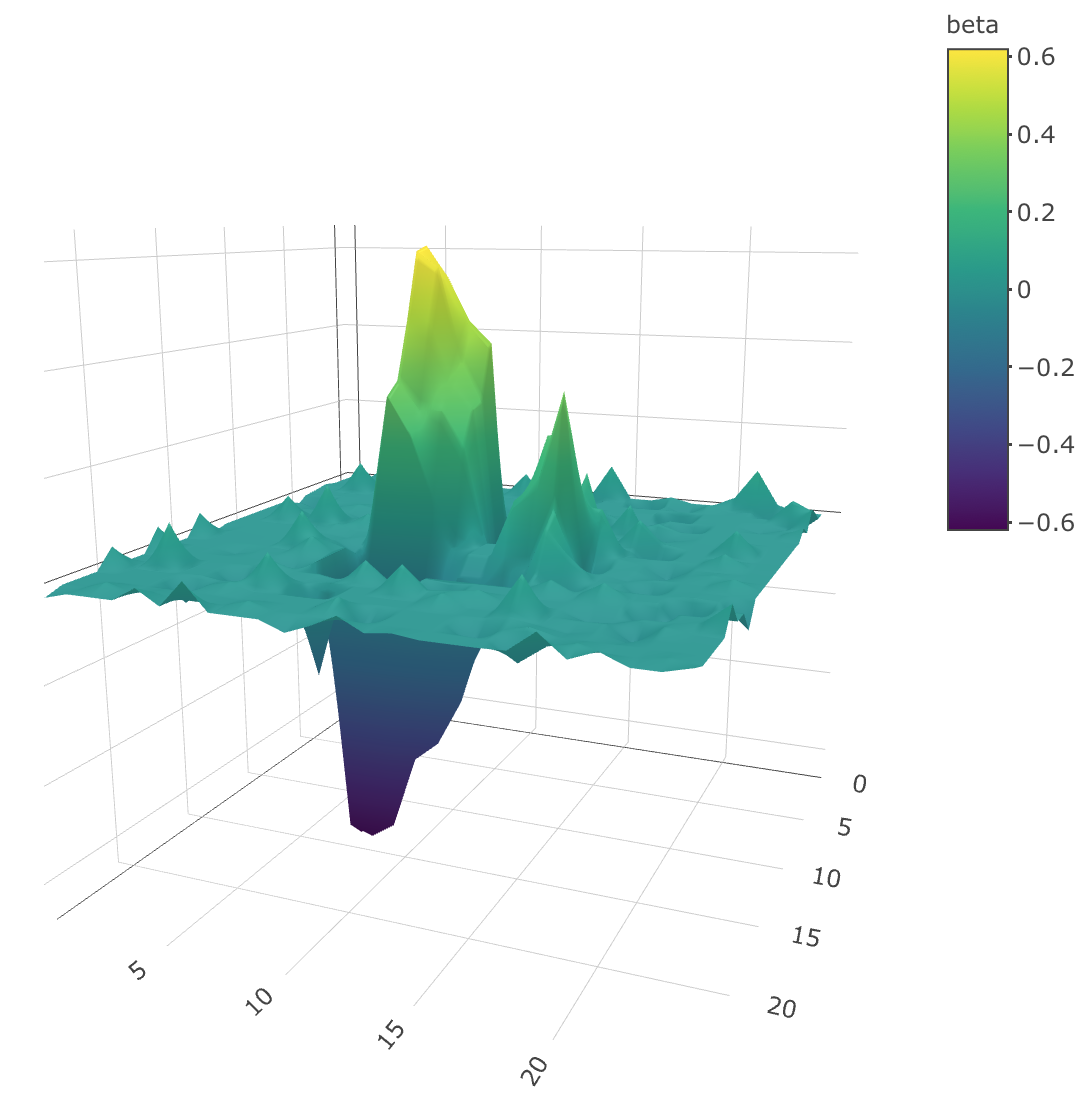}  
  \caption{Graph-Smooth-Lasso}
\end{subfigure}
\begin{subfigure}{.5\textwidth}
 \centering
  \includegraphics[width=.8\linewidth]{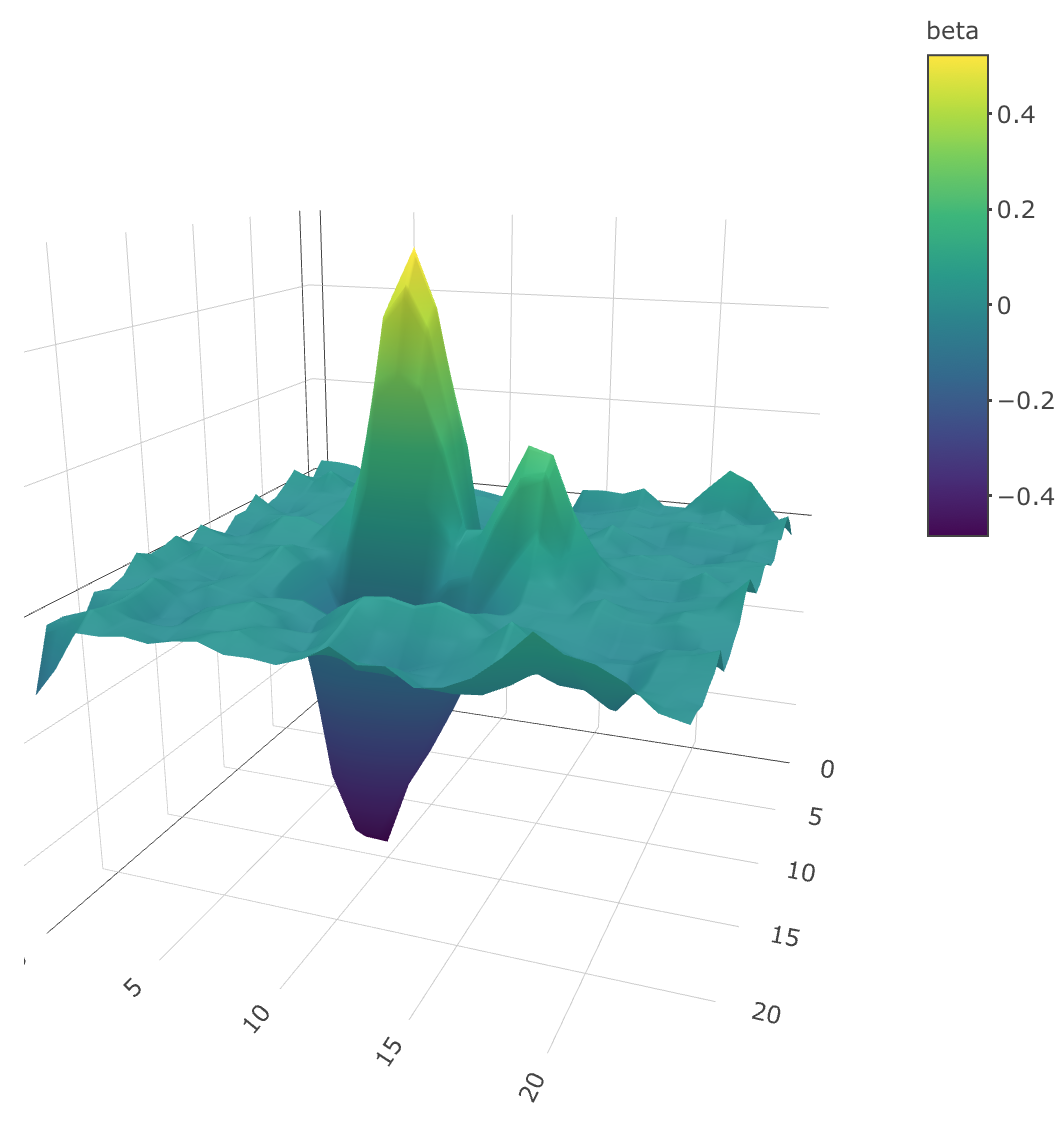}  
  \caption{Graph-Spline-Lasso}
\end{subfigure}
\caption{Estimation of $\beta^*$ plotted in (b) of Figure~\ref{piecewisepoly2dgraph} of the main paper. }
\label{2dlinearestimation}
\end{figure}

\begin{figure}[h]
\begin{subfigure}{.5\textwidth}
  \centering
  \includegraphics[width=.8\linewidth]{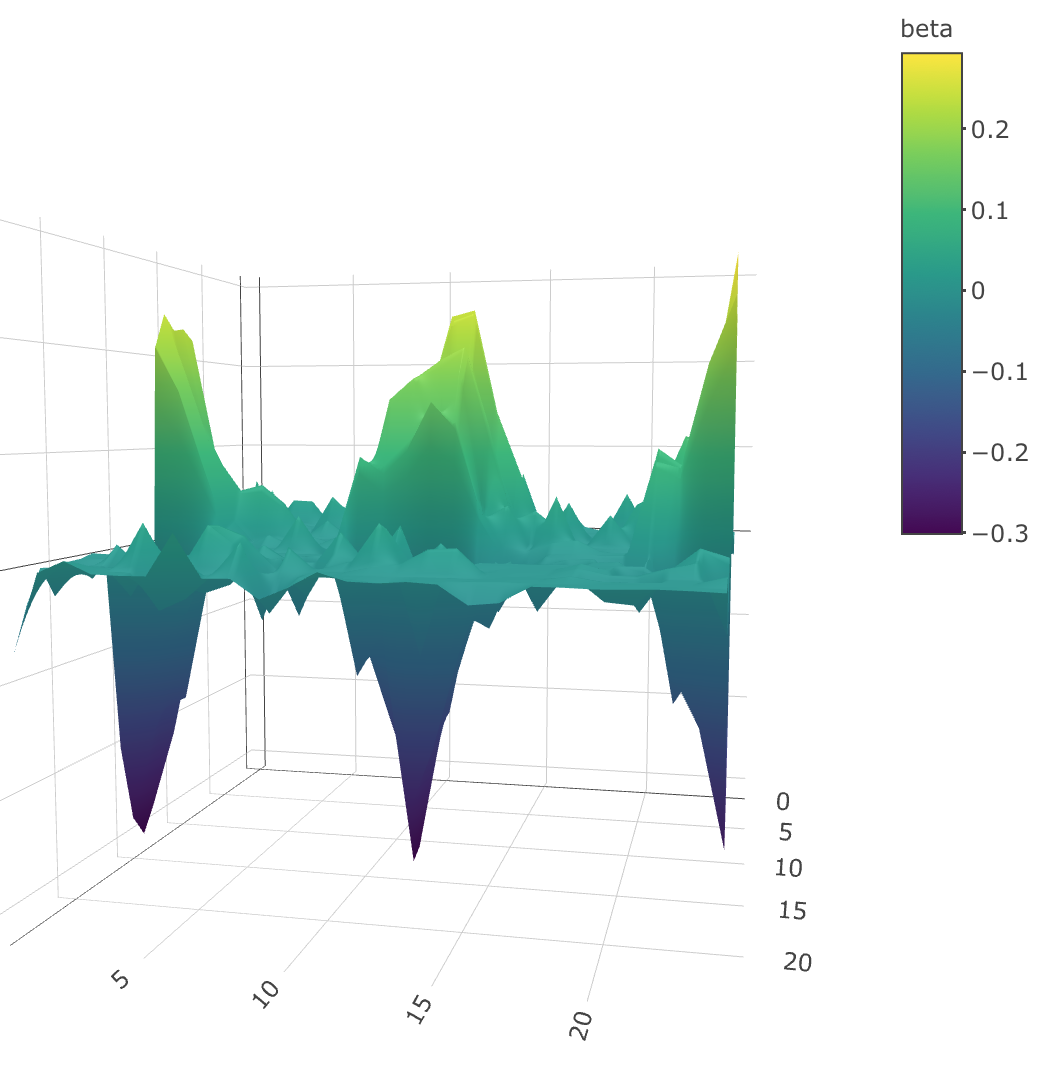}  
  \caption{Our approach}
\end{subfigure}
\begin{subfigure}{.5\textwidth}
 \centering
  \includegraphics[width=.8\linewidth]{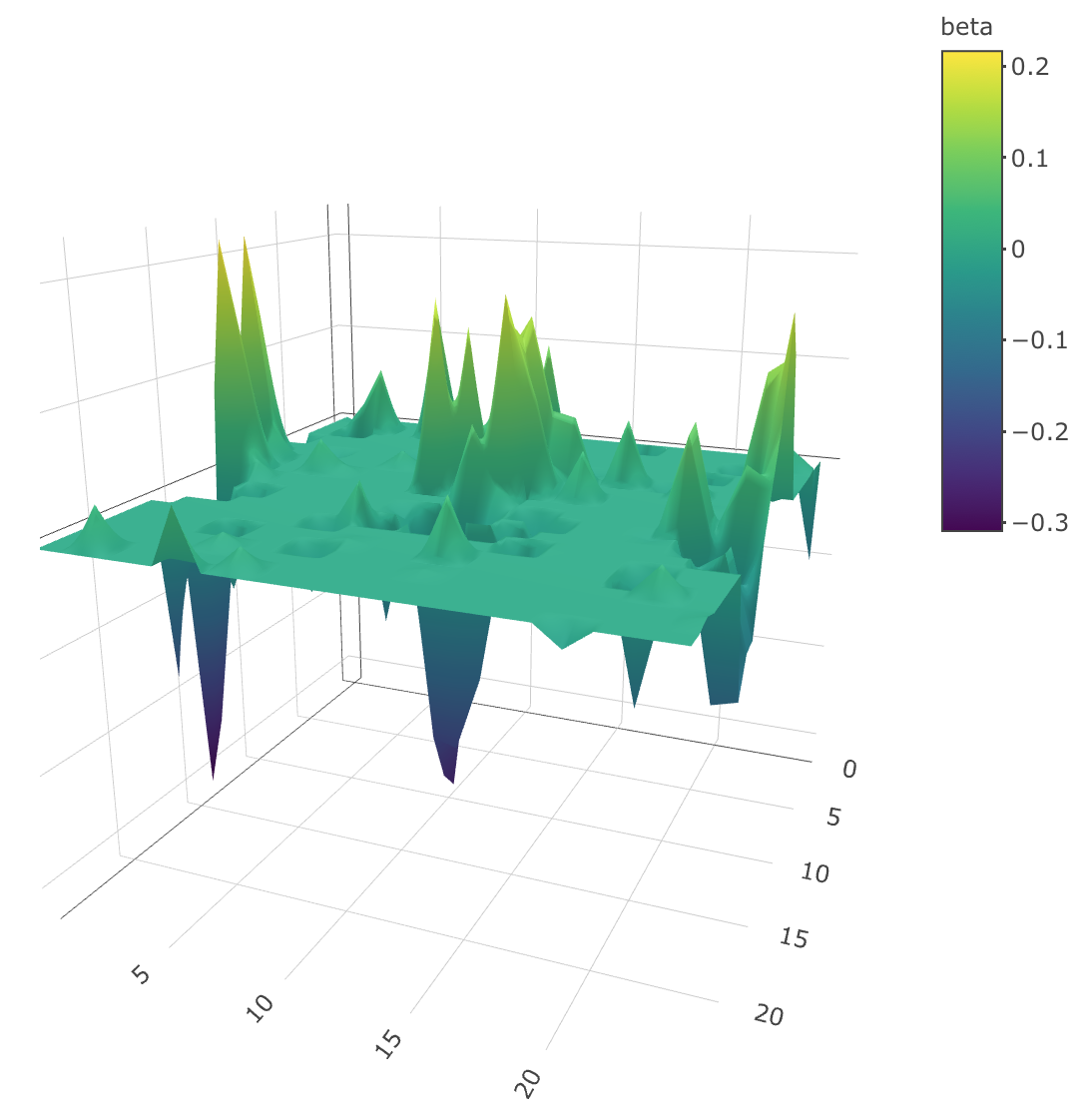}  
  \caption{Lasso}
\end{subfigure}
\begin{subfigure}{.5\textwidth}
  \centering
  \includegraphics[width=.8\linewidth]{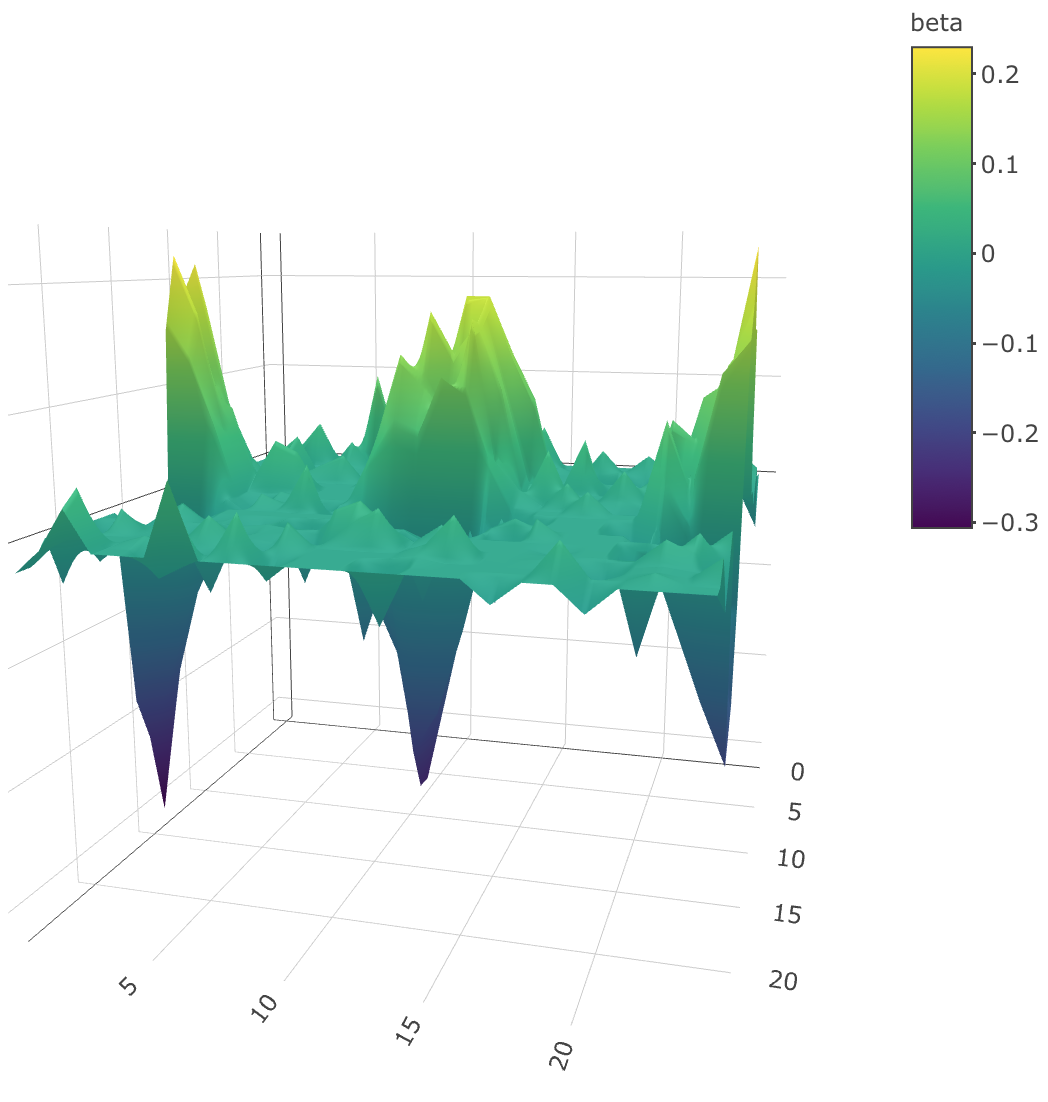}  
  \caption{Graph-Smooth-Lasso}
\end{subfigure}
\begin{subfigure}{.5\textwidth}
 \centering
  \includegraphics[width=.8\linewidth]{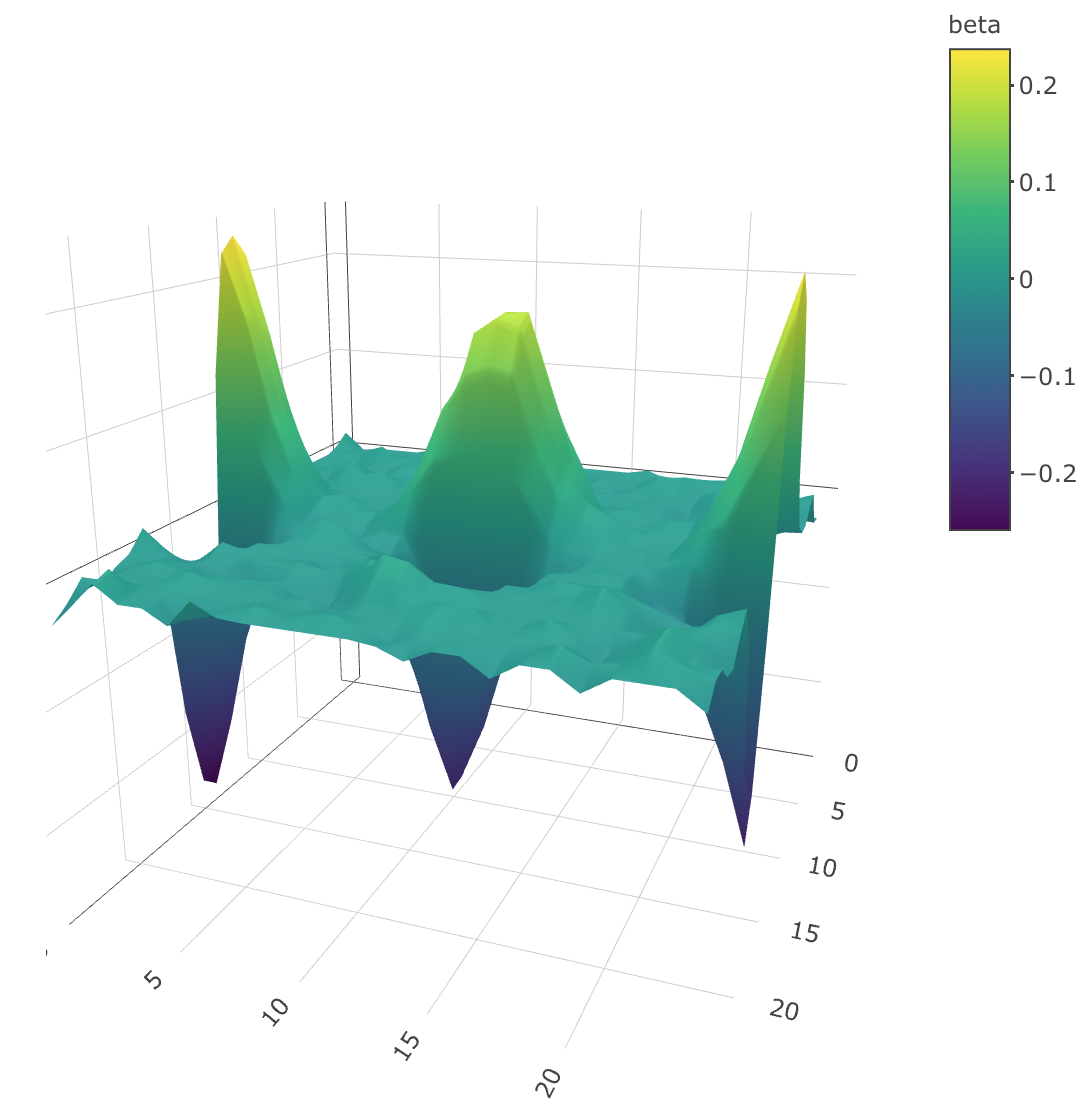}  
  \caption{Graph-Spline-Lasso}
\end{subfigure}
\caption{Estimation of $\beta^*$ plotted in (c) of Figure~\ref{piecewisepoly2dgraph} of the main paper. }
\label{2dquadraticestimation}
\end{figure}


\subsection{Supplementary results in Section~\ref{realdatasection}}
\label{realdataappendix}

Figure~\ref{realdatacoef} in this section illustrates the estimated regression coefficients of candidate genes. Table~\ref{realdatapathwaytable} provides details of genes selected within 58 pathways.   

\begin{figure}[h]
\centering 
\begin{subfigure}{.5\textwidth}
  \centering
  \includegraphics[width=.8\linewidth]{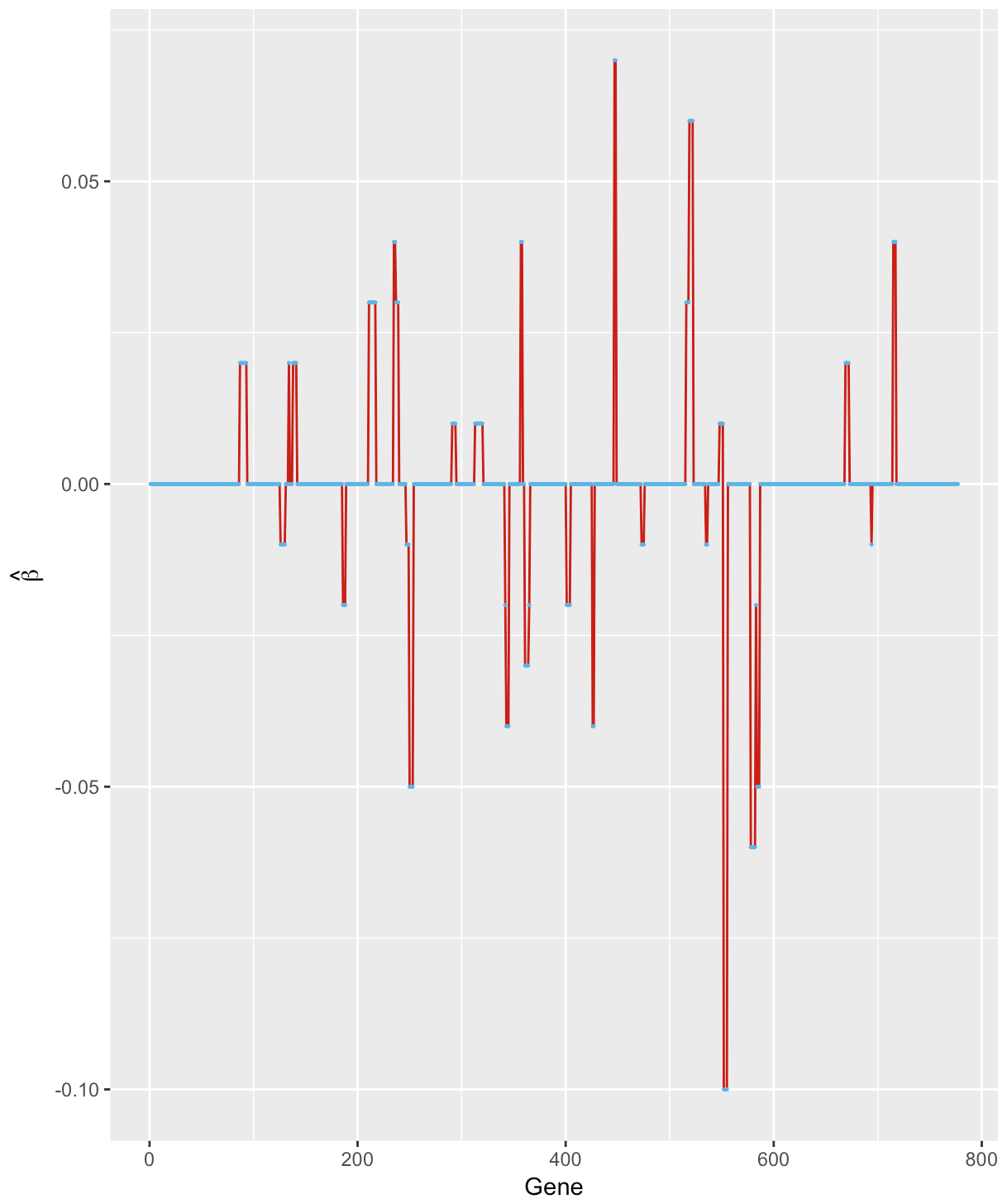}  
  \caption{}
\end{subfigure}

\begin{subfigure}{.5\textwidth}
 \centering
  \includegraphics[width=.8\linewidth]{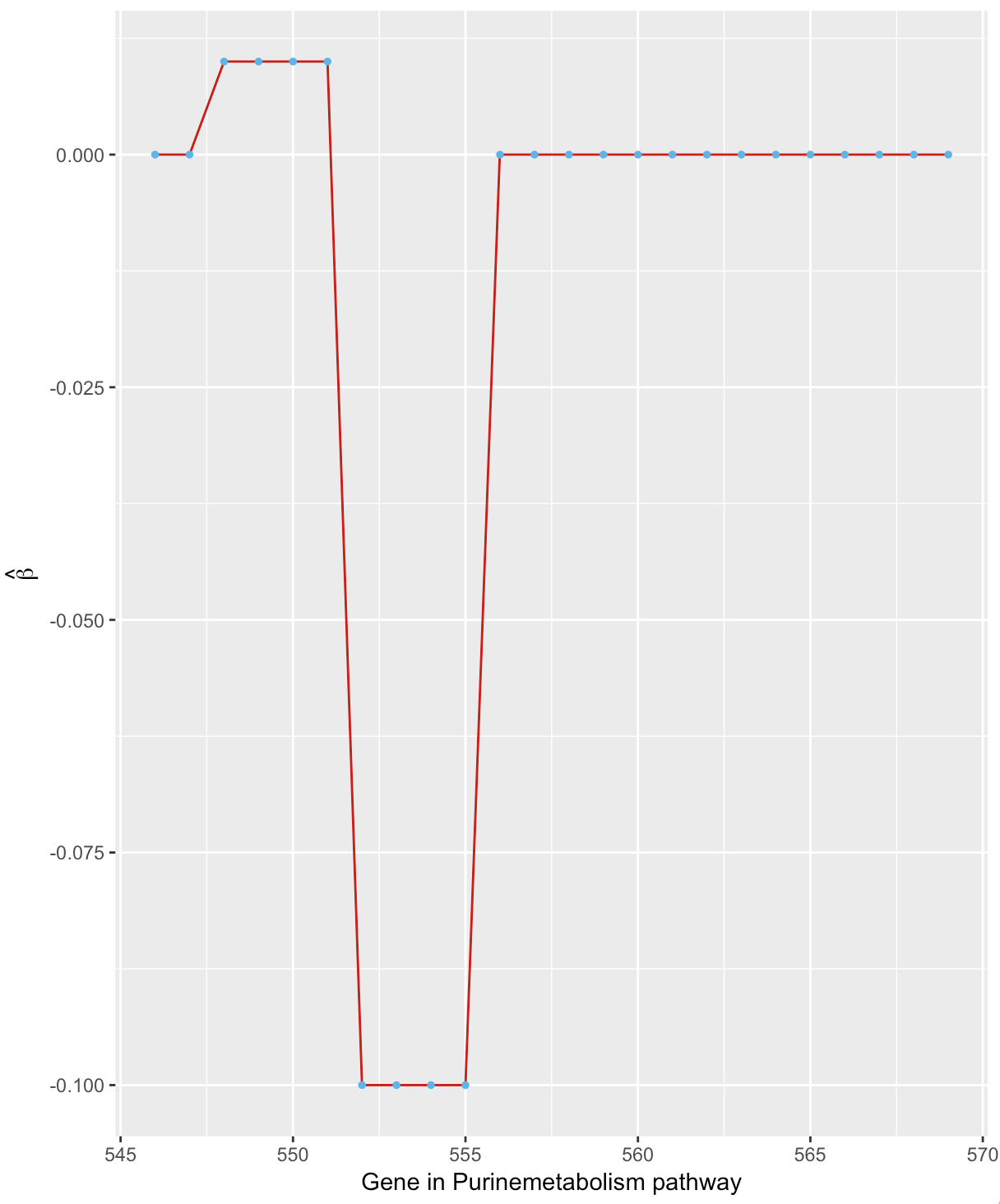}  
  \caption{}
\end{subfigure}
\caption{(a) Estimated regression coefficients of 777 candidate genes. (b) Estimated regression coefficients of genes in the Purinemetabolism pathway. Numbers in the x-axis are the gene codes. }
\label{realdatacoef}
\end{figure}

\begin{center}
\begin{longtable}{lllll}
\caption{Analysis of genes selected within each pathway.} 
\label{realdatapathwaytable} \\
\hline\hline \multicolumn{1}{c}{} & \multicolumn{1}{c}{Pathway} & \multicolumn{1}{c}{\makecell{Number of \\ genes}} & \multicolumn{1}{c}{\makecell{Number of \\ selected genes}}  & \multicolumn{1}{c}{\makecell{Percentage of \\ selected genes}} \\ \hline 
\endfirsthead
\multicolumn{5}{c}
{ \tablename\ \thetable{}: (continued)} \\
\hline \hline \multicolumn{1}{c}{} &\multicolumn{1}{c}{Pathway} & \multicolumn{1}{c}{\makecell{Number of \\ genes}} & \multicolumn{1}{c}{\makecell{Number of \\ selected genes}}  & \multicolumn{1}{c}{\makecell{Percentage of \\ selected genes}} \\ \hline 
\endhead
\hline\hline
\endfoot
\hline \hline
\endlastfoot
1 &Abscisicacidbiosynthesis  & 9  & 0 & 0      \\            
2& Arginine          & 2   & 0          &  0              \\
3&ArylpyronesStyrylpyronesStilbenesmetabolism & 3  &  0  &  0  \\
4&Asparaginemetabolism           & 4      &   0    &  0    \\
5&Auxinbiosynthesis              & 7  &   0  &  0     \\
6&Berberinemetabolism            & 12   &   0 &  0    \\    
7&Biotinmetabolism               &  3   &  0 &   0 \\
8&Brassinosteroidbiosynthesis    & 3  &  0 &  0    \\     
9&Calvincycle            & 31 &  0   &   0    \\
10&Carotenoidbiosynthesis      & 11       &  0       &  0  \\
11&Chorismatemetabolism      &     10     &  7 & $70\%$    \\
12&Citratecycle(TCAcycle)     & 36     &   5    &  $13.9\%$        \\
13&Co-enzymemetabolism    & 7    &  2       &   $28.6\%$           \\
14&Cytokininbiosynthesis      & 8       &  3   &    $37.5\%$     \\
15&Ethylenebiosynthesis         &  11      &  0   &  0     \\
16&Fattyacidbiosynthesis         & 34      &  3    &   $8.8\%$  \\  
17&Fattyacidoxidation         &  12      &  0   &  0   \\
18&Flavonoidmetabolism     &  15     &  7   &  $46.7\%$     \\      
19&Folatemetabolism        &  10    &  0   &  0     \\
20&Gibberellinbiosynthesis       &  19    &  6   &  $31.6\%$  \\       
21&GlutamateGlutaminemetabolism   & 17 & 6  &   $35.3\%$ \\        
22&Glutathionemetabolism        & 6     & 0   &  0     \\ 
23&Glycerolipidmetabolism        & 24 & 4    &  $16.7\%$   \\
24&GlycolysisGluconeogenesis     & 43   & 8  &   $18.6\%$ \\         
25&Glycoproteinbiosynthesis     & 17  &  4  &   $23.5\%$    \\
26&Histidinemetabolism          & 4    &  2   &  $50\%$    \\   
27&Inositolphosphatemetabolism  & 33  & 5  & $15.2\%$    \\
28&IsoleucineValineLeucinemetabolism  & 7 & 0   &  0  \\          
29&Jasmonicacidbiosynthesis        &  11   & 4   &  $36.4\%$   \\
30&Lysinemetabolism                 &  8      & 0  &  0  \\
31&Methioninemetabolism      &  4    & 0 &  0     \\
32&Mevalonatepathway     &       21  & 0  &   0              \\
33&Monoterpenemetabolism      &   4    & 0   &   0  \\        
34&Morphinemetabolism        & 2  & 2  &  $100\%$  \\ 
35&Non-Mevalonatepathway  & 17 & 0 &  0  \\
36&Onecarbonpool                  & 2  &  0  &  0 \\       
37&Pentosephosphatecycle     & 10 & 0 &  0 \\
38&PhenylalanineTyrosinemetabolism    & 8 & 1 & $12.5\%$ \\       
39&Phenylprpanoidmetabolism      & 17    & 1  &  $5.8\%$      \\
40&Phospholipiddegradation        &    9  & 1  &  $11.1\%$        \\ 
41&Phytosterolbiosynthesis      & 25   & 2   &   $8\%$          \\
42&Plastoquinonebiosynthesis       & 2    & 0    &  0   \\
43&Polyaminebiosynthesis            & 11     & 0     &  0    \\
44&PorphyrinChlorophyllmetabolism         & 24   & 8  & $33.3\%$ \\
45&Prolinemetabolism               & 3 &1        &  $33.3\%$         \\
46&Proteinprenylation                & 7     & 0    &  0        \\
47&Purinemetabolism        &  24      & 8     &  $33.3\%$       \\
48&Pyrimidinemetabolism      & 13      &  5  &  $38.5\%$         \\ 
49&Riboflavinmetabolism             &  22    & 4   &  $18.2\%$   \\
50&SerineGlycineCysteinemetabolism      & 17 & 0  &  0\\     
51&Sesquiterpenemetabolism            &  5 & 0  &  0  \\
52&Sphingophospholipidmetabolism     & 2 & 0 &  0\\         
53&Starchandsucrosemetabolism      & 70 & 5  &  $7.1\%$ \\
54&SynthesisofUDP-sugars             & 6    & 0  &  0  \\ 
55&Threoninemetabolism     &  10 & 0  &  0\\
56&Tocopherolbiosynthesis     & 2     & 2      &   $100\%$      \\
57&Tryptophanmetabolism   & 19 & 1  &  $5.3\%$\\
58&Ubiquinonebiosynthesis       & 4     & 0    &  0    
\end{longtable}
\end{center}

\end{document}